\documentclass[leqno,11pt]{amsart}
\usepackage{amssymb, amsmath,latexsym,amsfonts,amsbsy, amsthm,mathtools,graphicx,color}
\usepackage{float}
\usepackage{hyperref}
\usepackage{cancel}
\numberwithin{equation}{section}
\allowdisplaybreaks

\usepackage{enumerate, cases, bm, extarrows, bbm}

\newtheorem{thm}{Theorem}[section]
\newtheorem{lem}[thm]{Lemma}
\newtheorem{prop}[thm]{Proposition}
\newtheorem{cor}[thm]{Corollary}

\theoremstyle{definition}
\newtheorem{defn}{Definition}[section]

\theoremstyle{remark}
\newtheorem{rmk}[thm]{Remark}

\let\al=\alpha
\let\ga=\gamma

\let\e=\varepsilon

\let\f=\frac

\let\Om=\Omega
\let\wt=\widetilde
\let\wh=\widehat
\let\na=\nabla
\let\pa=\partial

\newcommand{\RR}{\mathbb{R}}
\newcommand{\ZZ}{\mathbb{Z}}
\newcommand{\NN}{\mathbb{N}}

\newcommand{\CC}{\mathbb{C}}
\newcommand{\TT}{\mathbb{T}}
\newcommand{\p}{\partial}
\newcommand{\aaa}{\mathbf{a}}
\newcommand{\bbb}{\mathbf{b}}
\newcommand{\vvv}{\mathbf{v}}
\newcommand{\www}{\bm{\omega}}

\newcommand{\FF}{\mathcal{F}}
\newcommand{\GG}{\mathcal{G}}
\newcommand{\WWW}{\mathcal{W}}
\newcommand{\NNN}{\mathcal{N}}
\newcommand{\FFF}{\mathcal{F}}
\newcommand{\LLL}{\mathcal{L}}
\newcommand{\ii}{\mathrm{i}}
\newcommand{\vp}{\varphi}

\newcommand{\beq}{\begin{equation}}
	\newcommand{\eeq}{\end{equation}}
\newcommand{\ben}{\begin{eqnarray}}
	\newcommand{\een}{\end{eqnarray}}
\newcommand{\beno}{\begin{eqnarray*}}
	\newcommand{\eeno}{\end{eqnarray*}}

\begin{document}
	
	\title[Self-Similar Spiral Solution of 2-D Euler Equations]{Self-similar algebraic spiral solution of 2-D incompressible Euler equations}
	
	\author[F. Shao]{Feng Shao}
	\address{School of Mathematical Sciences, Peking University, Beijing 100871, China}
	\email{fshao@stu.pku.edu.cn}

	\author[D. Wei]{Dongyi Wei}
	\address{School of Mathematical Sciences, Peking University, Beijing 100871,  China}
	\email{jnwdyi@pku.edu.cn}
	
	\author[Z. Zhang]{Zhifei Zhang}
	\address{School of Mathematical Sciences, Peking University, Beijing 100871, China}
	\email{zfzhang@math.pku.edu.cn}
	
	\date{\today}

	\begin{abstract}
		In this paper, we prove the existence of self-similar algebraic spiral solutions of the 2-D incompressible Euler equations
		for the initial vorticity of the form $|y|^{-\frac1\mu}\ \mathring{\omega}(\theta)$ with $\mu>\f12$ and $\mathring{\omega}\in L^1(\TT)$, satisfying $m$-fold symmetry ($m\ge 2$) and a dominant condition. As an important application, we  prove the existence of weak solution
		when $\mathring{\omega}$ is a Radon measure on $\TT$ with $m$-fold symmetry, which is related to the vortex sheet solution.
	\end{abstract}
	\maketitle
	

	\section{Introduction}
	
	In this paper, we study the 2-D incompressible Euler equations in $\RR^2\times[0,+\infty)$:
	\begin{equation}\label{2DEuler}
		\begin{cases}
			\vvv_t+\vvv\cdot \nabla_y \vvv+\nabla P=0,\\
			\text{ div }\vvv=0,
		\end{cases}
	\end{equation}
	where $\vvv$ is the velocity and $P$ is the pressure. The vorticity-stream formulation of \eqref{2DEuler} takes as follows
	\begin{equation}\label{2dEuler}
		\begin{cases}
			\www_t+\vvv\cdot\nabla_y\www=0,\\
			\vvv=\nabla^\perp_y\Psi,\quad \Delta_y\Psi=\www,
		\end{cases}
	\end{equation}
	where $\www$ is the vorticity, $\Psi$ is the stream function and $\nabla^\perp_y=(-\p_{y_2}, \p_{y_1})$.
	
	For smooth initial data, the 2-D Euler equations are globally well-posed due to  the conservation of the $L^\infty$ norm of the vorticity, $\|\www(t)\|_{L^\infty}\le \|\www_0\|_{L^\infty}$. For classical results regarding the well-posedness and stability of steady solutions, we refer to \cite{Che,MB,MP}. Additionally, we highlight the recent breakthrough by \cite{BL} on the ill-posedness in borderline spaces. However, the long-time behavior of the solutions remains a long-standing problem. For relevant results, we refer to \cite{Sve, Bou, KS, BM, IJ, WZZ} and the references therein.
		
	For non-smooth initial data, a classical result due to Yudovich \cite{Yu} establishes the global existence and uniqueness of weak solutions when the initial vorticity lies in $L^\infty(\RR^2)\cap L^1(\RR^2)$, which is related to the vortex patch solution. In fact, the global existence of weak solutions also holds for the initial vorticity in $L^p(\RR^2)\cap L^1(\RR^2), 1\le p<+\infty$. However, in the latter case, the uniqueness of weak solutions remains an open question; see \cite{Bre, Vis} for recent progress. Another classical result due to Delort \cite{De} is the global existence of weak solutions when the initial vorticity is a Radon measure with a distinguished sign. The qualitative behavior of Delort's solutions remains an important open question. For recent important progress on singular vortex patch solutions, we refer to \cite{Elgindi, EJ}.	\smallskip
	
	In a series of remarkable works \cite{ELL2013, ELL2016, ELL2016-2}, Elling constructed a class of self-similar algebraic spiral solutions for the 2-D Euler equations. These solutions are significant in applications due to their prevalence in various physical phenomena \cite{van}.
	Elling considered a class of locally integrable self-similar initial data of the form
	\begin{equation}\label{initial_data}
		\www(y,t)\xrightarrow{t\to0+}|y|^{-\frac1\mu}\ \mathring{\omega}(\theta),\qquad \theta\in\TT:=\RR/(2\pi \ZZ).
	\end{equation}
	
	The main result in \cite{ELL2016} is stated as follows.
	
	\begin{thm}\label{thm:Ell}
		Given $\e>0$ and $\mu>\f23$, there exists an $N_0\in \NN$ so that a weak solution of \eqref{2DEuler} and \eqref{initial_data} exists for all initial data
		$\mathring{\omega}$ satisfying the following conditions:
		\begin{itemize}
			\item[1.] Periodicity: $\mathring{\omega}$ is $\f {2\pi} N$-periodic for $N\ge N_0$;
			
			\item[2.] Dominant rotation: the Fourier coefficients satisfy
			\beno
			|\wh{\mathring{\omega}}(0)|\ge \e\sum_{n\neq 0}|\wh{\mathring{\omega}}(n)|.
			\eeno
		\end{itemize}
	\end{thm}

	The goal of this paper is to extend Elling's remarkable result in three important aspects:
	
	\begin{itemize}
		
		\item[1.] In Theorem \ref{thm:Ell}, $\mathring{\omega}$ lies in a {Wiener} algebra. We would like to allow $\mathring{\omega}\in L^1(\TT)$ and even $\mathring{\omega}\in \mathcal{M}(\TT)$, a Radon measure. This kind of data is crucial for the of construction self-similar vortex sheet solutions. It may be relevant to the non-uniqueness problem of weak solutions of the 2-D Euler equations with initial vorticity in $L^p(\RR^2), p<+\infty$.
		
		\item[2.] In Theorem \ref{thm:Ell}, $N_0$ is a large positive integer. We would like to improve $N_0$.
		
		\item[3.] Extend the range of $\mu$ to $\mu>\f12$, which is a natural condition ensuring $\www_0\in L^1_{\text{loc}}(\RR^2)$.
		
	\end{itemize}

	To state our results, let us first introduce the definition of weak solution.
	
	\begin{defn}\label{weak_sol_def}
		A vector field $\vvv(y,t)$ is called a \textit{weak solution} of the 2-D Euler equations \eqref{2DEuler}  provided that
		\begin{enumerate}[(i)]
			\item $\vvv\in C([0,\infty); L^2_{\text{loc}}(\RR^2; \RR^2))$, i.e., for any $R>0$ we have $\vvv\in C([0,\infty); L^2(B_R; \RR^2))$, where $B_R=\{y\in\RR^2: |y|\leq R\}$;
			\item $\text{ div}_y \vvv=0$ in the sense of distributions, i.e.,
			\[\int_{\RR^2}\nabla_y\eta(y,t)\cdot \vvv(y,t)\,dy=0,\qquad \forall t\geq 0, \ \forall \eta\in C_c^\infty\left(\RR^2\times[0,\infty)\right);\]
			\item $\vvv$ solves the 2-D Euler equations \eqref{2DEuler} in the sense of distributions, i.e., for any vector-field $\mathbf w\in C_c^\infty\left(\RR^2\times[0,\infty); \RR^2\right)$ with $\text{ div}_y \mathbf w=0$ there holds
			\[\int_{\RR^2}\vvv\cdot\mathbf w \,dy\Big|_{t=0}+\int_0^\infty\int_{\RR^2}\vvv\cdot \p_t\mathbf w+\left(\vvv\otimes\vvv\right):\nabla_y\mathbf w \,dy\,dt=0,\]
			where $\vvv\otimes\vvv=\left(\vvv^i\vvv^j\right)$, $\nabla_y\mathbf w=\left(\frac{\p \mathbf w^j}{\p y_i}\right)$ and $A : B =\sum_{i,j=1}^2A_{ij}B_{ij}$.
		\end{enumerate}
	\end{defn}
	
	For $\Omega\in L^1(\TT)$, we denote
	\[P_0\Omega:=\frac1{2\pi}\int_\TT\Omega(\theta)\,d\theta,\qquad P_{\neq}\Omega(\theta):=\Omega(\theta)-P_0\Omega\in L^1(\TT).\]
	
	Now our main result is stated as follows.
	
	\begin{thm}\label{mainthm}
		Let $\mu>\frac12$ and $m\in \NN_+\setminus\{1\}$.  There exists $\varepsilon>0$, which is independent of $m\geq2$, such that for any $\frac{2\pi}m-$periodic $\mathring{\omega}\in L^1(\TT)$ with the dominant condition
		\begin{equation}\label{1.5}
			\left\|P_{\neq}\mathring{\omega}\right\|_{L^1(\TT)}\leq \varepsilon m^{\frac12}|P_0\mathring{\omega}|,
		\end{equation}
		a weak solution $\vvv$ of \eqref{2DEuler} with the initial data \eqref{initial_data} exists. Moreover, the vorticity $\www\in C([0,\infty); L^1_{\text{loc}}(\RR^2))$.
	\end{thm}
	
	\begin{rmk}
The velocity $\vvv$ is strongly continuous at $t=0$. Given the initial vorticity  $\www_0(y)=|y|^{-\frac1\mu}\mathring{\omega}(\theta)$, the initial velocity can be recovered by $\vvv_0=\nabla_y^\perp\Psi_0$, with $\Psi_0(y)=|y|^{2-\frac1\mu}B(\theta)$, where $B$ is the only $\frac{2\pi}m$-periodic function satisfying $\gamma^2B+B''=\mathring{\omega}$, $\gamma=2-\frac1\mu$. See Proposition \ref{v_initial} for the details.
	
	\end{rmk}
	
	Let us give more remarks about our result.
	
	\begin{enumerate}[1.]
		
		\item It seems possible to extend our result to the case when $m\ge 1$. This will be conducted in a future work.
		
		\item We prove Theorem \ref{mainthm} using the implicit function theorem (IFT); thus, the self-similar solutions constructed in Theorem \ref{mainthm} are unique if the perturbation is sufficiently small, ensuring that the perturbed solution lies within the framework of the IFT. The uniqueness of these self-similar solutions in a broader class remains an open question.
		
		\item The stability of the self-similar type solution constructed in Theorem \ref{mainthm} is a very interesting question. This may be relevant to the non-uniqueness problem of weak solutions when the initial vorticity lies in $L^p(\RR^2), p<+\infty$, see \cite{Bre, JS, Vis, ABC} for relevant results.
		
		\item Another very interesting problem is to study the inviscid limit problem for the data considered in Theorem \ref{mainthm}. We conjecture that our solution can be obtained via the vanishing viscosity limit. Thus, it is a physical solution in the sense of \cite{Bar}.
		
		\item Let us mention a recent important result \cite{GG} about self-similar spirals for the generalized surface quasi-geostrophic equations:
		\begin{equation}\nonumber
			\begin{cases}
				\theta_t+\vvv\cdot \nabla \theta=0,\\
				\vvv=-\na^\perp(-\Delta)^{-1+\f \ga 2}\theta.
			\end{cases}
		\end{equation}
		The case $\ga=0$ corresponds to the 2-D Euler equations and $\ga=1$ to the surface quasi-geostrophic equation.
		For $\ga\in (0,1)$ and the initial data $\theta_0=r^{-(1+\al-\gamma)}\Om(\theta)$ with $\al\in (1,1+\gamma)$ and $\Om\in B_{L^p(\TT)}(1,\e), p>\f1{1-\ga}, \e$ small,
		they constructed a self-similar solution of the form
		\beno
		\theta(t,x)=\f 1 {t^{(1+\al-\ga)/(1+\al)}}\Theta\Big(\f x {t^{1/(1+\al)}}\Big).
		\eeno
		Our result deals with the important case $\ga=0$ for $\Om(\theta)\in L^1(\TT)$.
	\end{enumerate}

As an important application, we prove the existence of  weak solution of 2-D Euler equations when the initial vorticity is a measure, which is related to the vortex sheet solution. We denote by $\mathcal{M}(\TT)$ the set of signed Radon measures on $\TT$. Given a measure $\nu\in\mathcal{M}(\TT)$, we say that $\nu$ is \textit{$m$-fold symmetric} if
	\[\int_{\TT}\eta\left(\theta+\frac{2\pi}m\right)\,d\nu(\theta)=\int_{\TT}\eta(\theta)\,d\nu(\theta)\qquad \text{ for all }\eta\in C(\TT).\]
	For $\nu\in\mathcal{M}(\TT)$, we denote
	\[P_0\nu:=\frac{1}{2\pi}\nu(\TT)\in\RR, \qquad P_{\neq}\nu:=\nu-P_0\nu\in \mathcal{M}(\TT),\]
	and $\|\nu\|$ the total variation of $\nu$.
	
	\begin{cor}\label{maincor}
		Let $\mu>\frac12$ and $m\in \NN_+\setminus\{1\}$.  There exists $\varepsilon>0$, which is independent of $m\geq2$, such that for any $m$-fold symmetric $\mathring{\omega}\in \mathcal{M}(\TT)$ with the dominant condition
		\begin{equation}\label{1.6}
			\left\|P_{\neq}\mathring{\omega}\right\|\leq \varepsilon m^{\frac12}|P_0\mathring{\omega}|,
		\end{equation}
		a weak solution $\vvv$ of \eqref{2DEuler} with initial data \eqref{initial_data} exists. Moreover, the vorticity $\www\in C([0,\infty); \mathcal{D}'(\RR^2))$.
Here $ \mathcal{D}'$ is the space of distributions.
	\end{cor}
	
	\begin{rmk}
		The qualitative behavior of Delort's solution remains an important question. In a forthcoming work, we will explore the qualitative behavior of the weak solution constructed in Corollary \ref{maincor}. This is a key step toward the long-standing problem of the existence of self-similar algebraic spiral vortex sheet solutions. See \cite{Pull1, Pull2} for numerical results. For recent developments on the logarithmic spiral vortex sheet solution, we refer to \cite{CKO1, CKO2}.	
		\end{rmk}
In this paper,	$a\lesssim b$ stands for $a\leq Cb$ with some constant $C$, $a\gtrsim b$ stands for $b\lesssim a$ and $a\sim b$ stands for $a\lesssim b$ and $b\lesssim a$. In most circumstances, the constant $C$ depends only on the parameters $\alpha$ and $\mu$. Nevertheless, there are some exceptions where we explain more details accordingly, see e.g. the proof of Lemma \ref{uniqueness_ODE}.

	\section{Sketch of the proof}\label{Sec.sketch}
	
	The 2-D incompressible Euler equations \eqref{2dEuler} can be reformulated as follows: find two scalar functions $\www,\Psi$ defined for $(y,t)\in\mathbb R^2\times(0,\infty)$, such that
	\begin{equation}\label{2dEuler-vorticity}
		\begin{cases}
			\www_t+\nabla_y^\perp\Psi\cdot\nabla_y\www=0,\\
			\Delta_y\Psi=\www.
		\end{cases}
	\end{equation}
	We seek solutions of $\eqref{2dEuler-vorticity}$ which are (algebraically) \textit{self-similar}: with $x=t^{-\mu}y$,
	\begin{equation}\label{self-similar}
		\vvv(y,t)=t^{\mu-1}v(x),\qquad \www(y,t)=t^{-1}\omega(x),\qquad\Psi(y,t)=t^{2\mu-1}\psi(x)
	\end{equation}
	for some exponent $\mu>0$. Inserting \eqref{self-similar} into \eqref{2dEuler-vorticity} yields the equations
	\begin{equation}\label{2dEuler-selfsimilar}
		\begin{cases}
			\left(\nabla_x^\perp\psi-\mu x\right)\cdot\nabla_x\omega=\omega,\\
			\Delta_x\psi=\omega.
		\end{cases}
	\end{equation}
	To study spirals converging to a common origin, it is convenient to use polar coordinates:
	\[\mathbf{a}=(r,\theta),\qquad r=|x|,\qquad\theta\in\TT:=\RR/(2\pi\ZZ).\]
	Then we have
	\[\begin{cases}
		\frac1r\left(\psi_r\omega_\theta-\psi_\theta\omega_r\right)-\mu r\omega_r=\omega,\\
		\psi_{rr}+\frac1{r}\psi_{r}+\frac1{r^2}\psi_{\theta\theta}=\omega.
	\end{cases}\]
	Note that the second equation in the above system is indeed a Poisson equation in the coordinates $(\log r, \theta)$.\smallskip
	
	In Section \ref{sec_formulation}, we introduce a new system of coordinates $(\beta, \phi)\in\RR_+\times\TT$ with the relationship
	\begin{equation}\nonumber
		\theta=\beta+\phi,\qquad\psi_\beta=-\mu|x|^2=-\mu r^2.
	\end{equation}
	Under this new coordinates, the system can be transformed into $\mathcal F(\psi,\Omega)=0$, where
	\begin{equation}\label{Eq.F_expression}
		\begin{aligned}
			\mathcal{F}(\psi, \Om):&=\p_\vp\left(\frac{2\psi_\beta\psi_\vp}{\psi_{\beta\vp}}-\frac{\psi_{\beta\phi}}{\psi_{\beta\vp}}\cdot\frac{\psi_{\beta\vp}\psi_\phi-
				\psi_{\beta\phi}\psi_\vp}{2\psi_\beta}\right)\\
				&\qquad\qquad\qquad\qquad\qquad+\p_\phi\left(\frac{\psi_{\beta\vp}\psi_\phi-\psi_{\beta\phi}\psi_\vp}{2\psi_\beta}\right)+
			\frac{\psi_{\beta\vp}\psi_\vp^{-\frac1{2\mu}}}{2\mu}\Omega,
		\end{aligned}
	\end{equation}
	where $\p_\vp:=\p_\phi-\p_\beta$. See \eqref{nonlinear1}. This nonlinear equation has a special solution:
	\begin{equation}\nonumber
		\psi_0=\frac1{2\mu-1}\beta^{1-2\mu},\qquad\Omega_0=2-\mu^{-1}=:\gamma.
	\end{equation}
	
	Since the solution we find is singular in $\beta$, motivated by \cite{EJ}, we introduce the weighted H\"older type spaces. The choices of the functional spaces for the $m$-fold symmetric solution,  such as $\psi\in Y_m, \Om\in \WWW_m$ and $\mathcal{F}\in Z_m$, are quite subtle and are closely related to the properties of the solutions. Let us emphasize that our functional spaces have Banach algebraic properties, which will simplify the proof of $C^2$ regularity of the nonlinear map $\FF$ to a great extent. See Section \ref{Functional_regularity} for the definitions of various functional spaces.\smallskip

	We will apply the implicit function theorem to solve the nonlinear equation $\mathcal{F}(\psi,\Om)=0$ for $\Omega$ near $\Omega_0=\ga$.  The first step is to prove that nonlinear map
	$\mathcal{F}: Y_m\times \WWW_m\to Z_m$ is a $C^2$ map in a neighborhood of $(\psi_0, \Om_0)$, which will be conducted in Section \ref{sec_main_thm}. We remark that for the application of the implicit function theorem, it is enough to prove $C^1$ regularity of nonlinear map; however, here we need $C^2$ regularity, because we want to show that the small neighborhood given by the implicit function theorem is independent of $m$, although the functional spaces rely on $m$.
	The second step is to prove the solvability of the linearized problem around special solution $(\psi_0, \Om_0)$:
	\begin{equation}\label{Eq.Full}
		\LLL(\psi):=\frac{\pa\mathcal F}{\pa\psi}(\psi_0,\Omega_0)(\psi)=\frac1\mu\p_\vp(\beta H_\vp)+\frac\mu\beta H_{\phi\phi}+\frac\gamma{2\mu}\psi_\phi=G\in Z_m
	\end{equation}
	where $H=\psi+\frac\beta{2\mu}\psi_\beta$. That is to say, given $G\in Z_m$, we need to solve $\psi\in Y_m$ and to prove its uniqueness. This is the most difficult part of this paper.
	
	Before we sketch the proof of the invertibility of $\mathcal L$, we remark that $\mathcal L$ is a third-order differentiation operator acting on $\psi$. After introducing $H=\psi+\frac\beta{2\mu}\psi_\beta$, we convert $\mathcal L$ to a second-order operator acting on $H$, with an extra nonlocal term $\frac\gamma{2\mu}\psi_\phi$. This structure plays a crucial role in our proof of the invertibility of $\mathcal L$. 
	
First of all, we prove the uniqueness of the solution $(H,\psi)\in X_m\times Y_m$, namely,  Proposition \ref{uniqueness_prop}. The strategy is to express the linear homogeneous equation in Fourier modes with respect to $\phi\in\TT$, and then apply the theory of linear ODEs with regular singular points (see Lemma \ref{ODE_lem}) along with some properties of the generalized hypergeometric function. This is the first instance where the assumption $m\ge 2$ becomes crucial. Indeed, we can easily see that the proof of Lemma \ref{uniqueness_ODE} fails if $\mu\in\left(\frac12, 1\right]$ and $n=\pm1$,  in which case the corresponding linear homogeneous ODE may have a non-zero solution. Consequently, we cannot expect uniqueness. For details, see Subsection \ref{Uniqueness}.
	
It remains to prove the existence. The key point is to solve the simplified linearized problem: given $G\in Z_m$, find $H\in X_m$ solving
	\begin{equation}\label{Eq.simplified}
	\frac1\mu \p_\vp(\beta H_\vp)+\frac{\mu H_{\phi\phi}}{\beta}=G=\p_\vp F_1+\p_\phi F_2,
	\end{equation}
where we have used the fact that all functions in $Z_m$ can be represented as the form of $\p_\vp F_1+\p_\phi F_2$, according to the definition of $Z_m$ (see \eqref{target_Z_m}). Compared with $\LLL(\psi)$, we remove nonlocal term $\frac\gamma{2\mu}\psi_\phi$. Indeed, for high Fourier modes, the nonlocal term can be viewed as a perturbation. Notice that if $H$ solves \eqref{Eq.simplified},  then
	\[(\beta\p_\vp)^2H+(\mu\p_\phi)^2H=\left(\beta\p_\vp-\ii\mu\p_\phi\right)\frac{\mu F_1+\ii\beta F_2}{2}+\left(\beta\p_\vp+\ii\mu\p_\phi\right)\frac{\mu F_1-\ii\beta F_2}{2}.\]
	Let $Q_1, Q_2$ be solutions to $\left(\beta\p_\vp+\ii\mu\p_\phi\right)Q_1=\mu F_1+\ii\beta F_2$ and $\left(\beta\p_\vp-\ii\mu\p_\phi\right)Q_2=\mu F_1-\ii\beta F_2$, respectively. Then $H=\frac{Q_1+Q_2}{2}$. Thus, it is enough to solve the following first order system
	\begin{equation}\nonumber
		(\beta\p_\vp\pm \ii\mu\p_\phi)Q=G.
	\end{equation}
This task is achieved by Proposition \ref{linear_basic_estimate1}, whose proof is highly involved and independent of other parts of this paper. Hence, it is moved to Appendix \ref{AppendixB1}. For this first-order system, we can write down the explicit formula for the solution, which is a convolution integral formula, by using Fourier series (see \eqref{Q_expression}). Due to the special shape of our new coordinates, which consist of infinitely many circles near the physical origin (see Figure \ref{algebraic}), we introduce a partition of unity to explore the interactions among these circles more carefully (see \eqref{Q_k_expression}--\eqref{Q_k>0}). Another key point is the requirement $m\geq2$, in particular, we require $\hat G_{\pm1}=0$, where $\hat G_{\pm 1}$ is the $\pm1$ Fourier modes of $G$ in the new coordinates. Recall that our solution formula is a convolution of a kernel with the function $G$. The condition $\hat G_{\pm 1}=0$ allows us to subtract some of the non-contributing terms from the kernel to obtain better estimates on the integral formula (see \eqref{5.14} for instance). This condition is also necessary for the same reason as in the uniqueness part: we cannot show the triviality of the kernel of the linear operator, let alone the coercive estimates. For details, see Subsection \ref{section_Basic_lin}.
	
	After proving the solvability of the simplified linearized problem \eqref{Eq.simplified}, we turn to the full linearized problem \eqref{Eq.Full}. We note that the inhomogeneous problem \eqref{Eq.Full} can be converted into the following system (see Subsection \ref{Full_linear}):
	\begin{equation}\label{Eq.Q-psi}
			\begin{cases}
				\frac1\mu\p_\vp(\beta Q_\vp)+\frac{\mu Q_{\phi\phi}}{\beta}+\frac\gamma{2\mu}\psi_\phi=0,\\
				H+Q=\psi+\frac\beta{2\mu}\psi_\beta,
			\end{cases}
	\end{equation}
	where $H$ is given and $(Q,\psi)$ is the unknown. The solvability of \eqref{Eq.Q-psi} is stated in Proposition \ref{exist_2_prop}, whose proof is provided in Subsection \ref{Subsec.Proof-full}. Our strategy is to investigate the system \eqref{Eq.Q-psi} separately in high and low frequencies with respect to $\phi\in\mathbb T$.
	
	\begin{itemize}
		\item For high frequencies (higher than $N$), we note that the first equation of \eqref{Eq.Q-psi} implies $Q=-\mathcal H^N\psi_\phi$ for some operator $\mathcal H^N$ defined in \eqref{Eq.T_0^N_def}, which is provided by the existence theory for the simplified linearized problem \eqref{Eq.simplified}. The second equation of \eqref{Eq.Q-psi} implies $\psi=T_0^N(H+Q)$ for some operator $T_0^N$ defined in \eqref{Eq.T_0^N_def} (with $T_0$ defined in \eqref{H_to_F}). We define $T^N=\mathcal H^N\pa_\phi T_0^N$, and thus \eqref{Eq.Q-psi} becomes $(\operatorname{id}+T^N)(Q+H)=H$. We are able to show that $T^N$ has an operator norm that is $\mathcal O(1/N)$. Consequently, $\operatorname{id}+T^N$ is invertible for sufficiently large $N$. For details, see Proposition \ref{high_frequency}.
		
		\item For low frequencies, we can convert \eqref{Eq.Q-psi} into a finite number of third-order ODE systems for each Fourier mode. To avoid some technical arguments regarding the existence of solutions to third-order ODEs, we note that \eqref{Eq.Q-psi} is equivalent to the following system:
		\begin{equation}\label{Eq.Q-psi2}
			\begin{cases}
				\frac1\mu\p_\vp(\beta Q_\vp)+\frac\mu\beta Q_{\phi\phi}+\frac\gamma\beta\rho(\beta)Q\\
				\qquad=-\frac\gamma\beta\rho(\beta)H-\frac\gamma{2\mu}\p_\vp(\rho(\beta)\psi)-
				\frac\gamma{2\mu}\left(\rho'(\beta)-\frac{2\mu}\beta\rho(\beta)\right)\psi-\frac\gamma{2\mu}\left(1-\rho(\beta)\right)\psi_\phi\\
				H+Q=\psi+\frac{\beta}{2\mu}\psi_\beta,
			\end{cases}
		\end{equation}
		where $\rho$ is a fixed smooth bump function belonging to $C^\infty([0, \infty);[0,1])$ such that $\rho(\beta)=0$ for $\beta\in[0,1]$ and $\rho(\beta)=1$ for $\beta\geq2$. Using the existence theory for solutions to second-order ODEs (see Appendix \ref{AppendixB}), we can find $Q_1=T_1\psi$ and $Q_2=T_2H$, where $T_1$ and $T_2$ denote the solution operators, such that
		\begin{align*}
				&\frac1\mu\p_\vp(\beta \p_\vp Q_1)+\frac\mu\beta \p_\phi^2Q_{1}+\frac\gamma\beta\rho(\beta)Q_1\\
				&\ \ =-\frac\gamma{2\mu}\p_\vp(\rho(\beta)\psi)-\frac\gamma{2\mu}\left(\rho'(\beta)-\frac{2\mu}\beta\rho(\beta)\right)\psi-
				\frac\gamma{2\mu}\left(1-\rho(\beta)\right)\psi_\phi,\\
				&\qquad\qquad\frac1\mu\p_\vp(\beta \p_\vp Q_2)+\frac\mu\beta \p_\phi^2Q_{2}+\frac\gamma\beta\rho(\beta)Q_2=-\frac\gamma\beta\rho(\beta)H.
		\end{align*}
		In this way, \eqref{Eq.Q-psi} is converted to $Q=T_1\psi+T_2H$ and $\psi=T_0(H+Q)$ ($T_0$ is defined in \eqref{H_to_F}), which is equivalent to $(\operatorname{id}-T_1T_0)(H+Q)=T_2H+H$.
 By carefully choosing larger spaces, we can prove the compactness of $T_0$ in these spaces; see Lemma \ref{compact}. Consequently, Fredholm's theory implies that $\operatorname{id}-T_1T_0$ is a bijection in the larger space, with injectivity obtained using the same ideas as in Subsection \ref{Uniqueness}. Now, we have obtained the solution $(Q,\psi)$ to \eqref{Eq.Q-psi} in some larger space. A standard regularity analysis implies that $(Q,\psi)$ lies in our desired (smaller) space; see \hyperlink{step2}{\bf Step 2} in the proof of Proposition \ref{low_frequency}.
	\end{itemize}
	
	
	In Section \ref{sec_original_coor}, we recover the solutions obtained by the implicit function theorem in the original physical coordinates. We first check the invertible of the change of coordinates $x\mapsto (\beta,\phi)$. Then we check that the solutions in the original physical coordinates are actually weak solutions of the 2-D incompressible Euler equations \eqref{2DEuler} in the sense of Definition \ref{weak_sol_def}. Note that our solution is too weak to verify the vorticity formulation \eqref{2dEuler-vorticity}. \smallskip
	
	In Section \ref{sec_main_cor}, we prove Corollary \ref{maincor}. Our strategy is to regularize the initial data $\mathring{\omega}\in\mathcal{M}(\TT)$. Let $F_N=N\chi_{(0,1/N)}$,  then $F_N\geq 0$ and $\|F_N\|_{L^1(\TT)}=1$. For each $N\in\NN_+$, we define
	\begin{equation}\label{regularization}
		\mathring{\omega}_N(\theta)=(F_N*\mathring{\omega})(\theta)=N\int_{\theta-1/N}^{\theta}\,d\mathring{\omega}(\theta'),\qquad \theta\in\TT,
	\end{equation}
	then $\|\mathring{\omega}_N\|_{L^1(\TT)}\leq \|F_N\|_{L^1(\TT)}\|\mathring{\omega}\|=\|\mathring{\omega}\|$ and $\mathring{\omega}_N\rightarrow\mathring{\omega}$ in $\mathcal{M}(\TT)$. Moreover, it follows from \eqref{1.6} that \eqref{1.5} holds for each $\mathring{\omega}_N$. Hence, Theorem \ref{mainthm} is applicable to each $\mathring{\omega}_N$ and we get a sequence of weak solutions $\vvv_N$ of the 2-D Euler equations. It remains to take the limit $N\to\infty$ to get a weak solution for $\mathring{\omega}$.

	\section{New coordinates and reformulation}\label{sec_formulation}

	\subsection{New coordinates}
	We regard the first equation in \eqref{2dEuler-selfsimilar} as a linear, first order equation for $\omega$. The characteristic curves, also known as the pseudo-streamlines, are the integral curves of the vector field $\nabla_x^\perp\psi-\mu x$. Following the ideas in \cite{ELL2013, ELL2016}, it will be convenient to make a change of variables $\aaa=(r,\theta)\mapsto\bbb=(\beta,\phi)$, such that
	\begin{itemize}
		\item Pseudo-streamlines have the equation $\phi=$constant.
		\item $\theta=\beta+\phi$.
		\item For fixed $\phi$, we have
		$$\lim_{\beta\to0+}r(\beta,\phi)=+\infty,\qquad\lim_{\beta\to0+}\theta(\beta,\phi)=\phi.$$
	\end{itemize}
	Let us remark that the change of coordinates depends on the solution. So, after we construct the solution in the new variables, we need to check that the change of coordinates is non-degenerate, which will be demonstrated in Subsection \ref{sec_invertibility}.
	
	Assume that $s\mapsto x(s)$ is a pseudo-streamline, i.e.,
	\[\frac{dx(s)}{ds}=\nabla_{x}^\perp\psi(x(s))-\mu x(s).\]
	Since $\phi(x(s))$ is independent of $s$, we have
	\begin{equation}\label{streamline}
		\begin{aligned}
			0&=\phi_{x_1}x_1'(s)+\phi_{x_2}x_2'(s)\\
			&=\phi_{x_1}\left(-\psi_{x_2}-\mu x_1(s)\right)+\phi_{x_2}\left(\psi_{x_1}-\mu x_2(s)\right)\\
			&=-\phi_{x_1}\left(\psi_\beta\beta_{x_2}+\psi_\phi\phi_{x_2}\right)-\mu x_1\phi_{x_1}+\phi_{x_2}\left(\psi_\beta\beta_{x_1}+\psi_\phi\phi_{x_1}\right)-\mu x_2\phi_{x_2}\\
			&=\left(\beta_{x_1}\phi_{x_2}-\beta_{x_2}\phi_{x_1}\right)\psi_\beta-\mu(x_1\phi_{x_1}+x_2\phi_{x_2}).
		\end{aligned}
	\end{equation}
	We now observe that the $2\times2$ Jacobian matrix of the variable-transformation satisfies
	\begin{equation}\label{Jacobi-x-b}
		\begin{pmatrix} x_{1,\beta} & x_{1,\phi}\\ x_{2,\beta} & x_{2,\phi}\end{pmatrix}=\begin{pmatrix}\beta_{x_1} & \beta_{x_2}\\ \phi_{x_1} & \phi_{x_2} \end{pmatrix} ^{-1}=\frac1{\beta_{x_1}\phi_{x_2}-\beta_{x_2}\phi_{x_1}}\begin{pmatrix}\phi_{x_2} & -\beta_{x_2}\\ -\phi_{x_1} & \beta_{x_1} \end{pmatrix}.
	\end{equation}
	In view of \eqref{Jacobi-x-b}, it follows from \eqref{streamline}  that
	\begin{equation}\label{check{psi}_beta}
		\psi_\beta+\mu\left(x_1x_{2,\beta}-x_2x_{1,\beta}\right)=0.
	\end{equation}
	Using $x=(r\cos\theta,r\sin\theta)$ and $\theta=\beta+\phi$ gives
	\[x_1x_{2,\beta}-x_2x_{1,\beta}=r\cos\theta\left(r_{\beta}\sin\theta+r\cos\theta\theta_{\beta}\right)-
	r\sin\theta\left(r_{\beta}\cos\theta-r\sin\theta\theta_{\beta}\right)=r^2\theta_\beta=r^2.\]
	This along with \eqref{check{psi}_beta} yields
	\begin{equation}\label{beta-x}
		\psi_\beta=-\mu|x|^2=-\mu r^2.
	\end{equation}
	
	We rewrite the change of variable formulas as
	\[\theta=\beta+\phi,\qquad\psi_\beta=-\mu r^2=-\mu|x|^2.\]
	Hence, $2rr_\beta=-\frac1\mu\psi_{\beta\beta},\, 2rr_\phi=-\frac1\mu\psi_{\beta\phi}.$
	Denote $\p_\vp:=\p_\phi-\p_\beta$, then
	\begin{equation}\label{beta_to_r}
		\begin{pmatrix}
			r_\beta & r_\phi \\ \theta_{\beta} & \theta_{\phi}
		\end{pmatrix}=\begin{pmatrix}
			-\frac{\psi_{\beta\beta}}{2\mu r} & -\frac{\psi_{\beta\phi}}{2\mu r} \\ 1 & 1
		\end{pmatrix},
	\end{equation}
	\begin{equation}\label{r_to_beta}
		\begin{pmatrix}
			\beta_r & \beta_\theta \\ \phi_r & \phi_\theta
		\end{pmatrix}=\begin{pmatrix}
			r_\beta & r_\phi \\ \theta_{\beta} & \theta_{\phi}
		\end{pmatrix}^{-1}=\frac1{\psi_{\beta\vp}}\begin{pmatrix}
			2\mu r & \psi_{\beta\phi} \\ -2\mu r & -\psi_{\beta\beta}
		\end{pmatrix}.
	\end{equation}

	\subsection{Formulation in the new coordinates}
	Now we recall the equation \eqref{2dEuler-selfsimilar}:
	\begin{equation}\label{2dEuler-selfsimilar1}
		\begin{cases}
			\left(\nabla_x^\perp\psi-\mu x\right)\cdot\nabla_x\omega=\omega,\\
			\Delta_x\psi=\omega,
		\end{cases}
	\end{equation}
	and the change-of-variable formulas
	\begin{equation}\label{changeofvariable}
		\theta=\beta+\phi,\qquad\psi_\beta=-\mu|x|^2=-\mu r^2.
	\end{equation}
	
	In the new coordinates $(\beta,\phi)$, $\partial_\beta$ is tangential to pseudo-streamlines, so the first equation in \eqref{2dEuler-selfsimilar1} can be easily solved by the standard characteristic-curve method. We write the first equation of \eqref{2dEuler-selfsimilar1} in $(\beta,\phi)$:
	\begin{align*}
		\omega&=(-\psi_{x_2}-\mu x_1)(\omega_\beta\beta_{x_1}+\omega_\phi\phi_{x_1})+(\psi_{x_1}-\mu x_2)(\omega_\beta\beta_{x_2}+\omega_\phi\phi_{x_2})\\
		&=[\psi_{x_1}\beta_{x_2}-\psi_{x_2}\beta_{x_1}-\mu(x_1\beta_{x_1}+x_2\beta_{x_2})]\omega_\beta+\\
		&\qquad\qquad\qquad[\psi_{x_1}\phi_{x_2}-\psi_{x_2}\phi_{x_1}-\mu(x_1\phi_{x_1}+x_2\phi_{x_2})]\omega_\phi.
	\end{align*}
	Using \eqref{streamline}, we have $\psi_{x_1}\phi_{x_2}-\psi_{x_2}\phi_{x_1}-\mu(x_1\phi_{x_1}+x_2\phi_{x_2})=0$ and
	\begin{align*}
		&\ \ \ \psi_{x_1}\beta_{x_2}-\psi_{x_2}\beta_{x_1}-\mu(x_1\beta_{x_1}+x_2\beta_{x_2})\\
		&=(\psi_\beta\beta_{x_1}+\psi_\phi\phi_{x_1})\beta_{x_2}-(\psi_\beta\beta_{x_2}+\psi_\phi\phi_{x_2})\beta_{x_1}-\mu(x_1\beta_{x_1}+x_2\beta_{x_2})\\
		&=\psi_\phi(\phi_{x_1}\beta_{x_2}-\phi_{x_2}\beta_{x_1})-\mu(x_1\beta_{x_1}+x_2\beta_{x_2})\\
		&=-\mu\frac{\psi_\phi}{\psi_\beta}(x_1\phi_{x_1}+x_2\phi_{x_2})-\mu(x_1\beta_{x_1}+x_2\beta_{x_2});
	\end{align*}
	Also, it follows from \eqref{r_to_beta} that
	\[x_1\phi_{x_1}+x_2\phi_{x_2}=r\phi_r=-\frac{2\mu r^2}{\psi_{\beta\vp}}=\frac{2\psi_\beta}{\psi_{\beta\vp}},\]
	\[x_1\beta_{x_1}+x_2\beta_{x_2}=r\beta_r=\frac{2\mu r^2}{\psi_{\beta\vp}}=-\frac{2\psi_\beta}{\psi_{\beta\vp}}.\]
	Therefore, the first equation of \eqref{2dEuler-selfsimilar1} in $(\beta,\phi)$ variables is
	\[\omega=\left[-\mu\frac{\psi_\phi}{\psi_\beta}\cdot\frac{2\psi_\beta}{\psi_{\beta\vp}}+\mu\frac{2\psi_\beta}{\psi_{\beta\vp}}\right]\omega_\beta=
	-2\mu\frac{\psi_\vp}{\psi_{\beta\vp}}\omega_\beta\]
	and the solution is
	\begin{equation}\label{omega_Omega}
		\omega=\psi_\varphi^{-\frac1{2\mu}}\Omega(\phi),
	\end{equation}
	where $\Omega=\Omega(\phi)$ is some function that can be chosen freely as the data. The relationship between $\Omega$ and the initial data $\www|_{t=0}$ will be investigated later, see Proposition \ref{prop_data}.
	
	Now we write the Poisson equation in \eqref{2dEuler-selfsimilar1} in the new coordinates $(\beta, \phi)$. The Poisson equation in the polar coordinates is $r(r\psi_r)_{,r}+\psi_{\theta\theta}=r^2\omega$. Using \eqref{r_to_beta}, \eqref{changeofvariable} and $\p_\vp=\p_\phi-\p_\beta$ we can easily compute that
	\begin{equation}\label{p_rtheta}
		r\p_r=\frac{2\psi_\beta}{\psi_{\beta\vp}}\p_\vp,\qquad \p_\theta=\p_\phi-\frac{\psi_{\beta\phi}}{\psi_{\beta\vp}}\p_\vp.
	\end{equation}
	Hence, the Poisson equation is converted to
	\[\frac{2\psi_\beta}{\psi_{\beta\vp}}\p_\vp\left(\frac{2\psi_\beta\psi_\vp}{\psi_{\beta\vp}}\right)+
	\left(\p_\phi-\frac{\psi_{\beta\phi}}{\psi_{\beta\vp}}\p_\vp\right)\left(\frac{\psi_{\beta\vp}\psi_\phi-\psi_{\beta\phi}\psi_\vp}{\psi_{\beta\vp}}\right)=
	-\frac{\psi_\beta}{\mu}\omega.\]
	Rearranging and using \eqref{omega_Omega}, we get the nonlinear equation for $\psi$ and $\Omega$ as follows
	\begin{equation}\label{nonlinear1}
		\p_\vp\left(\frac{2\psi_\beta\psi_\vp}{\psi_{\beta\vp}}-\frac{\psi_{\beta\phi}}{\psi_{\beta\vp}}\cdot\frac{\psi_{\beta\vp}\psi_\phi-
			\psi_{\beta\phi}\psi_\vp}{2\psi_\beta}\right)+\p_\phi\left(\frac{\psi_{\beta\vp}\psi_\phi-\psi_{\beta\phi}\psi_\vp}{2\psi_\beta}\right)+
		\frac{\psi_{\beta\vp}\psi_\vp^{-\frac1{2\mu}}}{2\mu}\Omega=0.
	\end{equation}
	
	\subsection{Radially symmetric solutions}\label{sec_radial}
	Radially symmetric solutions of \eqref{2dEuler-selfsimilar1}, defined for all $r>0$, can be easily constructed. As in \cite{Bre, ELL2013, ELL2016}, such solutions play a fundamental role in our analysis.
	
	In the radially symmetric case, i.e., $\omega=\omega(r)$ and $\psi=\psi(r)$, the first equation in \eqref{2dEuler-selfsimilar1} is reduced to $-\mu r\omega_r=\omega$. Hence, the vorticity has the form $\omega(r)=c_0r^{-1/\mu}$ for some constant $c_0$. In turn, the second equation in \eqref{2dEuler-selfsimilar1} yields
	\[\psi_{rr}+\frac1{r}\psi_{r}=c_0r^{-\frac1\mu}.\]
	Denote $\gamma=2-\mu^{-1}$.
	Thus, the stream function is computed as
	\[\psi(r)=c_0\left(2-\frac1\mu\right)^{-2}r^{2-\frac1\mu}=c_0\gamma^{-2}r^{2-\frac1\mu}.\]
	
	Now we rewrite this solution in terms of the new coordinates $(\beta,\phi)$. We first calculate the pseudo-streamlines of this radially symmetric solution. Assume that $s\mapsto x(s)$ is a pseudo-streamline. Then
	\[\frac{dx(s)}{ds}=\nabla_{x}^\perp\psi(x(s))-\mu x(s).\]
	Since $\psi$ is radially symmetric, it will be more convenient to write the equation for the pseudo-streamline in the polar coordinates: write $x(s)=(r(s)\cos\theta(s), r(s)\sin\theta(s))$, then we have
	\[\begin{cases}
		\frac{dr(s)}{ds}=-\mu r(s),\\
		\frac{d\theta(s)}{ds}=c_0\gamma^{-1}\left(r(s)\right)^{-\frac1\mu}.
	\end{cases}\]
	Hence, $r(s)=r(0)e^{-\mu s}$ and
	\[\theta(s)=\theta(0)-c_0\gamma^{-1}r(0)^{-\frac1\mu}+c_0\gamma^{-1}r(0)^{-\frac1\mu}e^s.\]
	Pseudo-streamlines are algebraic spirals around the origin, see Figure \ref{algebraic}.
	\begin{figure}[htbp]
		\includegraphics[width=1\textwidth]{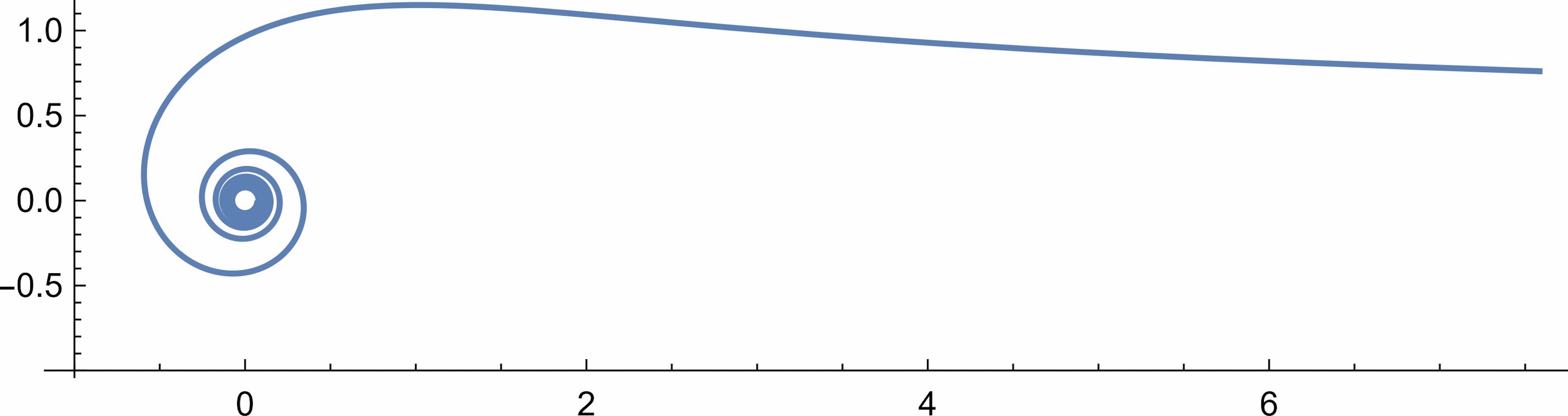}
		\caption{The algebraic spiral $\theta=\frac32 r^{-4/3}$ obtained by taking $c_0=1,\mu=3/4$.}
		\label{algebraic}
	\end{figure}
	We want to determine the relationship between $(\beta,\phi)$ and $(r,\theta)$, so that the new coordinates are adapted to our radially symmetric solutions. Since $\phi$ is constant along pseudo-streamlines, we can take
	\[\phi=\theta-c_0\gamma^{-1}r^{-\frac1\mu},\qquad\beta=c_0\gamma^{-1}r^{-\frac1\mu}.\]
	In this coordinates, our radially symmetric solutions can be written as
	\[\omega(\beta,\phi)=\gamma\beta,\qquad\psi(\beta,\phi)=c_0^{2\mu}\gamma^{-1-2\mu}\beta^{1-2\mu},\qquad \Omega(\phi)=c_0\mu^{\frac1{2\mu}}.\]
	
	Since the value of $c_0$ does not affect the subsequent analysis, taking $c_0=\gamma\mu^{-\frac1{2\mu}}$ yields a special solution of the nonlinear equation \eqref{nonlinear1}:
	\begin{equation}\label{special_solution}
		\psi_0=\frac1{2\mu-1}\beta^{1-2\mu},\qquad\Omega_0=\gamma.
	\end{equation}

	\section{Weighted functional spaces with $m$-fold symmetry}\label{Functional_regularity}
	
In this section, we introduce the functional framework for the implicit function theorem, which will be used to solve the nonlinear equation \eqref{nonlinear1}. In Subsection \ref{sec_pfs}, we introduce the weighted Hölder-type space. In Subsection 4.2, we introduce the functional spaces for $\psi$ and $\Om$. In Subsection \ref{m-fold}, we introduce the functional spaces with m-fold symmetry and establish the Banach algebraic properties for these spaces.

Throughout this section, we assume $\al\in (0,1), \mu>\f12$, and recall $\pa_{\varphi}=\pa_{\phi}-\pa_\beta$. In this section, the implicit constants in all $\lesssim$ and  $\sim$ depend only on $\alpha\in(0,1)$ and $\mu>1/2$. (In fact, some of them are independent of all parameters and all functions.) In particular, the implicit constants in Subsection \ref{m-fold} are independent of $m\in\NN_+$.

	\subsection{Weighted H\"older type spaces}\label{sec_pfs}
	
	We denote by $C_b\big((0,+\infty)\times\TT\big)$ the bounded continuous functions in $(0,+\infty)\times\TT$. For $\al\in (0,1)$, we define the weighted H\"older space
\beno
C_\beta^\al:=\Big\{f(\beta,\phi)\in C_b\big((0,+\infty)\times\TT\big): \|f\|_{C_\beta^\al}<+\infty\Big\},
\eeno
where	
\begin{align*}
		\|f\|_{C_\beta^\alpha}:=\|\langle\beta\rangle^{\alpha}f\|_{L^\infty}+\sup_{\phi\in\TT}\sup_{\substack{0<\beta_2<\beta_1<2\beta_2\\
				0<|\beta_1-\beta_2|<1}}|\beta_1+\beta_2|^\alpha\frac{|f(\beta_1,\phi)-f(\beta_2,\phi)|}{|\beta_1-\beta_2|^\alpha}.
	\end{align*}
	Here $ \langle\beta\rangle=\sqrt{\beta^2+1}$. We note that $f$ is required to be H\"older continuous only in the single $\beta$-direction.
	
As we mentioned in Section \ref{Sec.sketch}, our definition of the weighted H\"older space is motivated by \cite{EJ}, where Elgindi and Jeong used the space $\overset{\circ}{C^\al}(\mathbb R^2)$, which  satisfies  that if $f\in \overset{\circ}{C^\al}$, then $f$ is bounded and $|x|^\al f\in C_*^\al(\mathbb R^2)$, where $C_*^\al$ denots the ordinary H\"older space. Elgindi and Jeong utilized such weighted Hölder spaces to study the well-posedness of $m$-fold solutions to the 2-D incompressible Euler equations. We also note that similar, albeit distinct, weighted Hölder spaces have been employed in the recent work \cite{GG}.	 Some of the differences are as follows.
\begin{itemize}
    \item We only require $f$ to be bounded near $\beta=0$, whereas \cite{GG} requires $f$ to vanish at $\beta=0$;
    \item 
   The  Lipschitz continuity with respect to $\phi\in\mathbb T$ is not included in the definition of $Y_0$ in \eqref{Y0}, and we can prove such Lipschitz (in fact $C^1$) continuity  in Lemma \ref{lem_nonlinear}, whereas \cite[formula (46)]{GG}  requires such Lipschitz continuity directly.
\end{itemize}	

	Now we prove some useful properties of the weighted H\"older space $C_\beta^\alpha$. The following lemma shows that this norm is equivalent to the usual one.
	
	\begin{lem}[Equivalent norm on $C_\beta^\alpha$]\label{equvi_Holder_norm}
	It holds that
		\begin{equation}
			\|f\|_{C_\beta^\alpha}\sim \|\langle\beta\rangle^{\alpha}f\|_{L^\infty}+\sup_{\phi\in\TT}\sup_{\substack{\beta_1\neq\beta_2\\
					\beta_1, \beta_2>0}}|\beta_1+\beta_2|^\alpha\frac{|f(\beta_1,\phi)-f(\beta_2,\phi)|}{|\beta_1-\beta_2|^\alpha}.
		\end{equation}
	\end{lem}
	\begin{proof}
		The ``$\lesssim$'' part is a direct consequence of the definition. We only need to prove the ``$\gtrsim$'' part. Let $\beta_1, \beta_2>0$ be such that $\beta_1\neq \beta_2$. Without loss of generality, we assume that $0<\beta_2<\beta_1$. If $\beta_1<2\beta_2$ and $\beta_1-\beta_2<1$, then we get by the definition of the weighted H\"older norm that
		\[|\beta_1+\beta_2|^\alpha\frac{|f(\beta_1,\phi)-f(\beta_2,\phi)|}{|\beta_1-\beta_2|^\alpha}\lesssim\|f\|_{C_\beta^\alpha}.\]
		
		If $\beta_1<2\beta_2$ and $\beta_1-\beta_2\geq1$, then
		\[|\beta_1+\beta_2|^\alpha\frac{|f(\beta_1,\phi)-f(\beta_2,\phi)|}{|\beta_1-\beta_2|^\alpha}\leq (2\beta_1)^\alpha|f(\beta_1,\phi)|+(3\beta_2)^\alpha|f(\beta_2,\phi)|\lesssim\|\langle\beta\rangle^{\alpha}f\|_{L^\infty}\lesssim\|f\|_{C_\beta^\alpha}.\]
		
		If $\beta_1\geq 2\beta_2$, then $\beta_1+\beta_2\leq 2\beta_1$ and $\beta_1-\beta_2\geq\frac12\beta_1$, hence
		\[|\beta_1+\beta_2|^\alpha\frac{|f(\beta_1,\phi)-f(\beta_2,\phi)|}{|\beta_1-\beta_2|^\alpha}\leq 4^\alpha\Big(\left|f(\beta_1,\phi)\right|+\left|f(\beta_2,\phi)\right|\Big)\lesssim\|f\|_{L^\infty}\lesssim\|\langle\beta\rangle^{\alpha}f\|_{L^\infty}
\lesssim\|f\|_{C_\beta^\alpha}.\]
		
		This completes the proof of the lemma.
	\end{proof}

The following lemma will be frequently used.

\begin{lem} It holds that (for $f,\partial_{\beta}f\in C((0,\infty)\times\mathbb{T})$)
	\begin{align}\label{f1}
		\|f\|_{C_\beta^\alpha}\lesssim\|\langle\beta\rangle^{\alpha} f\|_{L^\infty}+\|\langle\beta\rangle^{\alpha-1}\beta\partial_{\beta}f\|_{L^\infty}.
	\end{align}	
\end{lem}

\begin{proof}
If $0<\beta_2<\beta_1<2\beta_2$, $\beta_1-\beta_2<1$ and $\phi\in\TT$, then by the mean value theorem, there exists $\beta_0\in (\beta_2, \beta_1)$ such that $|f(\beta_1, \phi)-f(\beta_2, \phi)|\leq|\p_\beta f(\beta_0, \phi)||\beta_1-\beta_2|$, hence
 \[|\beta_1+\beta_2|^\alpha\frac{|f(\beta_1,\phi)-f(\beta_2,\phi)|}{|\beta_1-\beta_2|^\alpha}\leq \frac{\langle\beta_0\rangle^{1-\al}}{\beta_0}|\beta_1-\beta_2|^{1-\al}|\beta_1+\beta_2|^\al\|\langle\beta\rangle^{\alpha-1}\beta\partial_{\beta}f\|_{L^\infty}.\]
Thanks to $\beta_0\sim \beta_2\sim \beta_1+\beta_2$ and $\langle\beta_0\rangle\sim \langle\beta_2\rangle$, we have
 \[|\beta_1+\beta_2|^\alpha\frac{|f(\beta_1,\phi)-f(\beta_2,\phi)|}{|\beta_1-\beta_2|^\alpha}\lesssim \left(\frac{\langle\beta_2\rangle}{\beta_2}|\beta_1-\beta_2|\right)^{1-\al}\|\langle\beta\rangle^{\alpha-1}\beta\partial_{\beta}f\|_{L^\infty}.\]
 Finally, if $\beta_2<1$, then by $\beta_2<\beta_1<2\beta_2$ we have
 $\frac{\langle\beta_2\rangle}{\beta_2}|\beta_1-\beta_2|\lesssim \frac{|\beta_1-\beta_2|}{\beta_2}\lesssim 1;$
 if $\beta_2\geq1$, then by $0<\beta_1-\beta_2<1$ we have
 $\frac{\langle\beta_2\rangle}{\beta_2}|\beta_1-\beta_2|\lesssim|\beta_1-\beta_2|\lesssim1.$ Hence, \eqref{f1} follows.
  \end{proof}

The following lemma shows the algebraic property of  $C_\beta^\al$.

	\begin{lem}[Algebraic property]\label{AppendixA_Algebra}
		Let $\alpha\in(0,1)$ and $f=f(\beta,\phi): (0,+\infty)\times\TT\to\RR$, $g=g(\beta,\phi): (0,+\infty)\times\TT\to\RR$.
		\begin{enumerate}[(1)]
			\item If $ \langle\beta\rangle^{-\alpha}f\in C_{\beta}^\alpha$ and $ g\in C_{\beta}^\alpha$, then $ fg\in C_{\beta}^\alpha$ and
			\[\| fg\|_{C_\beta^\alpha}\lesssim\|\langle\beta\rangle^{-\alpha} f\|_{C_\beta^\alpha}\| g\|_{C_\beta^\alpha}.\]
			\item If $ f\in C_{\beta}^\alpha$ and $ g\in C_{\beta}^\alpha$, then $ fg\in C_{\beta}^\alpha$ and
			\[\| fg\|_{C_\beta^\alpha}\lesssim\| f\|_{C_\beta^\alpha}\| g\|_{C_\beta^\alpha}.\]
		\end{enumerate}
	\end{lem}
	\begin{proof}
		\begin{enumerate}[(1)]
			\item Firstly, we have
			\[\left\|\langle\beta\rangle^{\alpha}fg\right\|_{L^\infty}\leq \left\|f\right\|_{L^\infty}\left\|\langle\beta\rangle^{\alpha}g\right\|_{L^\infty}\leq \|\langle\beta\rangle^{-\alpha} f\|_{C_\beta^\alpha}\| g\|_{C_\beta^\alpha}.\]
			For simplicity, we denote $\tilde f=\langle\beta\rangle^{-\alpha}f$. For any $\beta_1, \beta_2>0$ such that $0<\beta_2<\beta_1<2\beta_2$, $0<|\beta_1-\beta_2|<1$, we have
			\begin{align*}
				&\ \ \ \left| f(\beta_1,\phi)g(\beta_1,\phi)-  f(\beta_2,\phi)g(\beta_2,\phi)\right|=\left| \langle\beta_1\rangle^\alpha\tilde f(\beta_1,\phi)g(\beta_1,\phi)-  \langle\beta_2\rangle^\alpha\tilde f(\beta_2,\phi)g(\beta_2,\phi)\right|\\
				&\qquad\leq\left| \tilde f(\beta_1,\phi)- \tilde f(\beta_2,\phi)\right|\langle\beta_1\rangle^\alpha|g(\beta_1,\phi)|+\left|\langle\beta_1\rangle^\alpha-\langle\beta_2\rangle^\alpha\right|\left|\tilde f(\beta_2,\phi)\right||g(\beta_1,\phi)|\\
				&\qquad\qquad\qquad\qquad+\langle\beta_2\rangle^\alpha\left|\tilde f(\beta_2,\phi)\right|\left| g(\beta_1,\phi)- g(\beta_2,\phi)\right|\\
				&\qquad \leq\frac{\left|\langle\beta_1\rangle^\alpha-\langle\beta_2\rangle^\alpha\right|}{\langle\beta_1\rangle^\alpha\langle\beta_2\rangle^\alpha}\left\| \tilde f\right\|_{C_\beta^\alpha}\|g\|_{C_\beta^\alpha}+2\frac{ |\beta_1-\beta_2|^\alpha}{ |\beta_1+\beta_2|^\alpha}\left\| \tilde f\right\|_{C_\beta^\alpha}\|g\|_{C_\beta^\alpha}\lesssim \frac{ |\beta_1-\beta_2|^\alpha}{ |\beta_1+\beta_2|^\alpha}\left\| \tilde f\right\|_{C_\beta^\alpha}\|g\|_{C_\beta^\alpha},
			\end{align*}
			where in the last inequality we have used the fact
			\begin{equation*}\label{4.5}
				\frac{\left|\langle\beta t\rangle^\alpha-\langle\beta \rangle^\alpha\right|}{\langle\beta t\rangle^\alpha\langle\beta \rangle^\alpha}\lesssim \frac{|t-1|^\alpha}{|t+1|^\alpha},\qquad \forall\ \beta>0, \forall \ t\in(1,2).
			\end{equation*}
			Indeed, for all $\beta>0$ and $t\in(1,2)$, since $\left|\frac{d}{dt}\left(\langle\beta t\rangle^\alpha\right)\right|\leq \alpha\langle\beta t\rangle^{\alpha-1}\beta\leq \alpha\langle\beta \rangle^{\alpha-1}\beta$, we have
			\begin{align*}
				\frac{\left|\langle\beta t\rangle^\alpha-\langle\beta \rangle^\alpha\right|}{\langle\beta t\rangle^\alpha\langle\beta \rangle^\alpha}\leq \frac{\alpha\langle\beta \rangle^{\alpha-1}\beta|t-1|}{\langle\beta \rangle^\alpha}\leq \alpha|t-1|\leq \alpha|t-1|^\alpha\leq \alpha3^\alpha\frac{|t-1|^\alpha}{|t+1|^\alpha}.
			\end{align*}
			
			\item It follows from \eqref{f1} that
			\[\left\|\langle\beta\rangle^{-2\alpha}\right\|_{C_\beta^\alpha}\lesssim\left\|\langle\beta\rangle^{-\alpha}\right\|_{L^\infty}+
\left\|\langle\beta\rangle^{\alpha-1}\beta\p_\beta\left(\langle\beta\rangle^{-2\alpha}\right)\right\|_{L^\infty}\lesssim
1+\left\|\langle\beta\rangle^{\alpha-1}\beta^2\langle\beta\rangle^{-2\alpha-2}\right\|_{L^\infty}\lesssim1.\]
			Hence, $\|\langle\beta\rangle^{-\alpha} f\|_{C_\beta^\alpha}\lesssim\left\|\langle\beta\rangle^{-2\alpha}\right\|_{C_\beta^\alpha}\left\|f\right\|_{C_\beta^\alpha}\lesssim
\left\|f\right\|_{C_\beta^\alpha}$, and
			\[\| fg\|_{C_\beta^\alpha}\lesssim\|\langle\beta\rangle^{-\alpha}f\|_{C_\beta^\alpha}\| g\|_{C_\beta^\alpha}\lesssim\| f\|_{C_\beta^\alpha}\| g\|_{C_\beta^\alpha}.\]
		\end{enumerate}
		
	This concludes the proof of the lemma.
	\end{proof}

\subsection{Functional spaces for $\psi$ and $\Omega$}
	
In the sequel, we use the convention that if $\mathcal{X}$ and $\mathcal{Y}$ are Banach spaces embedded in the same linear Hausdorff space, then $\mathcal X+\mathcal Y$ is the Banach space with the induced norm
	\[\|u\|_{\mathcal X+\mathcal Y}:=\inf\left\{\|x\|_{\mathcal X}+\|y\|_{\mathcal Y}: u=x+y, x\in\mathcal X, y\in\mathcal Y\right\}.\]
	
	Define
	$\mathcal{G}^0:=C_\beta^\alpha$
	with the norm $\|f\|_{\GG^0}:=\left\| f\right\|_{C_\beta^\alpha}$. We also define
	\[\GG^-:=\GG^0\oplus\CC,\qquad \GG:=\GG^0\oplus C(\TT).\]
	Here $\CC$ denotes the space of constant functions, and we use $\oplus$ to denote the direct sum because $\GG^0\cap\CC=\GG^0\cap C(\TT)=\{0\}$.
		
	Let $\mu>\f12$. We introduce an auxiliary space $X_0$ defined by
	\[X_0:=\left\{H=H(\beta,\phi): \ \beta^{2\mu}H_\vp, \beta^{2\mu-1}H_\phi, \beta^{2\mu-1}H\in \GG^0\right\}\]
	with the norm
	\[\|H\|_{X_0}:=\|\beta^{2\mu}H_\vp\|_{\GG^0}+\|\beta^{2\mu-1}H_\phi\|_{\GG^0}+\|\beta^{2\mu-1}H\|_{\GG^0};\]
	and $$X:=X_0\oplus\langle\beta^{1-2\mu}\rangle,$$
	where $\langle\beta^{1-2\mu}\rangle$ denotes the one dimensional linear space generated by the function $\beta^{1-2\mu}$ with the standard norm.	We stress that $X_0\subset C^1$.

	Now we can define the functional space for $\psi$. Let
	\begin{equation}\label{Y0}Y_0:=\left\{\psi=\psi(\beta,\phi): H=\psi+\frac\beta{2\mu}\psi_\beta\in X_0, \beta^{2\mu-1}\psi\in C_b\big((0,+\infty)\times \TT\big)\right\}\end{equation}
	with the norm
	\[\|\psi\|_{Y_0}:=\left\|H\right\|_{X_0}+\|\beta^{2\mu-1}\psi\|_{L^\infty};\]
	and we define
	\[Y:=Y_0\oplus\langle\beta^{1-2\mu}\rangle.\]
	
The functional space for $\Omega$ is $\WWW:=L^1(\TT)$ and the target space is
	\begin{equation}\label{target_Z}
		\begin{aligned}
			Z:=\Big\{G=G(\beta,\phi):&\  G=\p_\vp F_1+\p_\phi F_2 \text{ weakly for some continuous functions}\\
			&\text{ $F_1$ and $F_2$ such that } \beta^{2\mu-1}F_1\in\GG^-, \beta^{2\mu}F_2\in\GG^0\Big\}
		\end{aligned}
	\end{equation}
	with the norm
	\[\|G\|_Z:=\inf\left(\|\beta^{2\mu-1}F_1\|_{\GG^-}+\|\beta^{2\mu}F_2\|_{\GG^0}\right),\]
	where the infimum is taken over all decompositions of $G$ in \eqref{target_Z}.
	
	\begin{rmk}
		As we mentioned in Section \ref{Sec.sketch}, we apply the implicit function theorem (IFT) to solve $\mathcal F(\psi, \Omega)=0$ for $\Omega$ near $\Omega_0=\gamma$. To this end, we need to construct functional spaces $Y_m, \WWW_m$ and $Z_m$ such that $\LLL:Y_m\times\WWW_m\to Z_m$ is a $C^2$ map in a neighborhood of $(\psi_0, \Omega_0)$. We apply IFT for $Y_m, \WWW_m$ and $Z_m$, rather than the functional spaces $Y, \WWW, Z$ defined in this subsection. In fact,  $Y_m, \WWW_m$ and $Z_m$ consists of all $m$-fold functions in $Y, \WWW$ and $Z$, respectively, with norms adapted slightly, depending on $m$, see the next subsection. We introduce the auxiliary spaces $X_0, Y_0$, $Y, \WWW$ and $Z$ to obtain estimates that are uniform with respect to $m$ for future analysis.	
 \end{rmk}
	
	For further usage, we introduce an equivalent norm on $X_0$. Define the functional space
	\[X_0':=\left\{H=H(\beta,\phi): \ \beta^{2\mu}H_\vp, \beta^{2\mu-1}H_\phi\in \mathcal{G}^0, \langle\beta\rangle^\alpha\beta^{2\mu-1}H\in C_b\big((0,+\infty)\times \TT\big)\right\},\]
	with the norm
	\[\|H\|_{X_0'}:=\|\beta^{2\mu}H_\vp\|_{\mathcal{G}^0}+\|\beta^{2\mu-1}H_\phi\|_{\mathcal{G}^0}+
	\|\langle\beta\rangle^\alpha\beta^{2\mu-1}H\|_{L^\infty}.\]

	\begin{lem}\label{X=tildeX}
		The two functional spaces $X_0$ and $X_0'$ are the same as sets, and their norms are equivalent:
		\[\|H\|_{X_0}\sim \|H\|_{X_0'},\qquad H\in X_0=X_0'.\]
	\end{lem}
	 \begin{proof}
 For $H\in X_0$, by the definition of $C_\beta^\alpha $ norm, we have
  \[\|\langle\beta\rangle^\alpha\beta^{2\mu-1}H\|_{L^\infty}\leq\|\beta^{2\mu-1}H\|_{C_\beta^\alpha}.\]
  This proves $\|H\|_{X_0'}\lesssim\|H\|_{X_0}$.

  Assume that $H\in X_0'$. Recall that $\p_\vp=\p_\phi-\p_\beta$, hence
  \begin{align*}
   \p_\beta(\beta^{2\mu-1}H)&=(2\mu-1)\beta^{-1}\beta^{2\mu-1}H+\beta^{2\mu-1}H_\beta\\
   &=(2\mu-1)\beta^{-1}\beta^{2\mu-1}H+\beta^{2\mu-1}H_\phi-\beta^{-1}\beta^{2\mu}H_\vp.
  \end{align*}
  By the definition of $C_\beta^\alpha $ norm, we have
  \begin{align*}
   \langle\beta\rangle^\alpha\left|\p_\beta(\beta^{2\mu-1}H)\right|&\lesssim \beta^{-1}\|\langle\beta\rangle^\alpha\beta^{2\mu-1}H\|_{L^\infty}+
   \|\langle\beta\rangle^\alpha\beta^{2\mu-1}H_\phi\|_{L^\infty}+
   \beta^{-1}\|\langle\beta\rangle^\alpha\beta^{2\mu}H_\vp\|_{L^\infty}\\
   &\lesssim (1+\beta^{-1})\|H\|_{X_0'},
  \end{align*}
 which implies that $\|\langle\beta\rangle^{\al-1}\beta\p_\beta(\beta^{2\mu-1}H)\|_{L^\infty}\lesssim \|H\|_{X_0'}.$ Now it follows from \eqref{f1} that \\
 $\beta^{2\mu-1}H\in C_{\beta}^\alpha$ and $\|\beta^{2\mu-1}H\|_{C_\beta^\alpha}\lesssim\|H\|_{X_0'}$. This proves $\|H\|_{X_0}\lesssim\|H\|_{X_0'}$.
 \end{proof}

	\subsection{Functional spaces with $m$-fold symmetry}\label{m-fold}
	Given a positive integer $m$, we say that a function $f=f(\beta,\phi): (0,+\infty)\times\TT\to\CC$ is \textit{$m$-fold symmetric} if the following holds for all $(\beta,\phi)\in(0,+\infty)\times\TT$:
	$$f\left(\beta, \phi+\frac{2\pi}m\right)=f\left(\beta,\phi\right).$$
	Note that if $f$ can be expressed in the form of Fourier series with respect to $\theta=\beta+\phi\in\TT$:
	\[f(\beta,\phi)=\sum_{n\in\ZZ}\hat f_n(\beta)e^{\ii n(\beta+\phi)},\qquad \hat f_n(\beta):=\frac1{2\pi}\int_\TT f(\beta, \phi)e^{-\ii n(\beta+\phi)}\,d\phi,\]
	then $f$ is $m$-fold symmetric if and only if $\hat f_n=0$ for all $n$ such that $m\nmid n$, that is to say,
	\[f(\beta,\phi)=\sum_{n\in\ZZ, m\mid n}\hat f_n(\beta)e^{\ii n(\beta+\phi)}.\]
	
	For notational convenience, we define the Fourier projections $P_0$ and $P_{\neq}$ by
	\[(P_0f)(\beta):=\frac1{2\pi}\int_\TT f(\beta,\phi)\,d\phi, \qquad (P_{\neq}f)(\beta,\phi):=f(\beta,\phi)-(P_0f)(\beta).\]
	
	We start with the Banach algebras. Define
	\[\GG^0_m:=\left\{f=f(\beta,\phi)\in\GG^0: f \text{ is $m$-fold symmetric}\right\}\]
	with the norm
	\[\|f\|_{\GG_m^0}=\|P_0f\|_{\GG^0}+m^{1/2}\|P_{\neq}f\|_{\GG^0}.\]
	{{Then $\|f\|_{C_{\beta}^{\alpha}}=\|f\|_{\GG^0}\lesssim \|f\|_{\GG_m^0}$.}}

	Note that $P_0, P_{\neq}$ are bounded linear operators on $\GG^0$, since the Fourier projection is applied to the variable $\phi$. Also note that letting $m=1$ reproduces the space $\GG^0$ defined in the previous subsection. We define
	\[\GG^-_m:= \GG^0_m\oplus \CC,\qquad \GG_m:=\GG^0_m\oplus V_m,\]
	where $V_m=C_m(\TT):=\{g=g(\phi)\in C(\TT): g \text{ is $m$-fold symmetric}\}$ and
	\[\|g\|_{V_m}:=|P_0g|+m^{1/2}\|P_{\neq}g\|_{L^\infty(\TT)}.\]
	Note that if $g=g(\phi)\in C(\TT)$ then $P_0g$ is a complex number and $P_{\neq}g$ is a function in the variable $\phi$, with $\|g\|_{L^\infty(\TT)}\leq|P_0g|+\|P_{\neq}g\|_{L^\infty(\TT)}\lesssim \|g\|_{V_m}$. We also have
	\[\|f\|_{\GG_m}=\|P_0f\|_{\GG}+m^{1/2}\|P_{\neq}f\|_{\GG}.\]
	
	Let's  show the Banach algebra properties of $\GG_m^0$ and $\GG_m$: $\GG_m$ is a unital Banach algebra and $\GG_m^0$ is an ideal of $\GG_m$, i.e.,
	\begin{equation*}
		\GG_m\cdot \GG_m\hookrightarrow \GG_m, \qquad \GG_m^0\cdot\GG_m\hookrightarrow\GG_m^0,
	\end{equation*}
	where the embeddings are uniform in $m\in\NN_+$.
	
	\begin{lem}\label{algebra}
		The following embeddings  are uniform in $m\in\NN_+$:
		\begin{align}
			\GG_m^0\cdot \GG_m^0&\hookrightarrow \GG_m^0, \label{G^0_G^0}\\
			\GG_m^0\cdot V_m & \hookrightarrow \GG_m^0, \label{G^0_V}\\
			V_m\cdot V_m & \hookrightarrow V_m, \label{V_V}\\
			\GG_m^0\cdot\GG_m &\hookrightarrow\GG_m^0, \label{G^0_G}\\
			\GG_m\cdot \GG_m &\hookrightarrow \GG_m. \label{G_G}
		\end{align}
	\end{lem}
	\begin{proof}
		Let $f_1=f_1(\beta,\phi)\in \GG_m^0$ and $f_2=f_2(\beta,\phi)\in \GG_m^0$. Recall that $P_0$ is the Fourier projection applied to the variable $\phi$. By Lemma \ref{AppendixA_Algebra}, we obtain
		\[\| P_0(f_1f_2)\|_{C_\beta^\alpha}\lesssim \| f_1f_2\|_{C_\beta^\alpha}\lesssim\| f_1\|_{C_\beta^\alpha}\| f_2\|_{C_\beta^\alpha}\lesssim \|f_1\|_{\GG_m^0}\|f_2\|_{\GG_m^0}.\]
		As for $P_{\neq}(f_1f_2)$, noting that
		\[P_{\neq}(f_1f_2)=P_0f_1\cdot P_{\neq}f_2+P_0f_2\cdot P_{\neq}f_1+P_{\neq}\left(P_{\neq}f_1\cdot P_{\neq}f_2\right),\]
		we get by Lemma \ref{AppendixA_Algebra} that
		\begin{align*}
			\| P_{\neq}(f_1f_2)\|_{C_\beta^\alpha}&\lesssim \|P_0f_1\|_{C_\beta^\alpha}\| P_{\neq}f_2\|_{C_\beta^\alpha}+\|P_0f_2\|_{C_\beta^\alpha}\| P_{\neq}f_1\|_{C_\beta^\alpha}+\| P_{\neq}f_1\|_{C_\beta^\alpha}\| P_{\neq}f_2\|_{C_\beta^\alpha}\\
			&\lesssim m^{-1/2}\|f_1\|_{\GG_m^0}\|f_2\|_{\GG_m^0}.
		\end{align*}
		Therefore, we have $\|f_1f_2\|_{\GG_m^0}\lesssim \|f_1\|_{\GG_m^0}\|f_2\|_{\GG_m^0}$. This shows \eqref{G^0_G^0}.
		
		The proof of \eqref{G^0_V} is very similar to \eqref{G^0_G^0},
		with the algebraic properties of $C_{\beta}^\alpha$ replaced by the trivial embedding $C_{\beta}^\alpha\cdot C(\TT)\hookrightarrow C_{\beta}^\alpha$.
		The proof of \eqref{V_V} is very similar to  \eqref{G^0_G^0} and we only need to replace the algebraic properties of $C_{\beta}^\alpha$ by the trivial embedding $C(\TT)\cdot C(\TT)\hookrightarrow C(\TT)$. \eqref{G^0_G} follows directly from the definition $\GG_m=\GG^0_m\oplus V_m$ and \eqref{G^0_G^0}, \eqref{G^0_V}.
		
		Now we prove \eqref{G_G}. By \eqref{G^0_G}, we have $\GG_m^0\cdot \GG_m\hookrightarrow \GG_m^0\hookrightarrow \GG_m$. It remains to show that $V_m\cdot \GG_m\hookrightarrow \GG_m$, which follows directly from the definition $\GG_m=\GG^0_m\oplus V_m$ and \eqref{G^0_V}, \eqref{V_V}.
	\end{proof}
	
	Now we define the $m$-fold version of the spaces $X$ and $Y$. Let
	\[X_{0,m}:=\left\{H=H(\beta,\phi)\in X_0: H \text{ is $m$-fold symmetric}\right\}\]
	with the norm
	\[\|H\|_{X_{0,m}}:=\|P_0H\|_{X_0}+m^{1/2}\|P_{\neq}H\|_{X_0},\]
	and we define $X_m:=X_{0,m}\oplus\langle\beta^{1-2\mu}\rangle$. Note that letting $m=1$ we recover the spaces $X_0$ and $X$ defined previously. Let
	\[Y_{0,m}:=\left\{\psi=\psi(\beta,\phi)\in Y_0: \psi \text{ is $m$-fold symmetric}\right\}\]
	with the norm
	\[\|\psi\|_{Y_{0,m}}:=\|P_0\psi\|_{Y_0}+m^{1/2}\|P_{\neq}\psi\|_{Y_0},\]
	and we define $Y_m:=Y_{0,m}\oplus\langle\beta^{1-2\mu}\rangle$. Note that letting $m=1$ we recover the spaces $Y_0$ and $Y$ defined previously.
	
	Finally, the functional space for $m$-fold $\Omega$ is $$\WWW_m:=L^1_m(\TT):=\left\{\Omega=\Omega(\phi)\in L^1(\TT): \Omega \text{ is $m$-fold symmetric}\right\},$$
	with the norm
	\[\|\Omega\|_{\WWW_m}:=|P_0\Omega|+m^{-1/2}\|P_{\neq}\Omega\|_{L^1(\TT)}.\]
	{{Then $\|\Omega\|_{L^1(\TT)}\lesssim m^{1/2}\|\Omega\|_{\WWW_m}$.}} And the target space is
	\begin{equation}\label{target_Z_m}
		\begin{aligned}
			Z_m:&=\Big\{G=G(\beta,\phi): G=\p_\vp F_1+\p_\phi F_2 \text{ weakly for some continuous $m$-fold}\\
			&\qquad\ \ \text{functions $F_1$ and $F_2$ such that } \beta^{2\mu-1}F_1\in\GG_m^-, \beta^{2\mu}F_2\in\GG_m^0\Big\}
		\end{aligned}
	\end{equation}
	with the norm
	\[\|G\|_{Z_m}:=\inf\left(\|\beta^{2\mu-1}F_1\|_{\GG_m^-}+\|\beta^{2\mu}F_2\|_{\GG_m^0}\right),\]
	where the infimum is taken over all decompositions of $G$ in \eqref{target_Z_m}.\smallskip
	
	A remarkable property of these functional spaces is the following proposition.
	
	\begin{prop}\label{prop3.6}
		The following embeddings  are uniform in $m\in\NN_+$:
		\begin{align}
			V_m\cdot \WWW_m&\hookrightarrow\WWW_m,\label{V_W}\\
			\beta^{-2\mu}\GG_m^0\cdot \WWW_m&\hookrightarrow Z_m,\label{G^0_W}\\
			\beta^{-2\mu}\WWW_m&\hookrightarrow Z_m,\label{W_Z}\\
			\beta^{-2\mu}\GG_m\cdot \WWW_m&\hookrightarrow Z_m.\label{G_W}
		\end{align}
	\end{prop}
	To prove Proposition \ref{prop3.6}, we need the following lemma.
	
	\begin{lem}\label{m_antiderivative}
		Let $f=f(\phi)\in L^1(\TT)$ be an $m$-fold symmetric function  such that $\int_\TT f(\phi)\,d\phi=0$. Then there exists $g\in C_m(\TT)$ such that $f=g'$, $\int_\TT g(\phi)\,d\phi=0$ and
		\[\|g\|_{L^\infty(\TT)}\leq\frac Cm\|f\|_{L^1(\TT)},\]
		where $C>0$ is an absolute constant independent of $m$ and $f$.
	\end{lem}
	
\begin{rmk}
{{In fact, $g$ is unique and $f\mapsto g$ is a linear operator.}}
\end{rmk}

	\begin{proof}
		Define
		\[g_0(\phi)=\int_0^\phi f(\Phi)\,d\Phi,\qquad\phi\in\TT.\]
		Since $f\in L^1(\TT)$ and $\int_\TT f(\phi)\,d\phi=0$, it is easy to check that $g_0$ is $2\pi$-periodic, hence $g_0$ is a well-defined continuous function on $\TT$, and moreover, we have $\|g_0\|_{L^\infty(\TT)}\leq \|f\|_{L^1(\TT)}$ and $g'=f$ a.e.. It follows from the $\frac{2\pi}m$-periodic property of $f$ and $\int_\TT f(\phi)\,d\phi=0$ that
		\[\int_\phi^{\phi+\frac{2\pi}m}f(\Phi)\,d\Phi=0,\qquad\phi\in\TT.\]
		As a result, by the definition of $g$, we have
		\[g_0\left(\phi+\frac{2\pi}m\right)=g_0(\phi)+\int_\phi^{\phi+\frac{2\pi}m}f(\Phi)\,d\Phi=g_0(\phi),\qquad\phi\in\TT,\]
		which implies that $g_0$ is also $m$-fold symmetric, hence $g_0\in C_m(\TT)$. Now we show that $|g_0(\phi)|\leq \frac1m \|f\|_{L^1(\TT)}$ for any $\phi\in\TT$. Due to the $m$-fold symmetry of $g_0$, we can assume without loss of generality that $\phi\in\left[0,\frac{2\pi}m\right)$. By the periodicity of $f$, we have
		\[g_0(\phi)=\frac1m\int_0^\phi\sum_{k=0}^{m-1}f\left(\Phi+\frac{2\pi}m\right)\,d\Phi=\frac1m\int_{\bigcup_{k=0}^{m-1}\left[k\frac{2\pi}m, \phi+k\frac{2\pi}m\right]}f(\Phi)\,d\Phi,\]
		therefore,
		\[\left|g_0(\phi)\right|\leq \frac1m\|f\|_{L^1(\TT)},\qquad\forall\ \phi\in\TT.\]
		
		Finally, let $g=P_{\neq}g_0$, then $g$ and $g_0$ differ by only a constant, hence $g\in C_m(\TT)$, $f=g'$ in the sense of distribution and
		\[\|g\|_{L^\infty(\TT)}\leq |P_0g_0|+\|g_0\|_{L^\infty(\TT)}\leq C\|g_0\|_{L^\infty(\TT)}\leq\frac Cm\|f\|_{L^1(\TT)}.\]
		The proof of the lemma is complete.
	\end{proof}

	Now we are ready to prove Proposition \ref{prop3.6}.
	
	\begin{proof}[Proof of Proposition \ref{prop3.6}]
		{\bf Step 1.} {Proof of \eqref{V_W}.} Let $g=g(\phi)\in V_m$ and $\Omega=\Omega(\phi)\in\WWW_m$. Since
		\[P_0(g\Omega)=P_0g\cdot P_0\Omega+P_0(P_{\neq}g\cdot P_{\neq}\Omega),\]
		and $P_0$ is a bounded linear functional on $L^1(\TT)$, we have
		\begin{align*}
			|P_0(g\Omega)|&\lesssim |P_0g||P_0\Omega|+\left\|P_{\neq}g\cdot P_{\neq}\Omega\right\|_{L^1(\TT)}\lesssim |P_0g||P_0\Omega|+\|P_{\neq}g\|_{L^\infty(\TT)}\|P_{\neq}\Omega\|_{L^1(\TT)}\\
			&\lesssim \|g\|_{V_m}\|\Omega\|_{\WWW_m}+m^{-1/2}\|g\|_{V_m}\cdot m^{1/2}\|\Omega\|_{\WWW_m}\lesssim\|g\|_{V_m}\|\Omega\|_{\WWW_m}.
		\end{align*}
		For $P_{\neq}(g\Omega)$, we have
		\begin{align*}
			\|P_{\neq}(g\Omega)\|_{L^1(\TT)}&\lesssim \|g\Omega\|_{L^1(\TT)}+|P_0(g\Omega)|\lesssim \|g\|_{L^\infty(\TT)}\|\Omega\|_{L^1(\TT)}+\|g\|_{V_m}\|\Omega\|_{\WWW_m}\\
			&\lesssim m^{1/2}\|g\|_{V_m}\|\Omega\|_{\WWW_m}.
		\end{align*}
		Hence, $g\Omega\in\WWW_m$ and $\|g\Omega\|_{\WWW_m}\lesssim\|g\|_{V_m}\|\Omega\|_{\WWW_m}.$\smallskip
		
		\noindent{\bf Step 2.} {Proof of \eqref{G^0_W}.} Let $f=f(\beta,\phi)\in\GG_m^0$ and $\Omega=\Omega(\phi)\in\WWW_m$. Denote
		\[f_0(\beta):=P_0(f\Omega)(\beta)=P_0f(\beta)\cdot P_0\Omega+P_0(P_{\neq}f\cdot P_{\neq}\Omega)(\beta).\]
		Then we have $ f_0\in C_{\beta}^\alpha$ and
		\begin{align*}
			\| f_0\|_{C_\beta^\alpha}\lesssim\| P_0f\|_{C_\beta^\alpha}|P_0\Omega|+\| P_{\neq}f\|_{C_\beta^\alpha}\|P_{\neq}\Omega\|_{L^1(\TT)}\lesssim \|f\|_{\GG_m^0}\|\Omega\|_{\WWW_m}.
		\end{align*}
		By Lemma \ref{m_antiderivative}, there exists an $m$-fold function $f_1=f_1(\beta,\phi)$ such that ${{P_0f_1=0}}$, $P_{\neq}(f\Omega)=\p_\phi f_1$ and
		\begin{align*}&\sup_{\phi\in\TT}|f_1(\beta,\phi)|\lesssim\frac1m \int_{\TT}|P_{\neq}(f\Omega)|(\beta,\Phi)d\Phi\lesssim
			\frac1m \int_{\TT}|f\Omega|(\beta,\Phi)d\Phi\leq\frac1m\sup_{\Phi\in \TT}|f(\beta,\Phi)|\|\Omega\|_{L^1(\TT)},\\
			&\sup_{\phi\in\TT}|f_1(\beta_1,\phi)-f_1(\beta_2,\phi)|\lesssim\frac1m\int_{\TT}|P_{\neq}(f\Omega)(\beta_1,\Phi)-P_{\neq}(f\Omega)(\beta_2,\Phi)|d\Phi
			\\ &\quad\lesssim\frac1m\int_{\TT}|(f\Omega)(\beta_1,\Phi)-(f\Omega)(\beta_2,\Phi)|d\Phi\leq \frac1m\sup_{\Phi\in \TT}|f(\beta_1,\Phi)-f(\beta_2,\Phi)|\|\Omega\|_{L^1(\TT)}.\end{align*}
		Then by the definition of the $C_{\beta}^\alpha $ norm, we have $ f_1\in C_{\beta}^\alpha$ with
		\[\| f_1\|_{C_\beta^\alpha}\lesssim\frac1m\| f\|_{C_\beta^\alpha}\|\Omega\|_{L^1(\TT)}\lesssim m^{-1/2}\|f\|_{\GG_m^0}\|\Omega\|_{\WWW_m}. \]
		Hence, we have
		\begin{align*}
			\beta^{-2\mu}f(\beta,\phi)\Omega(\phi)&=\beta^{-2\mu}f_0(\beta)+\p_\phi(\beta^{-2\mu}f_1(\beta,\phi))\\
			&=\p_\vp\left(\int_\beta^\infty s^{-2\mu}f_0(s)\,ds\right)+\p_\phi(\beta^{-2\mu}f_1(\beta,\phi))\\
			&=:\p_\vp F_1+\p_\phi F_2,
		\end{align*}
		where $F_1=P_0F_1$ and $F_2=P_{\neq }F_2$. For $F_2$, there holds
		\[\|\beta^{2\mu}F_2\|_{\GG_m^0}=m^{1/2}\|\beta^{2\mu}F_2\|_{C_\beta^\alpha}= m^{1/2}\| f_1\|_{C_\beta^\alpha}\lesssim\|f\|_{\GG_m^0}\|\Omega\|_{\WWW_m}. \]
		For $F_1=F_1(\beta)=\int_\beta^\infty s^{-2\mu}f_0(s)\,ds$, we have
		\[\left|\langle \beta\rangle^{\alpha}\beta^{2\mu-1}F_1(\beta)\right|\lesssim\langle \beta\rangle^{\alpha}\beta^{2\mu-1}\left(\int_\beta^\infty \langle s\rangle^{-\alpha}s^{-2\mu}\,ds\right)\|\langle \beta\rangle^{\alpha} f_0\|_{L^\infty}\lesssim \|\langle \beta\rangle^{\alpha} f_0\|_{L^\infty},\]
		and
		\begin{align*}
			\left|\p_\beta\left(\beta^{2\mu-1}F_1\right)\right|&\lesssim \beta^{2\mu-2}\left(\int_\beta^\infty s^{-2\mu}\,ds\right)\|f_0\|_{L^\infty}
			+\beta^{-1}\|f_0\|_{L^\infty}\lesssim\beta^{-1}\|f_0\|_{L^\infty}
		\end{align*}
		Then by \eqref{f1} and $\langle \beta\rangle^{\alpha}\geq 1\geq\langle \beta\rangle^{\alpha-1}$, we have $\beta^{2\mu-1}F_1\in C_{\beta}^\alpha$ and
		\[\|\beta^{2\mu-1}F_1\|_{\GG_m^-}=\|\beta^{2\mu-1}F_1\|_{C_\beta^\alpha}\lesssim  \|\langle \beta\rangle^{\alpha} f_0\|_{L^\infty}\leq\|f_0\|_{C_\beta^\alpha}\lesssim\|f\|_{\GG_m^0}\|\Omega\|_{\WWW_m}.\]
		Therefore, $\beta^{-2\mu}f\Omega\in Z_m$ and
		\[\|\beta^{-2\mu}f\Omega\|_{Z_m}\lesssim \|\beta^{2\mu-1}F_1\|_{\GG_m^-}+\|\beta^{2\mu}F_2\|_{\GG_m^0}\lesssim\|f\|_{\GG_m^0}\|\Omega\|_{\WWW_m}.\]
		
		\noindent{\bf Step 3.} {Proof of \eqref{W_Z}.} Let $\Omega=\Omega(\phi)\in \WWW_m$. By Lemma \ref{m_antiderivative}, we can find $g, h\in C_m(\TT)$ such that $P_0g=P_0h=0$, $P_{\neq}\Omega=\p_\phi g=\p_\phi^2h$ and
		\[\|g\|_{L^\infty(\TT)}\lesssim\frac1m\|P_{\neq}\Omega\|_{L^1(\TT)},\qquad \|h\|_{L^\infty(\TT)}\lesssim\frac1{m^2}\|P_{\neq}\Omega\|_{L^1(\TT)}.\]
		We introduce a smooth bump function
		\begin{equation}\label{bump_function}
			\rho\in C^\infty([0,\infty);[0,1])\qquad \text{ such that } \qquad \rho(\beta)=\begin{cases}
				0 & \beta\in[0,1],\\
				1 & \beta\geq 2.
			\end{cases}
		\end{equation}
		Due to $\p_\vp=\p_\phi-\p_\beta$, we obtain
		\begin{align*}
			\beta^{-2\mu}\Omega(\phi)&=\beta^{-2\mu}P_0\Omega+\beta^{-2\mu}\p_\phi g=P_0\Omega\p_\vp\left(\frac1{2\mu-1}\beta^{1-2\mu}\right)+\p_\phi\left(\beta^{-2\mu}g(\phi)\right)\\
			&=\p_\vp\left(\frac{P_0\Omega}{2\mu-1}\beta^{1-2\mu}\right)+\p_\phi\Big((1-\rho(\beta))\beta^{-2\mu}g(\phi)\Big)+
			\left(\p_\vp+\p_\beta\right)\Big(\rho(\beta)\beta^{-2\mu}g(\phi)\Big)\\
			&=\p_\vp\left(\frac{P_0\Omega}{2\mu-1}\beta^{1-2\mu}+\rho(\beta)\beta^{-2\mu}g(\phi)\right)\\
			&\qquad\qquad+\p_\phi\Big((1-\rho(\beta))\beta^{-2\mu}g(\phi)+\p_\beta(\rho(\beta)\beta^{-2\mu})h(\phi)\Big)\\
			&=:\p_\vp G_1+\p_\phi G_2.
		\end{align*}
		We also have the following estimates: for $G_1$,
		\begin{align*}
			\|\beta^{2\mu-1}G_1\|_{\GG_m^-}&=\|\beta^{2\mu-1}P_0G_1\|_{\GG^-}+m^{1/2}\|\beta^{2\mu-1}P_{\neq}{G_1}\|_{\GG^-}\\
			&=\left|\frac{P_0\Omega}{2\mu-1}\right|+m^{1/2}\left\|\frac{\rho(\beta)}{\beta}g\right\|_{\GG^0}\lesssim |P_0\Omega|+m^{1/2}\|g\|_{L^\infty(\TT)}\\
			&\lesssim |P_0\Omega|+m^{-1/2}\|P_{\neq}\Omega\|_{L^1(\TT)}\lesssim\|\Omega\|_{\WWW_m},
		\end{align*}
		and for ${G_2}$, noting that $P_0{G_2}=0$,
		\begin{align*}
			\|\beta^{2\mu}{G_2}\|_{\GG_m^0}&=m^{1/2}\|\beta^{2\mu}{G_2}\|_{\GG^0}\lesssim m^{1/2}\left(\|g\|_{L^\infty(\TT)}+\|h\|_{L^\infty(\TT)}\right)\\
			&\lesssim m^{-1/2}\|P_{\neq}\Omega\|_{L^1(\TT)}\lesssim\|\Omega\|_{\WWW_m}.
		\end{align*}
		Therefore, $\beta^{-2\mu}\Omega\in Z_m$ and $\|\beta^{-2\mu}\Omega\|_{Z_m}\lesssim\|\Omega\|_{\WWW_m}.$\smallskip
		
		Finally, {\eqref{G_W}}  follows directly from the definition $\GG_m=\GG_m^0\oplus V_m$ and \eqref{V_W}$\sim$\eqref{W_Z}.
	\end{proof}

	\section{Nonlinear problem in the new coordinates}\label{sec_main_thm}
	Let $m\geq 1$ be a positive integer. Given $\psi\in Y_m$ and $\Omega\in\WWW_m$, recall from \eqref{Eq.F_expression} that
	\begin{equation}\label{nonlin_map_main}
		\FFF(\psi,\Omega)=\p_\vp \NNN_1(\psi)+\p_\phi\NNN_2(\psi)+\NNN_3(\psi)\Omega,
	\end{equation}
	where
	\begin{equation}\label{Eq.N_123}
		\NNN_1(\psi)=\frac{2\psi_\beta\psi_\vp}{\psi_{\beta\vp}}-\frac{\psi_{\beta\phi}}{\psi_{\beta\vp}}\NNN_2(\psi),\ \  \NNN_2(\psi)=\frac{\psi_{\beta\vp}\psi_\phi-\psi_{\beta\phi}\psi_\vp}{2\psi_\beta},\ \ \NNN_3(\psi)=\frac{\psi_{\beta\vp}\psi_\vp^{-\frac1{2\mu}}}{2\mu}.
	\end{equation}

	By \eqref{special_solution} we have $\FF(\psi_0,\Omega_0)=0$.  Our goal is to solve the
	equation $\FFF(\psi,\Omega)=0$  for $\Omega$ near $\Omega_0=\gamma$. As we mentioned before, a natural idea is to use the implicit function theorem (see Theorem 1.2.1 in \cite{Chang} for example). For a Banach space $\mathcal{X}$, $x_0\in\mathcal{X}$ and a positive number $\delta>0$, we denote by $B_\delta^{(\mathcal{X})}(x_0)$ the ball in $\mathcal{X}$ of radius $\delta$ with the center $x_0$. Given $\mu>\frac12$, we define
	\begin{equation}\label{alpha_mu}
		\alpha_\mu:=\sqrt{4\mu^2-2\mu+1}-(2\mu-1)\in\left(\frac 12, 1\right).
	\end{equation}
	
	In this section, the implicit constants in all $\lesssim$ depend only on $\alpha$ and $\mu$. In particular, the implicit constants are independent of $m\in\NN_+$. The main result of this section is
	
	\begin{thm}\label{IFT}
		Assume that $m\geq 2$, $\mu>\frac12$ and $\alpha\in\left(0,\alpha_\mu\right)$. There exist $\varepsilon_\psi, \varepsilon_\Omega>0$ independent of $m\geq 2$ and a unique $C^2$ map $\Xi: B_{\varepsilon_\Omega}^{(\WWW_m)}(\Omega_0)\to B_{\varepsilon_\psi}^{(Y_m)}(\psi_0)$
		such that
		\[\FFF\big(\Xi(\Omega),\Omega)=0.\]
		Moreover, $\Xi(\Omega)$ is real-valued if $\Omega$ is real-valued.
	\end{thm}
	
	\begin{rmk}
		In fact,  we can show that $\FF$ is $C^\infty$, therefore $\Xi: B_{\varepsilon_\Omega}^{(\WWW_m)}(\Omega_0)\to B_{\varepsilon_\psi}^{(Y_m)}(\psi_0)$ is also $C^\infty$. This can be seen from the proof of Proposition \ref{nonlinear_part} below.
	\end{rmk}

	\subsection{$C^2$ regularity of nonlinear map}\label{sec_nonlin}
	The main result of this subsection is as follows.
	
	\begin{prop}\label{nonlinear_part}
		Assume that $\alpha\in(0,1)$, $m\in\NN_+$ and $\mu>\frac12$. There exists a small constant $\delta>0$ independent of $m\in\NN_+$, such that
		\[ \FF\in C^2\Big(B_{\delta}^{(Y_m)}(\psi_0)\times\WWW_m; Z_m\Big),\]
		and the $C^2$ norm of $\FF$ is independent of $m\in\NN_+$.
	\end{prop}
	
Thanks to the expression \eqref{nonlin_map_main} and the Banach algebraic properties established in the previous section, it suffices to prove that all the linear operators appearing in \eqref{Eq.N_123} map into the corresponding Banach algebras continuously. This follows from Lemma \ref{lem_nonlinear} below.
	
	Throughout this subsection, we assume that $\alpha\in(0,1)$, $m\in\NN_+$ and $\mu>\frac12$.
	
	\begin{lem}\label{F<H}
		Let $H\in X_0$ and
		\begin{equation}\label{H_to_F}
			T_0H(\beta,\phi)=2\mu\beta^{-2\mu}\int_0^\beta s^{2\mu-1}H(s,\phi)\,ds, \qquad (\beta,\phi)\in(0,+\infty)\times\TT.
		\end{equation}
		Then $F=T_0H$ is the unique function in $Y_0$ such that $H=F+\frac{\beta}{2\mu}F_\beta$ and
		\begin{align*}
			&\ \|F\|_{Y_0}+\|\beta^{2\mu}F_{\beta\phi}\|_{C_\beta^\alpha}+\|\beta^{2\mu}F_{\beta}\|_{C_\beta^\alpha}+\|\beta^{2\mu-1}F_{\phi}\|_{C_\beta^\alpha}
			+\|\beta^{2\mu-1}F\|_{C_\beta^\alpha}
			\lesssim\|H\|_{X_0}.
		\end{align*}
	\end{lem}
	
Then $T_0: X_0\to Y_0$ is a bounded linear operator.
	
	\begin{proof}
		Clearly, $F=T_0H$ is well-defined, i.e., the integral in \eqref{H_to_F} is absolutely convergent due to $s^{2\mu-1}H(s,\phi)\in \mathcal{G}^0$. Direct computation gives the identity $H=F+\frac{\beta}{2\mu}F_\beta$.  For simplicity, we omit the variable $\phi$ here in this proof.  Moreover, all implicit constants in $\lesssim$ are independent of $\phi$ in the proof. \smallskip
		
		\begin{enumerate}[(1)]
			\item \underline{$\|F\|_{Y_0}+\|\beta^{2\mu-1}F\|_{C_\beta^\alpha}\lesssim\|H\|_{X_0}$.}
			For any $\beta>0$, we have $$|\langle\beta\rangle^{\alpha}\beta^{2\mu-1}F(\beta)|\lesssim \langle\beta\rangle^{\alpha}\beta^{-1}
			\left(\int_0^\beta \langle s\rangle^{-\alpha}\,ds\right)\|\langle\beta\rangle^{\alpha}\beta^{2\mu-1}H\|_{L^\infty}\lesssim
			\|\langle\beta\rangle^{\alpha}\beta^{2\mu-1}H\|_{L^\infty},$$
			and 
			\begin{align*}
				\left|\p_\beta\left(\beta^{2\mu-1}F\right)(\beta)\right|&\lesssim\beta^{-2}\int_0^\beta s^{2\mu-1}|H(s)|\,ds+\beta^{-1}|\beta^{2\mu-1}H(\beta)|\\
				&\lesssim \beta^{-1} \|\beta^{2\mu-1}H\|_{L^\infty}\lesssim \beta^{-1}\|\langle\beta\rangle^{\alpha}\beta^{2\mu-1}H\|_{L^\infty}.
			\end{align*}
			Therefore, by \eqref{f1}, we have $\beta^{2\mu-1}F\in C_{\beta}^\alpha$ and
			\beno
			\|\beta^{2\mu-1}F\|_{C_\beta^\alpha}\lesssim
			\| \langle\beta\rangle^{\alpha}\beta^{2\mu-1}H\|_{L^\infty}\lesssim\| \beta^{2\mu-1}H\|_{C_\beta^\alpha}\lesssim\|H\|_{X_0}.
			\eeno
			We also have
			\[\|F\|_{Y_0}=\|H\|_{X_0}+\|\beta^{2\mu-1}F\|_{L^\infty}\lesssim\|H\|_{X_0}+\|\beta^{2\mu-1}F\|_{C_\beta^\alpha}\lesssim\|H\|_{X_0}.\]
			
			\item \underline{$\|\beta^{2\mu-1}F_\phi\|_{C_\beta^\alpha}\lesssim\|H\|_{X_0}$.} The same argument as in $(1)$ gives that $\beta^{2\mu-1}F_\phi\in C_{\beta}^\alpha$ and  $\|\beta^{2\mu-1}F_\phi\|_{C_\beta^\alpha}\lesssim\|\beta^{2\mu-1}H_\phi\|_{C_\beta^\alpha}\lesssim\|H\|_{X_0}$
			{(as $F=T_0H\Rightarrow F_\phi=T_0H_\phi$)}.
			
			\item \underline{$\|\beta^{2\mu}F_\beta\|_{C_\beta^\alpha}\lesssim\|H\|_{X_0}$.} The identity $H=F+\frac{\beta}{2\mu}F_\beta$ implies $F_\beta=2\mu\beta^{-1}H-2\mu\beta^{-1}F$, hence $\beta^{2\mu}F_\beta=2\mu \beta^{2\mu-1}H-2\mu \beta^{2\mu-1}F\in C_{\beta}^\alpha$ and $$\|\beta^{2\mu}F_\beta\|_{C_\beta^\alpha}\lesssim \|\beta^{2\mu-1}H\|_{C_\beta^\alpha}+\|\beta^{2\mu-1}F\|_{C_\beta^\alpha}\lesssim\|H\|_{X_0}.$$
			\item \underline{$\|\beta^{2\mu}F_{\beta\phi}\|_{C_\beta^\alpha}\lesssim\|H\|_{X_0}$.} The same argument as in $(3)$ gives that $\beta^{2\mu}F_{\beta\phi}\in C_{\beta}^\alpha$ and
			\beno
			\|\beta^{2\mu}F_{\beta\phi}\|_{C_\beta^\alpha}\lesssim \|\beta^{2\mu-1}H_\phi\|_{C_\beta^\alpha}+\|\beta^{2\mu-1}F_\phi\|_{C_\beta^\alpha}\lesssim\|H\|_{X_0}.
			\eeno
		\end{enumerate}
		
	The uniqueness follows by noting $\pa_\beta(\beta^{2\mu}F)=2\mu \beta^{2\mu-1}H$ and $\beta^{2\mu}F|_{\beta=0}=0$.
	\end{proof}

\begin{lem} It holds that (for $F,\partial_{\beta}F\in C((0,\infty)\times\mathbb{T})$)
		\begin{align}\label{f3}
			\|F\|_{\GG}\lesssim\|\langle\beta\rangle^{\alpha-1} F\|_{L^\infty}+\|\langle\beta\rangle^{\alpha}\beta\partial_{\beta}F\|_{L^\infty}.
		\end{align}	
	\end{lem}
	
	\begin{proof}
		We define
		\begin{equation}\label{fdef}
			\displaystyle f(\phi):= F(1,\phi)+\int_1^\infty F_\beta(s,\phi)\,ds,\qquad \forall\ \phi\in\TT.
		\end{equation}	
		Note that $|F(1,\phi)|\lesssim\|\langle\beta\rangle^{\alpha-1} F\|_{L^\infty} $ and
		\begin{align*}
			\int_1^\infty \left|F_\beta(s,\phi)\right|\,ds&\lesssim
			\|\langle\beta\rangle^{\alpha}\beta\partial_{\beta}F\|_{L^\infty}\int_1^\infty\frac1{s^{1+\alpha}}\,ds
			\lesssim\|\langle\beta\rangle^{\alpha}\beta\partial_{\beta}F\|_{L^\infty}<+\infty.
		\end{align*}
		Thus, $f$ is well-defined and
		$$\|f\|_{L^\infty(\TT)}\lesssim\|\langle\beta\rangle^{\alpha-1} F\|_{L^\infty}+\|\langle\beta\rangle^{\alpha}\beta\partial_{\beta}F\|_{L^\infty}.$$
		Let $F_1(\beta,\phi)=F(\beta,\phi)-f(\phi)$. For $0< \beta\leq 1$ we have\begin{align*}
			\langle\beta\rangle^{\alpha-1}|F_1(\beta,\phi)|&\lesssim|F_1(\beta,\phi)|\leq|F(\beta,\phi)|+|f(\phi)|\lesssim\|\langle\beta\rangle^{\alpha-1} F\|_{L^\infty}+\|f\|_{L^\infty(\TT)}\\&\lesssim\|\langle\beta\rangle^{\alpha-1} F\|_{L^\infty}+\|\langle\beta\rangle^{\alpha}\beta\partial_{\beta}F\|_{L^\infty}.
		\end{align*}
		For $\beta\geq 1$ we have
		\begin{align*}
			F(\beta,\phi)= &F(1,\phi)+\int_1^{\beta} F_\beta(s,\phi)=f(\phi)-\int_{\beta}^\infty F_\beta(s,\phi)\,ds,\\
			F_1(\beta,\phi)=&F(\beta,\phi)-f(\phi)=-\int_{\beta}^\infty F_\beta(s,\phi)\,ds.
		\end{align*}
	Thus, we get
	\begin{align*}
			|F_1(\beta,\phi)|\leq&\int_{\beta}^\infty |F_\beta(s,\phi)|\,ds\lesssim
			\|\langle\beta\rangle^{\alpha}\beta\partial_{\beta}F\|_{L^\infty}\int_{\beta}^\infty\frac1{s^{1+\alpha}}\,ds
			\lesssim \langle\beta\rangle^{-\alpha}\|\langle\beta\rangle^{\alpha}\beta\partial_{\beta}F\|_{L^\infty}.
		\end{align*}		
		This shows that
		\beno
		\|\langle\beta\rangle^{\alpha} F_1\|_{L^\infty}\lesssim\|\langle\beta\rangle^{\alpha-1} F\|_{L^\infty}+\|\langle\beta\rangle^{\alpha}\beta\partial_{\beta}F\|_{L^\infty} .\eeno
We also have  $ \partial_{\beta}F_1=\partial_{\beta}F$ and $ \|\langle\beta\rangle^{\alpha-1}\beta\partial_{\beta}F_1\|_{L^\infty}=\|\langle\beta\rangle^{\alpha-1}\beta\partial_{\beta}F\|_{L^\infty}
		\leq\|\langle\beta\rangle^{\alpha}\beta\partial_{\beta}F\|_{L^\infty}$. Now it follows from \eqref{f1} that $F_1\in C_{\beta}^{\alpha}$ and
		\beno
		\|F_1\|_{C_{\beta}^{\alpha}}\lesssim\|\langle\beta\rangle^{\alpha-1} F\|_{L^\infty}+\|\langle\beta\rangle^{\alpha}\beta\partial_{\beta}F\|_{L^\infty}.
		\eeno
		
		Recalling that $ \GG:=\GG^0\oplus C(\TT)$, $\mathcal{G}^0:=C_\beta^\alpha$, $F=F_1+f$, we infer that
		\begin{align*}
			\|F\|_{\GG}=\|F_1\|_{C_{\beta}^{\alpha}}+\|f\|_{L^\infty(\TT)}\lesssim\|\langle\beta\rangle^{\alpha-1} F\|_{L^\infty}+\|\langle\beta\rangle^{\alpha}\beta\partial_{\beta}F\|_{L^\infty}.
		\end{align*}This completes the proof.
	\end{proof}	
	\begin{lem}\label{f_phi}
		Let $H\in X_0$ and $F=T_0 H$ be given by \eqref{H_to_F}. Then $\left\|\beta^{2\mu}F_{\vp}\right\|_{\GG}\lesssim\|H\|_{X_0}.$
	\end{lem}
	
	\begin{proof}
		Lemma \ref{F<H} shows that $F\in Y_0$ and $\|F\|_{Y_0}\lesssim\|H\|_{X_0}$. It follows from \eqref{f3} that\begin{align}\label{f4}
			\|\beta^{2\mu}F_\vp\|_{\GG}\lesssim\|\langle\beta\rangle^{\alpha-1}\beta^{2\mu}F_\vp\|_{L^\infty}+
			\|\langle\beta\rangle^{\alpha}\beta\partial_{\beta}\left(\beta^{2\mu}F_\vp\right)\|_{L^\infty}.
		\end{align}
		Applying $\p_\vp$ to both sides of $H=F+\frac{\beta}{2\mu}F_\beta$ gives that $H_\vp=F_\vp-\frac1{2\mu}F_\beta+\frac\beta{2\mu}F_{\beta\vp}$, i.e., $\p_\beta\left(\beta^{2\mu}F_\vp\right)=\beta^{2\mu-1}(2\mu H_\vp+F_\beta)$, then we get by Lemma \ref{F<H} that\begin{align}\label{f5}
			&\|\langle\beta\rangle^{\alpha}\beta\partial_{\beta}\left(\beta^{2\mu}F_\vp\right)\|_{L^\infty}=
			\|\langle\beta\rangle^{\alpha}\beta^{2\mu}(2\mu H_\vp+F_\beta)\|_{L^\infty}\\ \notag&\lesssim\|\langle\beta\rangle^{\alpha}\beta^{2\mu}H_\vp\|_{L^\infty}+\|\langle\beta\rangle^{\alpha}\beta^{2\mu}F_\beta\|_{L^\infty}
			\lesssim\|\beta^{2\mu}H_\vp\|_{C_{\beta}^{\alpha}}+\|\beta^{2\mu}F_\beta\|_{C_{\beta}^{\alpha}}\lesssim\|H\|_{X_0}.
		\end{align}
		Since $F_\vp=F_\phi-F_\beta$, we get by Lemma \ref{F<H} that
		\begin{align}\label{f6}
			\|\langle\beta\rangle^{\alpha-1}\beta^{2\mu}F_\vp\|_{L^\infty}\lesssim& \|\langle\beta\rangle^{\alpha}\beta^{2\mu-1}F_\phi\|_{L^{\infty}}+\|\langle\beta\rangle^{\alpha}\beta^{2\mu}F_\beta\|_{L^{\infty}}\\
			\notag{\lesssim}& \|\beta^{2\mu-1}F_\phi\|_{C_\beta^\alpha}+\|\beta^{2\mu}F_\beta\|_{C_\beta^\alpha}\lesssim\|H\|_{X_0}.
		\end{align}Here we used $\langle\beta\rangle^{\alpha-1}\beta^{2\mu}\lesssim\min(\langle\beta\rangle^{\alpha}\beta^{2\mu-1},\langle\beta\rangle^{\alpha}\beta^{2\mu}) $.
		Now the result follows from \eqref{f4}, \eqref{f5}, \eqref{f6}.
	\end{proof}

		\begin{lem}\label{lem_nonlinear}
		If $\psi\in Y_m$, then we have
		\begin{align*}
			&\ \|\beta^{2\mu-1}\psi\|_{\GG_m^-}+\|\beta^{2\mu}\psi_\beta\|_{\GG_m^-}+\|\beta^{2\mu-1}\psi_\phi\|_{\GG_m^0}+\|\beta^{2\mu}\psi_{\beta\phi}\|_{\GG_m^0}\\
			&\qquad+\left\|\beta^{2\mu}\left(\psi_\vp+\frac\beta{2\mu}\psi_{\beta\vp}\right)\right\|_{\GG_m^0}+\|\beta^{2\mu}\psi_\vp\|_{\GG_m}+
			\|\beta^{2\mu+1}\psi_{\beta\vp}\|_{\GG_m}\lesssim \|\psi\|_{Y_m},
		\end{align*}
		where the implicit constant is independent of $m\in\NN_+$.
	\end{lem}
	\begin{proof}
		For $\psi\in Y_m$, we decompose $\psi$ as $\psi=F+c_0\beta^{1-2\mu}$, where $F\in Y_{0,m}$ and $c_0\in\CC$ by recalling that $Y_m$ is the direct sum of $Y_{0,m}$ and $\langle\beta^{1-2\mu}\rangle$. Since $F\in Y_{0,m}\subset Y_0$, the function $H:=F+\frac{\beta}{2\mu}F_\beta$ belongs to $X_0$ by the definition of the space $Y_0$ and hence lies in $X_{0,m}$, because $H$ is $m$-fold symmetric. Also, \eqref{H_to_F} holds (i.e., $F=T_0H$). Hence, $P_0F=T_0P_0H$, $P_{\neq}F=T_0P_{\neq}H$.
		
		It follows from Lemma \ref{F<H} that
		\begin{align*}
			\|\beta^{2\mu-1}\psi\|_{\GG_m^-}&\lesssim \|\beta^{2\mu-1}F\|_{\GG_m^0}+|c_0|\lesssim\|\beta^{2\mu-1}P_0F\|_{\GG^0}+m^{1/2}\|\beta^{2\mu-1}P_{\neq}F\|_{\GG^0}+|c_0|\\
			&\lesssim \|P_0H\|_{X_0}+m^{1/2}\|P_{\neq}H\|_{X_0}+|c_0|\lesssim\|H\|_{X_{0,m}}+|c_0|\lesssim\|\psi\|_{Y_m};
		\end{align*}
		and similarly one has
		 \[\|\beta^{2\mu}H_\vp\|_{\GG_m^0}+\|\beta^{2\mu}F_\beta\|_{\GG_m^0}+\|\beta^{2\mu-1}F_\phi\|_{\GG_m^0}+\|\beta^{2\mu}F_{\beta\phi}\|_{\GG_m^0}\lesssim\|H\|_{X_{0,m}}\lesssim\|\psi\|_{Y_m},\]
		\[\|\beta^{2\mu}\psi_\beta\|_{\GG_m^-}+\|\beta^{2\mu-1}\psi_\phi\|_{\GG_m^0}+\|\beta^{2\mu}\psi_{\beta\phi}\|_{\GG_m^0}\lesssim\|\psi\|_{Y_m}.\]
		
		It follows from Lemma \ref{f_phi}, $P_0F=T_0P_0H$, $P_{\neq}F=T_0P_{\neq}H$ that
		\begin{equation}\label{Eq.5.10}
			\begin{aligned}
				\left\|\beta^{2\mu}F_\vp\right\|_{\GG_m}&=\left\|\beta^{2\mu}\p_\beta(P_0F)\right\|_{\GG}+m^{1/2}\left\|\beta^{2\mu}\p_\vp (P_{\neq}F)\right\|_{\GG}\\
				&\lesssim \|P_0H\|_{X_0}+m^{1/2}\|P_{\neq}H\|_{X_0}\lesssim \|H\|_{X_{0,m}}\lesssim \|\psi\|_{Y_m}.
			\end{aligned}
		\end{equation}
		Now, note that $\beta^{2\mu}\psi_\vp=\beta^{2\mu}F_\vp+c_0(2\mu-1)$, hence by \eqref{Eq.5.10},
		\begin{align*}
			\left\|\beta^{2\mu}\psi_\vp\right\|_{\GG_m}&\leq\left\|\beta^{2\mu}F_\vp\right\|_{\GG_m}+|c_0(2\mu-1)|\lesssim\|\psi\|_{Y_m}+|c_0|\lesssim \|\psi\|_{Y_m}.
		\end{align*}Applying $\p_\vp$ to both sides of $H=F+\frac{\beta}{2\mu}F_\beta$ gives that $H_\vp=F_\vp-\frac1{2\mu}F_\beta+\frac\beta{2\mu}F_{\beta\vp}$, then\begin{align*}
			\left\|\beta^{2\mu}\left(F_\vp+\frac\beta{2\mu}F_{\beta\vp}\right)\right\|_{\GG_m^0}&=
			\left\|\beta^{2\mu}\left(H_\vp+\frac1{2\mu}F_{\beta}\right)\right\|_{\GG_m^0}\lesssim \|\beta^{2\mu}H_\vp\|_{\GG_m^0}+\|\beta^{2\mu}F_\beta\|_{\GG_m^0}\lesssim \|\psi\|_{Y_m}.
		\end{align*}
		Now, note that $\psi_\vp+\frac\beta{2\mu}\psi_{\beta\vp}=F_\vp+\frac\beta{2\mu}F_{\beta\vp}$ and $ \GG_m^0\subset\GG_m$, then\begin{align*}
			&\left\|\beta^{2\mu}\left(\psi_\vp+\frac\beta{2\mu}\psi_{\beta\vp}\right)\right\|_{\GG_m}
			 \leq\left\|\beta^{2\mu}\left(\psi_\vp+\frac\beta{2\mu}\psi_{\beta\vp}\right)\right\|_{\GG_m^0}=\left\|\beta^{2\mu}\left(F_\vp+\frac\beta{2\mu}F_{\beta\vp}\right)\right\|_{\GG_m^0}\lesssim \|\psi\|_{Y_m},\\
			&\|\beta^{2\mu+1}\psi_{\beta\vp}\|_{\GG_m}\leq 2\mu\left\|\beta^{2\mu}\left(\psi_\vp+\frac\beta{2\mu}\psi_{\beta\vp}\right)\right\|_{\GG_m}+2\mu\left\|\beta^{2\mu}\psi_\vp\right\|_{\GG_m}\lesssim \|\psi\|_{Y_m}.
		\end{align*}This completes the proof.
	\end{proof}

	Now we are in a position to prove Proposition \ref{nonlinear_part}.

	\begin{proof}[Proof of Proposition \ref{nonlinear_part}]
	The facts that $ \GG_m^0\subset\GG_m^-\subset\GG_m$
	and $\GG_m$ is a unital Banach algebra with the embedding norm independent of $m\in\NN_+$(by  Lemma \ref{algebra}) will be used repeatedly.
	\begin{enumerate}[(1)]
			\item \underline{$\NNN_3(\psi)\Omega:B_{\delta}^{(Y_m)}(\psi_0)\times\WWW_m\longrightarrow Z_m$ is $C^2$.} Taylor expansion of $x\mapsto x^{-\frac1{2\mu}}$ around $x=1$ gives an analytic function from $U$ to $\GG_m$ on a neighborhood $U\subset \GG_m$ of $1=\beta^{2\mu}\p_\vp\psi_0$, where $U$ is independent of $m\in\NN_+$ since $\|f\|_{L^\infty}\lesssim\|f\|_{\GG_m}$. Lemma \ref{lem_nonlinear} shows that $\|\beta^{2\mu}\psi_\vp\|_{\GG_m}\lesssim\|\psi\|_{Y_m}$, hence for $\delta>0$ small enough, we have
			\[\psi\mapsto \left(\beta^{2\mu}\psi_\vp\right)^{-\frac1{2\mu}}\in C^2\left(B_{\delta}^{(Y_m)}(\psi_0); \GG_m\right),\]
			with $C^2$ norm independent of $m\in\NN_+$. By Lemma \ref{lem_nonlinear}, the linear map $\psi\mapsto \beta^{2\mu+1}\psi_{\beta\vp}$ is in $C^2(Y_m; \GG_m)$. Hence, by $\beta^{-2\mu}\GG_m\cdot \WWW_m\hookrightarrow Z_m$ due to Proposition \ref{prop3.6}, we have
			\[\NNN_3(\psi)\Omega=\frac1{2\mu}\beta^{-2\mu}\cdot \beta^{2\mu+1}\psi_{\beta\vp}\cdot\left(\beta^{2\mu}\psi_\vp\right)^{-\frac1{2\mu}}\cdot\Omega\in C^2\Big(B_{\delta}^{(Y_m)}(\psi_0)\times\WWW_m; Z_m\Big),\]
			with $C^2$ norm independent of $m\in\NN_+$.
			
			\item \underline{$\p_\phi\NNN_2(\psi):B_{\delta}^{(Y_m)}(\psi_0)\longrightarrow Z_m$ is $C^2$.} The argument is similar. We write
			\[\beta^{2\mu}\NNN_2(\psi)=\frac{\beta^{2\mu+1}\psi_{\beta\vp}\cdot \beta^{2\mu-1}\psi_\phi-\beta^{2\mu}\psi_{\beta\phi}\cdot \beta^{2\mu}\psi_\vp}{2\beta^{2\mu}\psi_\beta}.\]
			Expanding $x\mapsto x^{-1}$ around $-1=\beta^{2\mu}\p_\beta\psi_0$, then by using Lemma \ref{lem_nonlinear} and $\GG_m\cdot \GG_m^0\hookrightarrow\GG_m^0$, we obtain
			$\beta^{2\mu}\NNN_2\in C^2\left(B_{\delta}^{(Y_m)}(\psi_0); \GG_m^0\right)$ with $C^2$ norm independent of $m\in\NN_+$. Then the desired result follows from the definition of $Z_m$.
			
			\item \underline{$\p_\vp\NNN_1(\psi):B_{\delta}^{(Y_m)}(\psi_0)\longrightarrow Z_m$ is $C^2$.} We write
			\[\beta^{2\mu-1}\NNN_1(\psi)=\frac{2\beta^{2\mu}\psi_\beta\cdot \beta^{2\mu}\left(\psi_\vp+\frac\beta{2\mu}\psi_{\beta\vp}\right)}{\beta^{2\mu+1}\psi_{\beta\vp}}-\frac1{\mu}\beta^{2\mu}\psi_\beta-
			\frac{\beta^{2\mu}\psi_{\beta\phi}}{\beta^{2\mu+1}\psi_{\beta\vp}}\beta^{2\mu}\NNN_2(\psi).\]
			Expanding  $x\mapsto x^{-1}$ around $-2\mu=\beta^{2\mu+1}\p_{\beta\vp}\psi_0$, then by using Lemma \ref{lem_nonlinear}, $\beta^{2\mu}\NNN_2\in C^2\left(B_{\delta}^{(Y_m)}(\psi_0); \GG_m^0\right)$ and $\GG_m\cdot \GG_m^0\hookrightarrow\GG_m^0$, we obtain $\beta^{2\mu-1}\NNN_1\in C^2\left(B_{\delta}^{(Y_m)}(\psi_0); \GG_m^-\right)$ with $C^2$ norm independent of $m\in\NN_+$. The desired result follows from the definition of $Z_m$.
					\end{enumerate}\smallskip
		
		Summing up, we conclude the proof of Proposition \ref{nonlinear_part}.
	\end{proof}

	\subsection{Proof of Theorem \ref{IFT}}
	
	The proof relies on Proposition \ref{nonlinear_part} and the following key Proposition \ref{linear_part}, which shows the solvability of the linearized problem.
	
	\begin{prop}\label{linear_part}
		Assume that $m\geq 2$, $\mu>\frac12$ and $\alpha\in\left(0,\alpha_\mu\right)$, where $\alpha_\mu$ is given by \eqref{alpha_mu}. The linearized operator $\LLL=\frac{\p\FF}{\p\psi}(\psi_0, \Omega_0):Y_m\to Z_m$ is given by
		\[\LLL(\psi)=\frac1\mu \p_\vp(\beta H_\vp)+\frac{\mu H_{\phi\phi}}{\beta}+\frac{\gamma}{2\mu}\psi_\phi,\]
		where $H=\psi+\frac\beta{2\mu}\psi_\beta$. Moreover, the operator $\LLL:Y_m\to Z_m$is bijective and has a bounded inverse whose norm is independent of $m\geq2$.
	\end{prop}

	The proof of Proposition \ref{linear_part} is rather complicated and will be presented next section.

	\begin{proof}[Proof of Theorem \ref{IFT}]
		By Proposition \ref{nonlinear_part}, $\FFF\in C^1\Big(B_{\delta}^{(Y_m)}(\psi_0)\times\WWW_m; Z_m\Big)$ and \\$\frac{\p\FF}{\p\psi}: B_{\delta}^{(Y_m)}(\psi_0)\times\WWW_m\to L(Y_m, Z_m)$ is continuous uniformly in $m\in\NN_+$, where $L(Y_m, Z_m)$ denotes the set of bounded linear operators from $Y_m$ to $Z_m$. Moreover, ${\FF(\psi_0,\Omega_0)=0}$. By Proposition \ref{linear_part}, $\frac{\p\FF} {\p\psi}(\psi_0,\Omega_0)$ is a linear isomorphism from $Y_m$ to $Z_m$, whose inverse has the norm independent of $m\geq2$. Hence, the proof of implicit function theorem for Banach spaces yields the existence of $\varepsilon_\psi, \varepsilon_\Omega>0$ which are independent of $m\geq 2$ and the $C^1$ regularity of $\Xi$. Using Proposition \ref{nonlinear_part} again, we know that $\Xi$ is $C^2$.
		
		It is obvious from \eqref{nonlin_map_main} that $\FF$ maps real-valued $\psi,\Omega$ into real-valued distributions, so the implicit function theorem yields a real-valued $\Xi(\Omega)$ for real-valued $\Omega$.
	\end{proof}

	\section{Solvability of the linearized problem}\label{sec_linear_problem}

	In this section, we prove Proposition \ref{linear_part}. The first step is to compute the Fr\'echet derivative $\frac{\p\FF}{\p\psi}(\psi_0, \Omega_0): Y_m\to Z_m$. Recall the nonlinear map
	\begin{equation*}
		\FFF(\psi,\Omega)=\p_\vp \NNN_1(\psi)+\p_\phi\NNN_2(\psi)+\NNN_3(\psi)\Omega,
	\end{equation*}
	where
	\[\NNN_1(\psi)=\frac{2\psi_\beta\psi_\vp}{\psi_{\beta\vp}}-\frac{\psi_{\beta\phi}}{\psi_{\beta\vp}}\NNN_2(\psi),\ \  \NNN_2(\psi)=\frac{\psi_{\beta\vp}\psi_\phi-\psi_{\beta\phi}\psi_\vp}{2\psi_\beta},\ \ \NNN_3(\psi)=\frac{\psi_{\beta\vp}\psi_\vp^{-\frac1{2\mu}}}{2\mu};\]
	and the special solution $\psi_0=\frac1{2\mu-1}\beta^{1-2\mu}$, $\Omega_0=\gamma=2-\frac1\mu$.
	
	By Proposition \ref{nonlinear_part}, the nonlinear map $\FF: B_{\delta}^{(Y_m)}(\psi_0)\times\WWW_m\to Z_m$ is $C^1$. As a consequence, to compute the Fr\'echet derivative $\frac{\p\FF}{\p\psi}(\psi_0, \Omega_0): Y_m\to Z_m$, it suffices to compute the G\^{a}teaux derivative. For $\psi\in Y_m$, we have
	\begin{align*}
		\frac{d}{dt}\Big|_{t=0}\NNN_2(\psi_0+t\psi)&=\frac{d}{dt}\Big|_{t=0}\frac{\left(-2\mu\beta^{-2\mu-1}+t\psi_{\beta\vp}\right)t\psi_\phi-
			t\psi_{\beta\phi}\left(\beta^{-2\mu}+t\psi_\vp\right)}{2\left(-\beta^{-2\mu}+t\psi_{\beta}\right)}\\
		&=\frac{-2\mu\beta^{-2\mu-1}\psi_\phi-\beta^{-2\mu}\psi_{\beta\phi}}{-2\beta^{-2\mu}}=\frac12\left(\psi_{\beta\phi}+\frac{2\mu}\beta\psi_\phi\right);\\
		\frac{d}{dt}\Big|_{t=0}\NNN_1(\psi_0+t\psi)&=\frac{d}{dt}\Big|_{t=0}\frac{2\left(-\beta^{-2\mu}+t\psi_\beta\right)
			\left(\beta^{-2\mu}+t\psi_\vp\right)}{-2\mu\beta^{-2\mu-1}+t\psi_{\beta\vp}}\\
		&\qquad-\NNN_2(\psi_0)\frac{d}{dt}\Big|_{t=0}\frac{t\psi_{\beta\phi}}{-2\mu\beta^{-2\mu-1}+t\psi_{\beta\vp}}-
		\frac{\p_{\beta\phi}\psi_0}{\p_{\beta\vp}\psi_0}\frac{d}{dt}\Big|_{t=0}\NNN_2(\psi_0+t\psi)\\
		&=\frac1\mu\beta^{1-2\mu}\frac{d}{dt}\Big|_{t=0}\left(1-t\beta^{2\mu}\psi_\beta\right)\left(1+t\beta^{2\mu}\psi_\vp\right)
		\left(1-\frac1{2\mu}t\beta^{2\mu+1}\psi_{\beta\vp}\right)^{-1}\\
		&=\frac\beta\mu\left(\frac\beta{2\mu}\psi_{\beta\vp}+\psi_\vp-\psi_\beta\right);\\
		\frac{d}{dt}\Big|_{t=0}\NNN_3(\psi_0+t\psi)&=\frac1{2\mu}\frac{d}{dt}\Big|_{t=0}
		 \left(-2\mu\beta^{-2\mu-1}+t\psi_{\beta\vp}\right)\left(\beta^{-2\mu}+t\psi_\vp\right)^{-\frac1{2\mu}}=\frac1{2\mu}\left(\psi_\vp+\beta\psi_{\beta\vp}\right),
	\end{align*}
	where we used $\NNN_2(\psi_0)=0$, $\p_{\beta\phi}\psi_0=0$. Hence,
	\begin{align*}
		\frac{d}{dt}\Big|_{t=0}\FF(\psi_0+t\psi, \Omega_0)&=\frac1\mu\p_\vp\left(\beta\left(\frac\beta{2\mu}\psi_{\beta\vp}+\psi_\vp-\psi_\beta\right)\right)+
		\frac\mu\beta\p_\phi\left(\psi_\phi+\frac\beta{2\mu}\psi_{\beta\phi}\right)\\
		&\qquad+\frac\gamma{2\mu}\left(\psi_\vp+\beta\psi_{\beta\vp}\right).
	\end{align*}
	For $\psi\in Y_m$, let $H=\psi+\frac\beta{2\mu}\psi_\beta\in X_m$, then we can rewrite the above expression as
	\[\frac{d}{dt}\Big|_{t=0}\FF(\psi_0+t\psi, \Omega_0)=\frac1\mu\p_\vp(\beta H_\vp)+\frac\mu\beta H_{\phi\phi}+\frac\gamma{2\mu}\psi_\phi=:\LLL(\psi).\]
	This proves the first part of Proposition \ref{linear_part}. It remains to show that $\LLL: Y_m\to Z_m$ is an isomorphism and the norm of $\LLL^{-1}$ is independent of $m\geq 2$. We restate it as the follows.
	
	\begin{prop}\label{linear_thm}
		Let $\alpha\in \left(0,\alpha_\mu\right)$, $m\geq 2$ and $\mu>\frac12$, where $\alpha_\mu$ is given by \eqref{alpha_mu}. For any $G=\p_\vp F_1+\p_\phi F_2\in Z_m$, there exists a unique solution $\psi\in Y_m$ to the linearized equation
		\begin{equation}\label{linear_eq}
			\begin{cases}
				\frac1\mu \p_\vp(\beta H_\vp)+\frac{\mu H_{\phi\phi}}{\beta}+\frac{\gamma}{2\mu}\psi_\phi=\p_\vp F_1+\p_\phi F_2,\\
				H=\psi+\frac\beta{2\mu}\psi_\beta\in X_m.
			\end{cases}
		\end{equation}
		Moreover, there holds  $\beta^{2\mu-1}F_1\in\GG_m^-, \beta^{2\mu}F_2\in\GG_m^0$ and
		\begin{equation}\label{linear_est}
			\|\psi\|_{Y_m}\lesssim\|\beta^{2\mu-1}F_1\|_{\GG_m^-}+\|\beta^{2\mu}F_2\|_{\GG_m^0}.
		\end{equation}
	\end{prop}
	
	The proof of Proposition \ref{linear_thm} involves several steps. In Subsection \ref{Uniqueness}, we prove the uniqueness part.  In Subsection \ref{section_Basic_lin}, we consider a simplified linearized equation, i.e., \eqref{linear_Poisson_eq}, in which we ignore the nonlocal term $\frac{\gamma}{2\mu}\psi_\phi$. In Subsection \ref{Full_linear}, we consider the full linearized problem. We note that for high frequencies, the nonlocal term is a small perturbation of \eqref{linear_Poisson_eq}; and for low frequencies, we use the compactness method to construct solutions.
	
	\subsection{Uniqueness of the solution}\label{Uniqueness}
	In this subsection, we prove the uniqueness part of Proposition \ref{linear_thm}.
	\begin{prop}\label{uniqueness_prop}
		Assume that $\alpha\in(0,1)$, $m\geq 2$ and $\mu>1/2$. If $(H, \psi)\in X_m\times Y_m$ solves
		\begin{equation}\label{homogeneous_lin}
			\begin{cases}
				\frac1\mu \p_\vp(\beta H_\vp)+\frac{\mu H_{\phi\phi}}{\beta}+\frac{\gamma}{2\mu}\psi_\phi=0,\\
				H=\psi+\frac\beta{2\mu}\psi_\beta,
			\end{cases}
		\end{equation}
		then $H=0, \psi=0$.
	\end{prop}
	
	\begin{proof}
		We expand $H\in X_m$ and $\psi\in Y_m$ in the form of Fourier series: $$\displaystyle \psi(\beta,\phi)=\sum_{n\in\ZZ, m\mid n}\hat \psi_n(\beta)e^{\ii n(\beta+\phi)}, \qquad H(\beta,\phi)=\sum_{n\in\ZZ, m\mid n}\hat H_n(\beta)e^{\ii n(\beta+\phi)},$$
		where the Fourier coefficients are defined by
		\[\hat \psi_n(\beta)=\frac1{2\pi}\int_\TT \psi(\beta,\phi)e^{-\ii n(\beta+\phi)}\,d\phi,\qquad \hat H_n(\beta)=\frac1{2\pi}\int_\TT H(\beta,\phi)e^{-\ii n(\beta+\phi)}\,d\phi.\]
		Due to $\p_\vp=\p_\phi-\p_\beta$, the homogeneous linear equation \eqref{homogeneous_lin} can be rewritten in modes:
		\[\begin{cases}
			(\beta\p_\beta)^2\hat H_n-\mu^2n^2\hat H_n+\frac\gamma2\ii n\beta \hat \psi_n=0,\\
			\hat H_n=\hat \psi_n+\frac\beta{2\mu}\left(\p_\beta\hat \psi_n+\ii n\hat \psi_n\right),
		\end{cases} m\mid n.\]
		Thanks to $\gamma=2-\frac1\mu$, we get
		\begin{equation*}
			\begin{cases}
				\left(\beta\p_\beta-2\mu+1\right)^2\left(2\mu\beta^{2\mu-1}\hat H_n\right)-\mu^2n^2\left(2\mu\beta^{2\mu-1}\hat H_n\right)+(2\mu-1)\ii n\beta^{2\mu}\hat \psi_n=0,\\
				2\mu\beta^{2\mu-1}\hat H_n=\p_\beta\left(\beta^{2\mu}\hat \psi_n\right)+\ii n\beta^{2\mu}\hat \psi_n.
			\end{cases}
		\end{equation*}
		Since $(H, \psi)\in X_m\times Y_m$, we know that
		\[\beta^{2\mu-1}\hat H_n\in L^\infty,\qquad \beta^{2\mu-1}\hat \psi_n\in L^\infty.\]
		As a result, Lemma \ref{uniqueness_ODE} below implies that $\hat H_n=0$ and $\hat \psi_n=0$, and therefore $H=0, \psi=0$. This completes the proof of the uniqueness.
	\end{proof}

	\begin{lem}\label{uniqueness_ODE}
		Let $n\in\ZZ\setminus\{\pm1\}$ and $\mu>1/2$. If $Q\in L^\infty((0,\infty))$ and $\Psi$ solve
		\begin{equation}\label{homogen_ODE_freq}
			\begin{cases}
				\left(\beta\p_\beta-2\mu+1\right)^2Q-\mu^2n^2Q+(2\mu-1)\ii n\Psi=0,\\
				Q=\Psi_\beta+\ii n\Psi
			\end{cases}\beta>0,
		\end{equation}
		and $\Psi(\beta)/\beta\in L^\infty((0,+\infty))$, then $Q=0$, $\Psi=0$.
	\end{lem}
	
	We start with a classical ODE lemma regarding linear ordinary differential equations with regular singular points, whose proof is omitted here and we refer the readers to \cite{Fedoryuk}, section 2.1 of Chapter 1.
	
	\begin{lem}\label{ODE_lem}
		Consider the $3$-rd order linear equation
		\begin{equation}\label{3-ODE}
			w'''+\frac{p_1(z)}{z}w''+\frac{p_2(z)}{z^2}w'+\frac{p_3(z)}{z^3}w=0,
		\end{equation}
		where $p_1, p_2, p_3$ are holomorphic functions. Then $z=0$ is a regular singular point of  \eqref{3-ODE}. The indicial polynomial $p(\lambda)$ is given by
		\[p(\lambda)=\lambda(\lambda-1)(\lambda-2)+p_1(0)\lambda(\lambda-1)+p_2(0)\lambda+p_3(0).\]
		Assume that the indicial polynomial has three real roots $\lambda_1>\lambda_2\geq \lambda_3$. Then, by the standard Frobenius' method, these roots correspond to a fundamental system of solutions $\{w_1, w_2, w_3\}$ of the linear equation \eqref{3-ODE}, where $w_1$ can be expressed as a formal series of the Frobenius form which is convergent near $z=0$: $$w_1(z)=z^{\lambda_1}\sum_{k=0}^\infty a_k z^k,\qquad a_0=1\neq 0,$$and they have the asymptotic behavior near $z=0$ as follows (below the implicit constants in $\sim$ are independent of $z$)
		\begin{align*}
			w_1(z)&\sim z^{\lambda_1},\qquad  w_2(z)\sim z^{\lambda_2},\\
			w_3(z)&\sim z^{\lambda_3} \ \ \text{ if } \ \ \lambda_2\neq \lambda_3, \qquad w_3(z)\sim z^{\lambda_3}\ln z \ \ \text{ if } \ \ \lambda_2=\lambda_3.
		\end{align*}
	\end{lem}
	
	Now we are ready to prove Lemma \ref{uniqueness_ODE}.
	
	\begin{proof}[Proof of Lemma \ref{uniqueness_ODE}]
		For $n=0$, the function $Q\in L^\infty((0,\infty))$ satisfies the equation
		\beno
		\left(\beta\p_\beta-2\mu+1\right)^2Q=0,\quad  \text{i.e.},\quad (\beta\p_\beta)^2\left(\beta^{1-2\mu}Q\right)=0,
		\eeno
		hence, $Q=C_1\beta^{2\mu-1}\ln\beta+C_2\beta^{2\mu-1}$ for some constants $C_1, C_2$. Since $\mu>1/2$ and $Q\in L^\infty((0,\infty))$, we have $C_1=C_2=0$, hence $Q=0$. It follows from $\Psi_\beta=Q=0$ and $\lim\limits_{\beta\to0+}\Psi(\beta)=0$ that $\Psi=0$.
		
		Now we show the result for $|n|\geq 2$. We only prove the lemma for $n\geq 2$, as for $n\leq -2$ the proof is similar. So we assume that $n\geq 2$. The equation \eqref{homogen_ODE_freq} can be rewritten as
		\begin{align*}
			\Psi'''+\frac{(1-a_n^+-a_n^-)+\ii n\beta}{\beta}\Psi''&+\frac{a_n^+a_n^-+\ii n(1-a_n^+-a_n^-)\beta}{\beta^2}\Psi'\\
			&+\frac{0+\ii n\left(a_n^+a_n^-+2\mu-1\right)\beta}{\beta^3}\Psi=0,\qquad\beta>0,
		\end{align*}
		where $a_n^+=2\mu-1+n\mu, a_n^-=2\mu-1-n\mu$. Lemma \ref{ODE_lem} implies that $\beta=0$ is a regular singular point. We have the indicial polynomial
		\begin{align*}
			p(\lambda)&=\lambda(\lambda-1)(\lambda-2)+\lambda(\lambda-1)(1-a_n^+-a_n^-)+\lambda a_n^+a_n^-\\
			&=\lambda\left(\lambda-(2\mu+n\mu)\right)\left(\lambda-(2\mu-n\mu)\right),
		\end{align*}
		whose roots are $\lambda_1=2\mu+n\mu, \lambda_2=0, \lambda_3=2\mu-n\mu$, with $\lambda_1>\lambda_2\geq \lambda_3$. By Lemma \ref{ODE_lem}, they correspond to a fundamental system of solutions $\{\Psi_1, \Psi_2, \Psi_3\}$ of the linear equation \eqref{homogen_ODE_freq}, with the asymptotic behavior as $\beta\to0+$: $\Psi_1(\beta)\sim \beta^{2\mu+n\mu}$, $\Psi_2(\beta)\sim 1$, and $\Psi_3(\beta)\sim \ln \beta$ for $n=2$, $\Psi_3(\beta)\sim \beta^{2\mu-n\mu}$ for $n\geq 3$,  where the implicit constants depend on $\mu, n$ and $\Psi_1, \Psi_2, \Psi_3$ but independent of $\beta$. Since $\{\Psi_1, \Psi_2, \Psi_3\}$ is a fundamental system of solutions, there exist constants $C_1, C_2, C_3$ such that $\Psi=C_1\Psi_1+C_2\Psi_2+C_3\Psi_3$. It follows from $\lim\limits_{\beta\to0+}\Psi(\beta)=0$ that $C_2=C_3=0$, hence $\Psi=C_1\Psi_1$. Lemma \ref{ODE_lem} also implies that $\Psi_1$ can be expressed as a formal series of the Frobenius form
		\[\Psi_1(\beta)=\beta^{2\mu+n\mu}\sum_{k=0}^\infty a_k\beta^k,\qquad a_0=1\neq 0.\]
		It corresponds to a solution for $Q$ in the form
		\[Q_1(\beta)=\beta^{2\mu+n\mu}\sum_{k=-1}^\infty b_k\beta^k.\]
		Plugging the above two identities into the equation \eqref{homogen_ODE_freq}, we can deduce the recurrence relations of two sequences $\{a_k\}$ and $\{b_k\}$, then using $a_0=1$ we obtain (here we omit the tedious calculations)
		\begin{align*}
			\Psi_1(\beta)&=\beta^{2\mu+n\mu}\sum_{k=0}^\infty\frac{(\alpha_1+1)_k(\alpha_2+1)_k}{(2\mu+n\mu+1)_k(2n\mu+1)_k}\frac{(-\ii n\beta)^k}{k!}\\
			&=\beta^{2\mu+n\mu} \ _2F_2(\alpha_1+1, \alpha_2+1; 2\mu+n\mu+1, 2n\mu+1; -\ii n\beta),
		\end{align*}
		where $\alpha_1, \alpha_2\in\RR$ satisfies $\alpha_1+\alpha_2=2n\mu$ and $\alpha_1\alpha_2=2\mu-1$, $(a)_k$ denotes the (rising) Pochhammer symbol defined by
		\[(a)_k:=\begin{cases}
			1 & k=0\\
			a(a+1)\cdots(a+k-1) & k\in\NN_+
		\end{cases}\]
		for $a\in\RR$, and$\ _2F_2$ is the generalized hypergeometric function, see Chapter 16 in \cite{NIST}. By the properties of the generalized hypergeometric functions, the series defining $\Psi_1$ is convergent for all $\beta\in[0,\infty)$ and it is an analytic function. Moreover, by the general asymptotic properties of generalized hypergeometric functions, see section 16.11 in \cite{NIST}, we have
		\[\Psi_1(\beta)\sim \beta^{2\mu+n\mu}\beta^{-\left(\min\{\alpha_1, \alpha_2\}+1\right)}\sim \beta^{2\mu-1+\sqrt{n^2\mu^2-2\mu+1}},\qquad \beta\to+\infty.\]
		It follows from $\mu>1/2$ and $n\geq2$ that
		\[2\mu-1+\sqrt{n^2\mu^2-2\mu+1}\geq2\mu-1+\sqrt{4\mu^2-2\mu+1}=2\mu-1+\sqrt{(2\mu-1)^2+2\mu}>\sqrt{2\mu}>1,\]
		hence $C_1=0$, then $\Psi=0$ and $Q=0$. This proves Lemma \ref{uniqueness_ODE} .
	\end{proof}

	\subsection{Solvability of a simplified problem}\label{section_Basic_lin}
	In this subsection, we solve a simplified linearized problem. Note that all the results in this subsection are valid for all $\alpha\in(0,1)$. We remark that all solutions constructed in this subsection not only exist, but also are unique and the uniqueness can be proved by using the same idea as in the previous subsection.

	Given a positive integer $N$, we denote by $X^N$ the subspace of $X_0$ consisting of functions with Fourier modes higher than $N$. More precisely, we define
	\[X^N:=\left\{H\in X_0:\hat{H }_n(\beta)=0\  \text{ for all } |n|\leq N,\ n\in \ZZ\right\},\quad \|H\|_{X^N}:=\|H\|_{X_0}.\]
	 Here and in what follows ${\hat{H }_n(\beta):=\frac{1}{2\pi}}\int_\TT H(\beta,\phi)e^{-\ii n(\beta+\phi)}\,d\phi$ denotes the Fourier coefficient of $H$ with respect to $ \theta=\beta+\phi$.
	Similarly, we define
	\begin{align*}
		Y^N:=\left\{\psi\in Y_0:\hat{\psi}_n(\beta)=0\  \forall\ |n|\leq N,\ n\in \ZZ\right\}, \quad \|\psi\|_{Y^N}:=\|\psi\|_{Y_0},\\
		\GG^{0, N}:=\left\{f\in \GG^0:\hat{f}_n(\beta)=0\  \forall\ |n|\leq N,\ n\in \ZZ\right\}, \quad \|f\|_{\GG^{0,N}}:=\|f\|_{\GG^0}.
	\end{align*}
Note that $X^N$, $Y^N$, $\GG^{0,N}$ are closed subspaces of $X_{0}$, $Y_{0}$ and $\GG^0$, respectively.  From here up to Lemma \ref{lem6.11}, all constants in $\lesssim$ depend only on $\alpha,\mu$, hence they are independent of $m\geq2$ and $N\in\NN_+$.
	\begin{prop}\label{linear_schauder_prop}
		Assume that $\alpha\in(0,1)$, $m\geq 2$ and $\mu>1/2$. There exists a solution $H\in X^1$ to the equation
		\begin{equation}\label{linear_Poisson_eq}
			\frac1\mu \p_\vp(\beta H_\vp)+\frac{\mu H_{\phi\phi}}{\beta}=G=\p_\vp F_1+\p_\phi F_2
		\end{equation}
		for any $\beta^{2\mu-1}F_1\in\GG^{0,1},$
$ \beta^{2\mu}F_2\in\GG^{0,1} $. And we have
		\[\|H\|_{X_{0}}\lesssim\|\beta^{2\mu-1}F_1\|_{C_{\beta}^\alpha}+\|\beta^{2\mu}F_2\|_{C_{\beta}^\alpha}.\]
		Moreover for $n\in \ZZ$, if $\widehat{F_1}_{,n}=\widehat{F_2}_{,n}=0$, then $\widehat{H}_n=0$. In particular, if $F_1$ and $F_2$ are $m$-fold symmetric, then so does $H$.
	\end{prop}
	
	 For each $N\in\NN_+$, we define
	\begin{equation*}
		Z^N:=\left\{G\in Z: G=\p_\vp F_1+\p_\phi F_2 \text{ weakly with } \beta^{2\mu-1}F_1\in \GG^{0, N}, \beta^{2\mu}F_2\in\GG^{0,N}\right\},
	\end{equation*}
	and the norm is $\|G\|_{Z^1}:=\inf\left(\|\beta^{2\mu-1}F_1\|_{\GG^0}+\|\beta^{2\mu}F_2\|_{\GG^0}\right)$, $\|G\|_{Z^N}:=\|G\|_{Z^1} (N\geq 2)$, noticing that $Z^{N}\subset Z^{N-1}$ for all $N\geq 2$.
For further usage, we denote the solution operator from $G$ to $H$ in Proposition \ref{linear_schauder_prop} by $\mathcal{H}: Z^1\to X^1$. Then $\mathcal{H}$ is a bounded linear operator. We also define the restriction
	\[\mathcal{H}^N:=\mathcal H|_{Z^N}: Z^N\to X^N,\qquad N\geq1.\]
	Then $\mathcal{H}^N$ are well-defined bounded linear operators with norms independent of $N$.\smallskip
	
\begin{rmk}
This proposition could be viewed as a partial Schauder estimate for the elliptic equation, see  \cite{Dong, JLW}  for relevant results.
 \end{rmk}		

 The proof of Proposition \ref{linear_schauder_prop} relies on the following  proposition, whose proof is quite complicated and will be given in Appendix \ref{AppendixB1}.  We suggest that readers skip Appendix \ref{AppendixB1}  when reading for the first time in order to quickly grasp the overall idea of this article.

 \begin{prop}\label{linear_basic_estimate1}
		Assume that $\alpha\in(0,1)$, $m\geq 2$ and $\mu>1/2$. There exists a solution $Q\in X^1$ to the equation
		\begin{equation}\label{linear_basic_eq}
			(\beta\p_\vp+\ii\mu\p_\phi)Q=-G,
		\end{equation}
		for any $G$ satisfying  $\beta^{2\mu-1}G\in C_{\beta}^\alpha$ and the Fourier coefficients $\hat G_0=\hat G_{\pm1}=0$, and we have
		\begin{equation}\label{linear_basic_est}
			\|Q\|_{X_0}\lesssim \|\beta^{2\mu-1}G\|_{C_{\beta}^\alpha}.
		\end{equation}
		Moreover for $n\in \ZZ$, if $\hat G_n=0$ then $\hat Q_n=0$.
		In particular, if $G$ is $m$-fold symmetric and $\hat G_0=0$, then  $\hat G_{\pm1}=0$, and $Q$ is also $m$-fold symmetric and $\hat Q_0=0$. The same results hold for the equation $(\beta\p_\vp-\ii\mu\p_\phi)Q=-G$.
	\end{prop}

	\begin{proof}[Proof of Proposition \ref{linear_schauder_prop}]
		The linear equation \eqref{linear_Poisson_eq} can be rewritten as
		\[(\beta\p_\vp)^2H+(\mu\p_\phi)^2H=\left(\beta\p_\vp-\ii\mu\p_\phi\right)\frac{\mu F_1+\ii\beta F_2}{2}+\left(\beta\p_\vp+\ii\mu\p_\phi\right)\frac{\mu F_1-\ii\beta F_2}{2}.\]
		Let $Q_1, Q_2\in X^1$ be solutions to $\left(\beta\p_\vp+\ii\mu\p_\phi\right)Q_1=\mu F_1+\ii\beta F_2$ and $\left(\beta\p_\vp-\ii\mu\p_\phi\right)Q_2=\mu F_1-\ii\beta F_2$ respectively, the existence of which is ensured by Proposition \ref{linear_basic_estimate1}. Then $H=\frac{Q_1+Q_2}{2}\in X^1$ is a solution to the equation \eqref{linear_Poisson_eq}. Moreover, Proposition \ref{linear_basic_estimate1} implies that $$\|Q_1\|_{X_0}\lesssim\|\beta^{2\mu-1}(\mu F_1+\ii\beta F_2)\|_{C_\beta^\alpha}\lesssim\|\beta^{2\mu-1}F_1\|_{C_{\beta}^\alpha}+\|\beta^{2\mu}F_2\|_{C_{\beta}^\alpha}.$$
		Similar argument gives the estimate for $Q_2$, hence the estimate for $H$.
	\end{proof}

	\subsection{Solvability of the full linearized problem}\label{Full_linear}
	In this subsection,  we prove Proposition \ref{linear_thm}. It suffices to show the following proposition.
	
	\begin{prop}\label{exist_2_prop}
		Assume that $\mu>\frac12$ and $\alpha\in\left(0,\alpha_\mu\right)$, where $\alpha_\mu$ is given by \eqref{alpha_mu}. For any $H\in X_0$ such that $\hat{H}_0=\hat{H}_{\pm1}=0$, there exists $(Q, \psi)\in X_0\times Y_0$ solving the linear system
		\begin{equation}\label{exist_2_Q}
			\begin{cases}
				\frac1\mu\p_\vp(\beta Q_\vp)+\frac{\mu Q_{\phi\phi}}{\beta}+\frac\gamma{2\mu}\psi_\phi=0,\\
				H+Q=\psi+\frac\beta{2\mu}\psi_\beta,
			\end{cases}
		\end{equation}
		and we have the estimate
		\beno
		\|Q\|_{X_0}+\|\psi\|_{Y_0}\lesssim\|H\|_{X_0}.
		\eeno
	\end{prop}
	
	Similar to Proposition \ref{linear_basic_estimate1}, if $\hat H_n=0$, then $\hat Q_n=\hat \psi_n=0$.
	In particular, if $H$ is $m$-fold symmetric with $P_0H=0$, then so are $Q$ and  $\psi$.
	
	\begin{proof}[Proof of Proposition \ref{linear_thm}]
		Recall that the uniqueness part of Proposition \ref{linear_thm} has been proved in Subsection \ref{Uniqueness}. We only need to prove the existence part. We decompose
		\begin{align*}
			\beta^{2\mu-1}F_1&=\beta^{2\mu-1}P_0F_1+\beta^{2\mu-1}P_{\neq}F_1=\beta^{2\mu-1}F_0+c_0+\beta^{2\mu-1}P_{\neq}F_1,\\
			\beta^{2\mu}F_2&=\beta^{2\mu}P_0F_2+\beta^{2\mu}P_{\neq}F_2.
		\end{align*}
		Then we have
		\begin{align*}
			\|\beta^{2\mu-1}F_1\|_{\GG_m^-}&=\|\beta^{2\mu-1}{F_0}\|_{C_\beta^\alpha}+|c_0|+m^{1/2}\|\beta^{2\mu-1}P_{\neq}F_1\|_{C_\beta^\alpha}\\
			\|\beta^{2\mu}F_2\|_{\GG_m^0}&=\|\beta^{2\mu}P_0F_2\|_{C_\beta^\alpha}+m^{1/2}\|\beta^{2\mu}P_{\neq}F_2\|_{C_\beta^\alpha}.
		\end{align*}
		
		It is easy to check that
		\[H_0(\beta):=\mu\int_\beta^\infty\frac{F_0(s)}{s}\,ds+\frac{c_0\mu}{2\mu-1}\beta^{1-2\mu}\]
		is a solution to $\frac1\mu\p_\beta(\beta\p_\beta H_0)=-\p_\beta(P_0F_1)$, i.e., $\frac1\mu\p_{\varphi}(\beta\p_{\varphi} H_0)=\p_{\varphi}(P_0F_1)$. We need to estimate $\|H_0\|_{X}$. Let $H_{0,1}(\beta)=\mu\int_\beta^\infty\frac{{F_0}(s)}{s}\,ds$. By the definition and Lemma \ref{X=tildeX}, we have
		\[\|H_0\|_{X}=\|H_{0,1}\|_{X_0}+\left|\frac{c_0\mu}{2\mu-1}\right|\lesssim\|\beta^{2\mu}\p_\beta H_{0,1}\|_{C_\beta^\alpha}+\|\langle\beta\rangle^\alpha\beta^{2\mu-1}H_{0,1}\|_{L^\infty}+|c_0|.\]
		Direct computation gives that $\p_\beta H_{0,1}=-\mu\frac{{F_0}}{\beta}$, hence $\|\beta^{2\mu}\p_\beta H_{0,1}\|_{C_\beta^\alpha}\lesssim\|\beta^{2\mu-1}{F_0}\|_{C_\beta^\alpha}$.
		For $\|\langle\beta\rangle^\alpha\beta^{2\mu-1}H_{0,1}\|_{L^\infty}$, we have
		\begin{align*}
			\left|H_{0,1}(\beta)\right|&\lesssim \left(\int_\beta^\infty \langle s\rangle^{-\alpha}s^{-2\mu}\,ds\right)\|\langle\beta\rangle^\alpha\beta^{2\mu-1}{F_0}\|_{L^\infty}\lesssim
			\langle\beta\rangle^{-\alpha}\beta^{1-2\mu}\|\beta^{2\mu-1}{F_0}\|_{C_\beta^\alpha},
		\end{align*}
		hence $\|\langle\beta\rangle^\alpha\beta^{2\mu-1}H_{0,1}\|_{L^\infty}\lesssim\|\beta^{2\mu-1}{F_0}\|_{C_\beta^\alpha}.$ Therefore, $\|H_0\|_X\lesssim \|\beta^{2\mu-1}F_1\|_{\GG_m^-}$. Let
		\[\psi_0(\beta)=2\mu\beta^{-2\mu}\int_0^\beta s^{2\mu-1}H_{0,1}(s)\,ds+\frac{2c_0\mu^2}{2\mu-1}\beta^{1-2\mu}=T_0(H_{0,1})(\beta)+\frac{2c_0\mu^2}{2\mu-1}\beta^{1-2\mu}.\]
		According to Lemma \ref{F<H}, $\psi_0(\beta)\in Y$ with $\|\psi_0\|_{Y}\lesssim\|H_{0,1}\|_{X_0}+|c_0|\lesssim \|\beta^{2\mu-1}F_1\|_{\GG_m^-}$ and
(recall that $\frac1\mu\p_{\varphi}(\beta\p_{\varphi} H_0)=\p_{\varphi}(P_0F_1)$ and $\p_{\phi}^2H_0=0 $, $\p_{\phi}\psi_0=0 $)
		\begin{equation*}
			\begin{cases}
				\frac1\mu\p_{\varphi}(\beta\p_{\varphi} H_0)+\frac{\mu}{\beta}\p_\phi^2H_0+\frac\gamma{2\mu}\p_\phi \psi_0=\p_{\varphi}(P_0F_1),\\
				H_0=\psi_0+\frac\beta{2\mu}\p_\beta \psi_0.
			\end{cases}
		\end{equation*}
		
		 As $\beta^{2\mu-1}F_1\in\GG_m^-,$
		$ \beta^{2\mu}F_2\in\GG_m^-$ with $m\geq 2$, we have $\beta^{2\mu-1}P_{\neq}F_1\in\GG^{0,1},$
		$ \beta^{2\mu}P_{\neq}F_2\in\GG^{0,1} $. By Proposition \ref{linear_schauder_prop}, there exists $H_{\neq}\in X_{0,m}$ such that $P_0H_{\neq}=0$, $(\widehat{H_{\neq}})_{\pm1}=0$,
		\[\frac1\mu\p_\vp\left(\beta\p_\vp H_{\neq}\right)+\frac{\mu\p_\phi^2H_{\neq}}{\beta}=\p_\vp P_{\neq}F_1+\p_\phi P_{\neq}F_2,\]
		and
		\[\|H_{\neq}\|_{X_0}\lesssim\|\beta^{2\mu-1}P_{\neq}F_1\|_{C_\beta^\alpha}+\|\beta^{2\mu}P_{\neq}F_2\|_{C_\beta^\alpha}\lesssim
		m^{-1/2}(\|\beta^{2\mu-1}F_1\|_{\GG_m^-}+\|\beta^{2\mu}F_2\|_{\GG_m^0}).\]
		By Proposition \ref{exist_2_prop}, there exists  $(Q, \psi_{\neq})\in X_m\times Y_m$  such that $P_0Q=P_0\psi_{\neq}=0$,
		\begin{align*}\frac1\mu\p_\vp(\beta Q_\vp)+\frac{\mu Q_{\phi\phi}}{\beta}+\frac\gamma{2\mu}\partial_\phi\psi_{\neq}=0,\quad
			H_{\neq}+Q=\psi_{\neq}+\frac\beta{2\mu}\partial_\beta\psi_{\neq},\end{align*} and $\|\psi_{\neq}\|_{Y_0}\lesssim\|H_{\neq}\|_{X_0}$. Then
		$(H,\psi)=(H_0+H_{\neq}+Q, \psi_{0}+\psi_{\neq})\in X_m\times Y_m$ is a solution to \eqref{linear_eq} (note that $\p_\phi P_0F_2=0$), and 	
		\begin{align*}
			\|\psi\|_{Y_m}=\|\psi_0\|_{Y}+m^{1/2}\|\psi_{\neq}\|_{Y_0}\lesssim\|\psi_0\|_{Y}+m^{1/2}\|H_{\neq}\|_{X_0}\lesssim
			\|\beta^{2\mu-1}F_1\|_{\GG_m^-}+\|\beta^{2\mu}F_2\|_{\GG_m^0}.
		\end{align*}
		
		This concludes the proof of Proposition \ref{linear_thm}.
	\end{proof}
	
\subsection{Proof of Proposition \ref{exist_2_prop}}\label{Subsec.Proof-full}

 Our strategy is to investigate the equation \eqref{exist_2_Q} separately in high frequencies  and low frequencies with respect to $\phi\in\TT$. For high frequencies, the nonlocal term $\frac\gamma{2\mu}\psi_\phi$ can be viewed as a perturbation; and for low frequencies, we can convert the equation \eqref{exist_2_Q} to a finite number of ODE systems for each frequency and then use the compactness method to construct the solution.
 	
\begin{lem}[Reverse Bernstein's inequality]\label{Reverse_bernstein}
		Let $N\geq 1$ be a positive integer. Let $f=f(\phi)\in C^1(\TT)$,
		$\int_\TT f(\phi)e^{-\ii n\phi}\,d\phi=0,$ $\forall \ |n|\leq N,$
		then
		\begin{equation}\label{reverse_bernstein1}
			\|f\|_{L^\infty(\TT)}\leq \frac CN\|f'\|_{L^\infty(\TT)},
		\end{equation}
		where $C>0$ is a positive constant not depending on $N$ and $f$.
	\end{lem}
	The proof of this lemma is technical and involved. Luckily, this is a classical result that be found in many text books, for example \cite{Katznelson}, section 8 in Chapter 1. So we omit the proof of Lemma \ref{Reverse_bernstein}.
	
	\begin{rmk}
		{Let $g=g(\beta,\phi)\in C((0,+\infty)\times\TT)$ and $N\in\NN_+$ satisfy $$\hat g_n(\beta):=\frac{1}{2\pi}\int_\TT g(\beta,\phi)e^{-\ii n(\beta+\phi)}\,d\phi=0 \qquad\text{for all}\quad |n|\leq N$$  and $g_{\phi}(\beta,\phi)\in C((0,+\infty)\times\TT)$. Then we have}
		\begin{equation}\label{reverse_bernstein2}
			\|g\|_{L^\infty}\leq\frac CN\|g_\phi\|_{L^\infty},\qquad \|g\|_{C_\beta^\al}\leq\frac CN\|g_\phi\|_{C_\beta^\al},
		\end{equation}
		where $C>0$ is a positive constant independent of  $N$ and $g$, $\al\in(0,1)$. Indeed, for fixed $\beta>0$, if we let $f(\phi)=g(\beta,\phi)$, then by \eqref{reverse_bernstein1} we get
		\[\sup_{\phi\in\TT}|g(\beta,\phi)|=\|f\|_{L^\infty(\TT)}\leq \frac CN\|f_\phi\|_{L^\infty(\TT)}=\frac CN\sup_{\phi\in\TT}|g_\phi(\beta,\phi)|.\]
		Taking the supremum in the above identity with respect to $\beta>0$ gives $\|g\|_{L^\infty}\leq\frac CN\|g_\phi\|_{L^\infty}$. For fixed $\beta_1\neq\beta_2>0$, if we take $f(\phi)=g(\beta_1, \phi)-g(\beta_2, \phi)$, then by \eqref{reverse_bernstein1} we get
		\[\sup_{\phi\in\TT}|g(\beta_1,\phi)-g(\beta_2,\phi)|=\|f\|_{L^\infty(\TT)}\leq \frac CN\|f_\phi\|_{L^\infty(\TT)}=\frac CN\sup_{\phi\in\TT}|g_\phi(\beta_1,\phi)-g_\phi(\beta_2,\phi)|.\]
		Hence,
		\begin{align*}
			&\sup_{\substack{0<\beta_2<\beta_1<2\beta_2\\
					0<|\beta_1-\beta_2|<1}}\sup_{\phi\in\TT}|\beta_1+\beta_2|^\alpha\frac{|g(\beta_1,\phi)-g(\beta_2,\phi)|}{|\beta_1-\beta_2|^\alpha}\\
				&\qquad\qquad\qquad\qquad \leq\frac CN\sup_{\substack{0<\beta_2<\beta_1<2\beta_2\\
					0<|\beta_1-\beta_2|<1}}\sup_{\phi\in\TT}|\beta_1+\beta_2|^\alpha\frac{|g_\phi(\beta_1,\phi)-g_\phi(\beta_2,\phi)|}{|\beta_1-\beta_2|^\alpha},
		\end{align*}
		 which along with $\|\langle\beta\rangle^\alpha g\|_{L^\infty}\leq\frac CN\|\p_\phi\left(\langle\beta\rangle^\alpha g\right)\|_{L^\infty}=\frac CN\|\langle\beta\rangle^\alpha\p_\phi g\|_{L^\infty}$ gives the second inequality in \eqref{reverse_bernstein2}.
	\end{rmk}
	
\begin{lem}[Bernstein's inequality]\label{Bernstein}
		If $f(\phi)=\sum\limits_{|n|\leq N}a_n e^{\ii n\phi}, \phi\in\TT$, then
		\[\left\|f'\right\|_{L^\infty(\TT)}\leq N\left\|f\right\|_{L^\infty(\TT)}.\]
	\end{lem}
	
	This  is a classical result and we refer to \cite{Katznelson}, section 7 in Chapter 1, Exercise 15.
	
{\begin{rmk}
		Similar to \eqref{reverse_bernstein2}, if $g(\beta,\phi)=\sum\limits_{|n|\leq N}\widehat{g}_n(\beta) e^{\ii n(\beta+\phi)}$, then
 we have
		\begin{equation}\label{bernstein2}
			\|\partial_{\phi}g\|_{L^\infty}\leq N\|g\|_{L^\infty},\quad \|\partial_{\phi}g\|_{C_{\beta}^{\alpha}}\leq N\|g\|_{C_{\beta}^{\alpha}}.\end{equation}
	\end{rmk}

	
	\begin{lem}\label{lem6.11}
		For each $N\in\NN_+$, we have $\p_\phi: Y^N\to Z^N$ with
		\begin{equation}\label{5.36}
			\|\p_\phi G\|_{Z}\leq\frac CN\|G\|_{Y_0} \qquad \text{for all}\quad  G\in Y^N,
		\end{equation}
		here $C$ is independent of $N$.
	\end{lem}
	\begin{proof}
		For $G\in Y^N$, by Lemma \ref{m_antiderivative} there exists $\psi$ such that $G=\psi_\phi$ and $P_0\psi=0$. Then $\p_\phi G=\p_\vp G+G_\beta=\p_\vp F+\p_\phi\psi_{\beta}$ and thus
		\begin{equation*}
			\begin{aligned}
				\|\p_\phi G\|_{Z}\lesssim&\|\beta^{2\mu-1}G\|_{C_\beta^\alpha}+\|\beta^{2\mu}\psi_\beta\|_{C_\beta^\alpha}\overset{\text{\eqref{reverse_bernstein2}}}{\lesssim} \frac1N \left\|\beta^{2\mu-1}G_\phi\right\|_{C_\beta^\alpha}+\frac1N \left\|\beta^{2\mu}\psi_{\beta\phi}\right\|_{C_\beta^\alpha}\\ \lesssim&\frac1N \left\|\beta^{2\mu-1}G_\phi\right\|_{C_\beta^\alpha}+\frac1N \left\|\beta^{2\mu}G_{\beta}\right\|_{C_\beta^\alpha}\overset{\text{Lemma \ref{F<H}}}{\lesssim}\frac1N\|G\|_{Y_0}.
			\end{aligned}
		\end{equation*}
		This completes the proof of the lemma.
	\end{proof}

	Now we are ready to construct the solution to \eqref{exist_2_Q} in high frequencies. Recall the operators $T_0: X_0\to Y_0$  defined in Lemma \ref{F<H} and $\mathcal{H}: Z^1\to X^1$ defined in Proposition \ref{linear_schauder_prop}. By our construction of these two operators, their restriction
	\begin{equation}\label{Eq.T_0^N_def}
		T_0^N= T_0|_{X^N}: X^N\to Y^N,\qquad \mathcal{H}^N= \mathcal{H}|_{Z^N}: Z^N\to X^N
	\end{equation}
	are well-defined bounded linear operators with bounds independent of $N$.
	
	\begin{prop}\label{high_frequency}
		Assume that $\alpha\in\left(0,1\right)$ and $\mu>1/2$. There exists a positive integer $N$ such that for any $H\in X^N$, we can find $Q\in X^N$ and $\psi\in Y^N$ satisfying the linear system \eqref{exist_2_Q} with the estimate
		\beno
		\|Q\|_{X_0}+\|\psi\|_{Y_0}\lesssim\|H\|_{X_0}.
		\eeno
	\end{prop}
	
	\begin{proof}
	By the definitions of the operators $\mathcal{H},\ \mathcal{H}^N,\ T_0,\ T_0^N $, it is enough to find $Q\in X^N$, $\psi\in Y^N$ such that
	\begin{equation}\label{Q}
			Q=-\frac{\gamma}{2\mu}\mathcal{H}^N\psi_\phi,\quad \psi=T_0^N(H+Q).
		\end{equation}
	We define the linear operator $T^N=\frac{\gamma}{2\mu}\mathcal{H}^N\partial_{\phi}T_0^N$, then $T^N: X^N\to X^N$ is well defined, since $T_0^N: X^N\to Y^N$,
$\p_\phi:Y^N\to Z^N$(by Lemma \ref{lem6.11}),
$\mathcal{H}^N: Z^N\to X^N $. And \eqref{Q} becomes $Q+T^N(H+Q)=0$.
It is sufficient to
show that $\mathrm{id} + T^N$ is invertible for $N$ large enough.
For any $Q_1\in X^N$, by \eqref{5.36} and Lemma \ref{F<H}, we have
		\begin{align*}
			&\ \ \ \left\|T^N(Q_1)\right\|_{X_0}\lesssim \left\|\mathcal{H}^N\p_{\phi}\left(T_0 Q_1\right)\right\|_{X_0}\lesssim\left\|\p_{\phi}\left(T_0 Q_1\right)\right\|_{Z}{\lesssim}\frac1N\|T_0Q_1\|_{Y_0}{\lesssim}\frac1N\|Q_1\|_{X_0}.
		\end{align*}Here $ \lesssim$ is with a constant independent of $N$ and $Q_1\in X^N$, and we used that the
norm of $ \mathcal{H}^N$ is independent of $N$.
		Hence, for $N$ large enough, we have $\left\|T^N\right\|_{X^N\to X^N}<\frac12$. Now we fix such an $N$, then $\mathrm{id} + T^N$ is invertible
with $\left\|(\mathrm{id} + T^N)^{-1}(Q_1)\right\|_{X_0}\leq2 \|Q_1\|_{X_0}$ for $Q_1\in X^N$, and the solution to \eqref{Q} is $Q=-(\mathrm{id} + T^N)^{-1}T^NH\in X^N$
and $\psi:=T_0^N(H+Q)\in Y^N$ (as $H\in X^N$ and $T_0^N: X^N\to Y^N$).
		Hence, $\|Q\|_{X_0}\leq 2\|T^NH\|_{X_0}\leq \|H\|_{X_0} $ and (using Lemma \ref{F<H} for $\psi=T_0^N(H+Q)=T_0(H+Q)$)
		\begin{align*}
			\|Q\|_{X_0}+\|\psi\|_{Y_0}\lesssim& \|H\|_{X_0}+\|Q\|_{X_0}\lesssim \|H\|_{X_0}.
		\end{align*}
		
		The proof of Proposition \ref{high_frequency} is completed.
	\end{proof}
	
	Now, we turn to the construction of low-frequency solution of \eqref{exist_2_Q}. Fix the integer $N$ given by Proposition \ref{high_frequency}.  To avoid ambiguity, we denote this fixed $N$ by $N_0$ thenceforth. In the rest of the subsection, we assume that $\alpha\in\left(0,\alpha_\mu\right)$ with $\alpha_\mu$ given by \eqref{alpha_mu}.  Also, the implicit constants in all $\lesssim$ in the rest of this subsection depend only on parameters $\alpha,\mu$, this fixed $N_0$ and the bump function $\rho$ introduced below.
	
	\begin{prop}\label{low_frequency}
		Assume that $\mu>\frac12$ and $\alpha\in\left(0,\alpha_\mu\right)$. For any $\displaystyle H=\sum_{2\leq |n|\leq N_0}\hat H_n(\beta)e^{\ii n(\beta+\phi)}\\ \in X_0$, there exist $\displaystyle Q=\sum_{2\leq |n|\leq N_0}\hat Q_n(\beta)e^{\ii n(\beta+\phi)}\in X_0$ and $\displaystyle \psi=\sum_{2\leq |n|\leq N_0}\hat \psi_n(\beta)e^{\ii n(\beta+\phi)}\in Y_0$ solving the linear system \eqref{exist_2_Q} with the estimate
		\beno
		\|Q\|_{X_0}+\|\psi\|_{Y_0}\lesssim\|H\|_{X_0}.
		\eeno
	\end{prop}
	
	Combining the above two propositions, we are able to prove Proposition \ref{exist_2_prop}.
	
	\begin{proof}[Proof of Proposition \ref{exist_2_prop}]
		Let $N_0=N$ be given by Proposition \ref{high_frequency}. For $H\in X^1$, we perform the decomposition
		\begin{equation*}
			H(\beta,\phi)=H^{\text{low}}(\beta,\phi)+H^{\text{high}}(\beta,\phi),
		\end{equation*}
		where
		\[H^{\text{low}}(\beta,\phi)=\sum_{2\leq |n|\leq N_0}\hat H_n(\beta)e^{\ii n(\beta+\phi)},\quad  H^{\text{high}}\in X^{N_0}.\]
		By Proposition \ref{high_frequency}, there exist $Q^{\text{high}}\in X^{N_0}$ and $\psi^{\text{high}}\in Y^{N_0}$ such that
		\begin{equation*}
			\begin{cases}
				\frac1\mu\p_\vp\left(\beta \p_\vp Q^{\text{high}}\right)+\frac{\mu}{\beta}\p_\phi^2Q^{\text{high}}+\frac\gamma{2\mu}\p_\phi \psi^{\text{high}}=0,\\
				H^{\text{high}}+Q^{\text{high}}=\psi^{\text{high}}+\frac\beta{2\mu}\p_\beta \psi^{\text{high}},
			\end{cases}
		\end{equation*}
		with the estimate
		\beno
		\left\|Q^{\text{high}}\right\|_{X_0}+\left\|\psi^{\text{high}}\right\|_{Y_0}\lesssim\left\|H^{\text{high}}\right\|_{X_0}\lesssim \|H\|_{X_0}.
		\eeno
		By Proposition \ref{low_frequency}, there exist $Q^{\text{low}}\in X_0$ and $\psi^{\text{low}}\in Y_0$ such that
		\begin{equation*}
			\begin{cases}
				\frac1\mu\p_\vp\left(\beta \p_\vp Q^{\text{low}}\right)+\frac{\mu}{\beta}\p_\phi^2Q^{\text{low}}+\frac\gamma{2\mu}\p_\phi \psi^{\text{low}}=0,\\
				H^{\text{low}}+Q^{\text{low}}=\psi^{\text{low}}+\frac\beta{2\mu}\p_\beta \psi^{\text{low}},
			\end{cases}
		\end{equation*}
		with the estimate
		\beno
		\left\|Q^{\text{low}}\right\|_{X_0}+\left\|\psi^{\text{low}}\right\|_{Y_0}\lesssim\left\|H^{\text{low}}\right\|_{X_0}\lesssim \|H\|_{X_0}.
		\eeno
		
		Finally, let $Q=Q^{\text{low}}+Q^{\text{high}}$ and $\psi=\psi^{\text{low}}+\psi^{\text{high}}$, then $(Q, \psi)\in X_0\times Y_0$ solves the linear system \eqref{exist_2_Q} and $\|Q\|_{X_0}+\|\psi\|_{Y_0}\lesssim \|H\|_{X_0}$.
	\end{proof}
	
	As a result, it suffices to prove Proposition \ref{low_frequency}. Recall that here we fix the integer $N_0=N$ given by Proposition \ref{high_frequency} and let $\alpha\in\left(0, \alpha_\mu\right)$ and $\mu>\frac12$. We define two auxiliary functional spaces as follows
	\[\tilde X:=\Big\{H=\sum_{2\leq |n|\leq N_0}\hat H_n(\beta)e^{\ii n(\beta+\phi)}\in C^1\ \Big|\ \left\|H\right\|_{\tilde X}<+\infty\Big\},\]
	with the norm
	\[\left\|H\right\|_{\tilde X}:=\left\|\langle\beta\rangle^\alpha\beta^{2\mu}\p_\vp H\right\|_{L^\infty}+\left\|\langle\beta\rangle^\alpha\beta^{2\mu-1}H\right\|_{L^\infty};\]
	and
	\[\tilde Y:=\Big\{\psi=\sum_{2\leq |n|\leq N_0}\hat \psi_n(\beta)e^{\ii n(\beta+\phi)}\in C^0\ \Big|\ \lim_{\beta\to0+}\beta^{2\mu}\psi(\beta,\phi)=0, \left\|\psi\right\|_{\tilde Y}<+\infty\Big\},\]
	with the norm
	\[\left\|\psi\right\|_{\tilde Y}:=\left\|\langle\beta\rangle^{\alpha-1}\beta^{2\mu}\psi\right\|_{L^\infty}.\]
	Note that if $\displaystyle H=\sum_{2\leq |n|\leq N_0}\hat H_n(\beta)e^{\ii n(\beta+\phi)}\in X_0$, then $H\in\tilde X$, and similarly for $\psi\in Y_0$ (since $\langle\beta\rangle^\alpha\beta^{2\mu-1}\psi\in L^\infty$ if $\psi\in Y_0$, by {Lemma \ref{F<H}}).
	
	To prove Proposition \ref{low_frequency}, our strategy  is to construct a solution firstly in the larger spaces $\tilde X$ and $\tilde Y$ by using the compactness method, and then we show that the solution constructed in the first step has the required $C^\alpha$ regularity to be lying in $X_0$ and $Y_0$.

	\begin{lem}\label{compact}
		Define the operator $\tilde T_0: \tilde X\to \tilde Y$ by
		\[F(\beta,\phi)=\tilde T_0(H)(\beta,\phi)=2\mu\beta^{-2\mu}\int_0^\beta s^{2\mu-1}H(s,\phi)\,ds,\qquad H\in \tilde X.\]
		Then $H=\tilde T_0(H)+\frac\beta{2\mu}\p_\beta \tilde T_0(H)$ for all $H\in \tilde X$ and $\tilde T_0: \tilde X\to \tilde Y$ is a compact operator.
	\end{lem}
	
	\begin{proof}
		It is direct to show that $\tilde T_0(H)\in \tilde Y$ for each $H\in \tilde X$ and $\left\|\tilde T_0(H)\right\|_{\tilde Y}\lesssim\left\|H\right\|_{\tilde X}$, see {Lemma \ref{F<H}}. Hence, $\tilde T_0: \tilde X\to \tilde Y$ is a well-defined bounded linear operator. It is also direct to show that
		$H=\tilde T_0(H)+\frac\beta{2\mu}\p_\beta \tilde T_0(H)$.
		We claim that for any $H\in \tilde X$ with $\left\|H\right\|_{\tilde X}\leq 1$, we have the following estimates for $F=\tilde T_0(H)$:
		\begin{equation}\label{5.37}
			\sup_{\beta,\phi}\left|\beta^{2\mu}F(\beta,\phi)\right|\lesssim 1,
		\end{equation}
		\begin{equation}\label{5.38}
			\sup_{\beta,\phi}\left|\p_\beta\Big(\langle
			\beta\rangle^{\alpha-1}\beta^{2\mu}F(\beta,\phi)\Big)\right|\lesssim1.
		\end{equation}
		
		\underline{Estimates \eqref{5.37} and \eqref{5.38} imply the compactness of $\tilde T_0$}. Let $\{H_n\}_{n\geq 1}$ be a sequence of functions in $\tilde X$ with $\|H_n\|_{\tilde X}\leq 1$, and let $F_n=\tilde T_0(H_n)\in\tilde Y$.  We are going to show that $\{F_n\}_{n\geq1}$ has a subsequence converging in $\tilde Y$. By $\alpha\in(0,1)$, \eqref{5.37}, \eqref{5.38} and  \eqref{bernstein2}, we have
		\[{\sup_{\beta,\phi,n}}\left(\left|\langle\beta\rangle^{\alpha-1}\beta^{2\mu}F_n(\beta,\phi)\right|+
		\left|\p_\beta\Big(\langle\beta\rangle^{\alpha-1}\beta^{2\mu}F_n(\beta,\phi)\Big)\right|+
		\left|\p_\phi\Big(\langle\beta\rangle^{\alpha-1}\beta^{2\mu}F_n(\beta,\phi)\Big)\right|\right)\lesssim1.\]
		Hence, using Arzel\`a-Ascoli lemma and Cantor's diagonal arguments, we know that there exist  $F_0\in\tilde Y$ and a subsequence of $\{F_n\}_{n\geq 1}$, still denoted by $\{F_n\}$, such that
		\begin{equation}\label{5.39}
			\langle\beta\rangle^{\alpha-1}\beta^{2\mu}F_n\longrightarrow\langle\beta\rangle^{\alpha-1}\beta^{2\mu}F_0\ \ \text{ in }\ \ L^\infty_{\text{loc}}([0,\infty)\times\TT),
		\end{equation}
		and
		\begin{equation}\label{5.40}
			{\sup_{\beta,\phi}}\left|\beta^{2\mu}F_0(\beta,\phi)\right|\lesssim 1.
		\end{equation}
		We show that the convergence in \eqref{5.39} is indeed in $L^\infty([0,\infty)\times\TT)$. For any $\varepsilon>0$, by $\alpha\in(0,1)$, \eqref{5.37} and \eqref{5.40}, there exists $R>0$ such that
		\begin{equation}\label{5.41}
			\left|\langle\beta\rangle^{\alpha-1}\beta^{2\mu}F_n(\beta,\phi)-\langle\beta\rangle^{\alpha-1}\beta^{2\mu}F_0(\beta,\phi)\right|\leq C\langle\beta\rangle^{\alpha-1}<\varepsilon,\qquad  \forall \ \beta\geq R, \phi\in\TT.
		\end{equation}
		On the other hand, \eqref{5.39} implies that $\langle\beta\rangle^{\alpha-1}\beta^{2\mu}F_n\to\langle\beta\rangle^{\alpha-1}\beta^{2\mu}F_0$ in $L^\infty([0,R]\times\TT)$. Combining this with \eqref{5.41}, there exists $K=K(\varepsilon)>0$ such that $\|F_n-F_0\|_{\tilde Y}\leq 2\varepsilon$ for all $n>K$. Hence, $F_n$ converges to $F_0$ in $\tilde Y$. This proves that $\tilde T_0: \tilde X\to \tilde Y$ is a compact operator.\smallskip
		
		\underline{Proof of \eqref{5.37}}. For simplicity, we omit the variable $\phi\in\TT$ in the rest of this proof. { It follows from $F=\tilde T_0H$ that $\left|\beta^{2\mu}F(\beta)\right|\lesssim\|F\|_{\tilde Y}=\left\|\tilde T_0(H)\right\|_{\tilde Y}\lesssim\left\|H\right\|_{\tilde X} \leq1$ for $\beta\in(0,2)$.
		Then by Lemma \ref{Bernstein} for $f(\phi)=F(1,\phi)$ we have $|\p_\phi F(1)|\lesssim\|F\|_{\tilde Y}\lesssim1 $.} Noting that $\p_\beta(\beta^{2\mu}F)=2\mu(\beta^{2\mu-1}H)$, $\p_\vp\p_\beta(\beta^{2\mu}F)=2\mu\p_\vp(\beta^{2\mu-1}H)$, we have
		\begin{align}\label{5.42}
			&\p_\vp(\beta^{2\mu}F)(\beta)=\p_\vp(\beta^{2\mu}F)(1)+2\mu\int_1^{\beta}\p_\vp(s^{2\mu-1}H)(s)\,ds\\ \notag
			&=\p_\phi F(1)-2\mu H(1)+2\mu\int_1^{\beta} [s^{2\mu-1}\p_\vp H(s)-(2\mu-1)s^{2\mu-2}H(s)]ds,\qquad\beta\geq1.
		\end{align}
		Then we obtain
		\[|\p_\vp(\beta^{2\mu}F)(\beta)|\lesssim\left\|H\right\|_{\tilde X}+\int_1^{\beta}\frac{\left\|H\right\|_{\tilde X}}{s^{\alpha+1}}\,ds\lesssim1.\]
		{Finally, note that $\p_\vp(\beta^{2\mu}F)=\beta^{2\mu}\p_\phi F-\p_\beta(\beta^{2\mu}F)=\beta^{2\mu}\p_\phi F-2\mu\beta^{2\mu-1}H$, using Lemma \ref{Reverse_bernstein} (for $N=1$, $f(\phi)=\beta^{2\mu}F(\beta,\phi)$), we get
		\[\|\beta^{2\mu}F(\beta)\|_{L_{\phi}^{\infty}}\lesssim\|\beta^{2\mu}\p_\phi F(\beta)\|_{L_{\phi}^{\infty}}\lesssim \|\p_\vp(\beta^{2\mu}F)(\beta)\|_{L_{\phi}^{\infty}}+\|H\|_{\tilde X}\lesssim 1,\qquad \beta\geq 1.\]
		This proves \eqref{5.37}.}
		
		\underline{Proof of \eqref{5.38}}. Direct computation gives
		\[\p_\beta\Big(\langle\beta\rangle^{\alpha-1}\beta^{2\mu}F\Big)=(\alpha-1)\langle\beta\rangle^{\alpha-3}\beta^{2\mu+1}F{+}
		2\mu\langle\beta\rangle^{\alpha-1}\beta^{2\mu-1}H,\]
		hence,
		\[\left|\p_\beta\Big(\langle
		\beta\rangle^{\alpha-1}\beta^{2\mu}F(\beta)\Big)\right|\lesssim\|F\|_{\tilde Y}+\|H\|_{\tilde X}\lesssim1.\]
		This completes the proof of \eqref{5.38}.
	\end{proof}
	
	Now we fix a smooth bump function \eqref{bump_function}:
	\begin{equation*}
		\rho\in C^\infty([0,\infty);[0,1])\qquad \text{ such that } \qquad \rho(\beta)=\begin{cases}
			0 & \beta\in[0,1]\\
			1 & \beta\geq 2.
		\end{cases}
	\end{equation*}

	\begin{lem}\label{compact_bounded}
		If $\displaystyle F_1=\sum_{2\leq |n|\leq N_0}\widehat{F_1}_{,n}(\beta)e^{\ii n(\beta+\phi)}$ and $\displaystyle F_2=\sum_{2\leq |n|\leq N_0}\widehat{F_2}_{,n}(\beta)e^{\ii n(\beta+\phi)}$ are such that $\langle\beta\rangle^\alpha\beta^{2\mu-1}F_1, \langle\beta\rangle^\alpha\beta^{2\mu}F_2\in L^\infty$, then there exists a solution $H\in\tilde X$ to the equation
		\begin{equation}\label{T_1T_2_eq}
			\frac1\mu\p_\vp(\beta H_\vp)+\frac\mu\beta H_{\phi\phi}+\frac\gamma\beta\rho(\beta)H=\p_\vp F_1+F_2
		\end{equation}
		with the estimate
		\[\|H\|_{\tilde X}\lesssim\left\|\langle\beta\rangle^\alpha\beta^{2\mu-1}F_1\right\|_{L^\infty}+\left\|\langle\beta\rangle^\alpha\beta^{2\mu}F_2\right\|_{L^\infty}.\]
	\end{lem}
	
	This lemma is proved by expanding the solution in finitely many Fourier modes. Therefore, we need to establish some existence results for a class of inhomogeneous second-order ODEs with singular points:
	\[x^2y''(x)+xy'(x)-q(x)y(x)=f,\quad x>0.\]
	These results are collected in Appendix \ref{AppendixB}. We will use Lemma \ref{prop_B2}, where the force $f$ lies in a weighted $L^\infty$ space, and Lemma \ref{prop_B3},  where the force $f=x\wt f'$ for some $\wt f$ in a weighted $L^\infty$ space. For more details, readers may refer to Appendix \ref{AppendixB}.
		
	\begin{proof}
		We write $\displaystyle H=\sum_{2\leq |n|\leq N_0}\hat{H}_{n}(\beta)e^{\ii n(\beta+\phi)}$. Using $\gamma=2-\frac1\mu$, the linear equation \eqref{T_1T_2_eq} can be rewritten as the following ODE for $2\leq |n|\leq N_0$:
		\begin{equation}\label{5.45}
			(\beta\p_\beta)^2 \hat H_n-\mu^2n^2\hat H_n+(2\mu-1)\rho(\beta)\hat H_n=-\mu\beta\p_\beta \widehat{F_1}_{,n}+\mu\beta \widehat{F_2}_{,n}, \ \ \beta>0.
		\end{equation}
		To construct a solution $H$ to \eqref{T_1T_2_eq}, it suffices to prove the existence of each $\hat H_n$ to the ODE \eqref{5.45} for $2\leq |n|\leq N_0$. We borrow some results from the standard ODE theory, see Appendix \ref{AppendixB}. Let $q_n=\mu^2n^2-(2\mu-1)\rho$ for $2\leq |n|\leq N_0$, then $q_n\in C^\infty([0,\infty); (0,\infty))$, $q_n(\beta)=q_{n,1}^2$ for $\beta\in[0,1]$ and $q_n(\beta)=q_{n,2}^2$ for $\beta\geq 2$, where
		\[q_{n,1}=\mu|n|,\qquad q_{n,2}=\sqrt{\mu^2n^2-(2\mu-1)}.\]
		In order to apply Lemma \ref{prop_B2} and Lemma \ref{prop_B3}, we need to check the relation \eqref{alpha<1/2}. Since $|n|\geq 2$ and $\mu>1/2$, we have $q_{n,1}-2\mu+1>0$. Thanks to $\alpha\in\left(0,\alpha_\mu\right)$, we also have
		\[q_{n,2}-2\mu-\alpha+1\geq \sqrt{4\mu^2-2\mu+1}-(2\mu-1)-\alpha=\alpha_\mu-\alpha>0.\]
		Now, Lemma \ref{prop_B2} and Lemma \ref{prop_B3} are applicable. Using them, we know that there exists a solution $\hat H_n$ to \eqref{5.45} such that
		\begin{align*}
			&\ \ \ \left\|\langle\beta\rangle^\alpha\beta^{2\mu}\p_\beta\hat H_n\right\|_{L^\infty((0,\infty))}+\left\|\langle\beta\rangle^\alpha\beta^{2\mu-1}\hat H_n\right\|_{L^\infty((0,\infty))}\\
			&\lesssim_n \left\|\langle\beta\rangle^\alpha\beta^{2\mu-1}\widehat{F_1}_{,n}\right\|_{L^\infty((0,\infty))}+
			\left\|\langle\beta\rangle^\alpha\beta^{2\mu-1}\beta\widehat{F_2}_{,n}\right\|_{L^\infty((0,\infty))}\\
			&\lesssim_n \left\|\langle\beta\rangle^\alpha\beta^{2\mu-1}F_1\right\|_{L^\infty}+\left\|\langle\beta\rangle^\alpha\beta^{2\mu}F_2\right\|_{L^\infty}.
		\end{align*}
		Therefore, with $\hat H_n$ constructed above, the function $\displaystyle H=\sum_{2\leq |n|\leq N_0}\hat{H}_{n}(\beta)e^{\ii n(\beta+\phi)}$ is a solution to \eqref{T_1T_2_eq} satisfying the estimate (recall that $N_0$ is a fixed integer)
		\begin{align*}
			\left\|H\right\|_{\tilde X}&=\left\|\langle\beta\rangle^\alpha\beta^{2\mu}\p_\vp H\right\|_{L^\infty}+\left\|\langle\beta\rangle^\alpha\beta^{2\mu-1}H\right\|_{L^\infty}\\
			&\lesssim\sum_{2\leq |n|\leq N_0}\left(\left\|\langle\beta\rangle^\alpha\beta^{2\mu}\p_\beta\hat H_n\right\|_{L^\infty((0,\infty))}+\left\|\langle\beta\rangle^\alpha\beta^{2\mu-1}\hat H_n\right\|_{L^\infty((0,\infty))}\right)\\
			&\lesssim \left\|\langle\beta\rangle^\alpha\beta^{2\mu-1}F_1\right\|_{L^\infty}+\left\|\langle\beta\rangle^\alpha\beta^{2\mu}F_2\right\|_{L^\infty}.
		\end{align*}
		Here we remark that the implicit constants depend on $q_n$ by Appendix \ref{AppendixB}; nevertheless, as $N_0$ is fixed and finite, we can take a uniform implicit constant.
		This completes the proof of Lemma \ref{compact_bounded}.
	\end{proof}
	
	Now we are in a position to prove Proposition \ref{low_frequency}.
	
	\begin{proof}[Proof of Proposition \ref{low_frequency}]
		\textbf{Step 1.} In this step, we first construct a solution $(Q, \psi)$ to \eqref{exist_2_Q} in the larger space $\tilde X\times \tilde Y$. Using $\p_\vp=\p_\phi-\p_\beta$ and the second equation in \eqref{exist_2_Q}, we get
		\begin{align*}
			\frac\gamma{2\mu}\psi_\phi&=\frac\gamma{2\mu}\left(1-\rho(\beta)\right)\psi_\phi+\frac\gamma{2\mu}\rho(\beta)(\psi_\vp+\psi_\beta)\\
			&=\frac\gamma{2\mu}\left(1-\rho(\beta)\right)\psi_\phi+\frac\gamma{2\mu}\rho(\beta)\psi_\vp+\frac\gamma\beta\rho(\beta)(H+Q-\psi)\\
			&=\frac\gamma\beta\rho(\beta)Q+\frac\gamma\beta\rho(\beta)H+\frac\gamma{2\mu}\p_\vp(\rho(\beta)\psi)+
			\frac\gamma{2\mu}\left(\rho'(\beta)-\frac{2\mu}\beta\rho(\beta)\right)\psi+\frac\gamma{2\mu}\left(1-\rho(\beta)\right)\psi_\phi.
		\end{align*}
		This motivates us to consider the following linear system
		\begin{equation}\label{5.46}
			\begin{cases}
				\frac1\mu\p_\vp(\beta Q_\vp)+\frac\mu\beta Q_{\phi\phi}+\frac\gamma\beta\rho(\beta)Q\\
				\qquad=-\frac\gamma\beta\rho(\beta)H-\frac\gamma{2\mu}\p_\vp(\rho(\beta)\psi)-
				\frac\gamma{2\mu}\left(\rho'(\beta)-\frac{2\mu}\beta\rho(\beta)\right)\psi-\frac\gamma{2\mu}\left(1-\rho(\beta)\right)\psi_\phi\\
				H+Q=\psi+\frac{\beta}{2\mu}\psi_\beta.
			\end{cases}
		\end{equation}
		It follows from $\displaystyle H=\sum_{2\leq |n|\leq N_0}\hat H_n(\beta)e^{\ii n(\beta+\phi)}\in X_0$ that $H\in\tilde X$. By Lemma \ref{compact_bounded}, for any $\psi\in \tilde Y$, we can find $Q_1, Q_2\in \tilde X$ satisfying
		\begin{equation*}
			\begin{aligned}
				&\frac1\mu\p_\vp(\beta \p_\vp Q_1)+\frac\mu\beta \p_\phi^2Q_{1}+\frac\gamma\beta\rho(\beta)Q_1\\
				&\ \ =-\frac\gamma{2\mu}\p_\vp(\rho(\beta)\psi)-\frac\gamma{2\mu}\left(\rho'(\beta)-\frac{2\mu}\beta\rho(\beta)\right)\psi-
				\frac\gamma{2\mu}\left(1-\rho(\beta)\right)\psi_\phi
			\end{aligned}
		\end{equation*}
		and
		\begin{equation*}
			\frac1\mu\p_\vp(\beta \p_\vp Q_2)+\frac\mu\beta \p_\phi^2Q_{2}+\frac\gamma\beta\rho(\beta)Q_2=-\frac\gamma\beta\rho(\beta)H,
		\end{equation*}
		with $\|Q_2\|_{\tilde X}\lesssim \left\|\langle\beta\rangle^\alpha\beta^{2\mu}\frac{\rho(\beta)}\beta H\right\|_{L^\infty}\lesssim\|H\|_{\tilde X}$ and
		\begin{align*}
			\|Q_1\|_{\tilde X}&\lesssim \left\|\langle\beta\rangle^\alpha\beta^{2\mu-1}\rho(\beta)\psi\right\|_{L^\infty}+
			\left\|\langle\beta\rangle^\alpha\beta^{2\mu}\left(\rho'(\beta)-\frac{2\mu}\beta\rho(\beta)\right)\psi\right\|_{L^\infty}\\
			&\qquad\qquad+\left\|\langle\beta\rangle^\alpha\beta^{2\mu}\left(1-\rho(\beta)\right)\psi_\phi\right\|_{L^\infty}\\
			&
\lesssim\left\|\langle\beta\rangle^{\alpha-1}\beta^{2\mu}\psi\right\|_{L^\infty}\lesssim \|\psi\|_{\tilde Y}.
		\end{align*}
	Here we used \eqref{bernstein2} for $g=\langle\beta\rangle^\alpha\beta^{2\mu}\left(1-\rho(\beta)\right)\psi$ and
	\begin{align*}
				\beta^{2\mu-1}|\rho(\beta)|+\beta^{2\mu}\left|\rho'(\beta)-\frac{2\mu}\beta\rho(\beta)\right|+\beta^{2\mu}\left|1-\rho(\beta)\right|
				\lesssim\beta^{2\mu}\langle\beta\rangle^{-1}.
		\end{align*}
		
		We denote the solution map from $\psi$ to $Q_1$ by $T_1: \tilde Y\to \tilde X$, and denote the solution map from $H$ to $Q_2$ by $T_2: \tilde X\to \tilde X$. Both of them are bounded linear operators. To solve the system \eqref{5.46} in $\tilde X\times\tilde Y$, we only need to find $(Q,\psi)\in \tilde X\times\tilde Y$ such that
		\[Q=T_1\psi+T_2H,\qquad \psi=\tilde T_0(H+Q),\]
		i.e.,
		\begin{equation*}
			\left(I-T_1\tilde T_0\right)(H+Q)=T_2H+H.
		\end{equation*}
		It follows from Lemma \ref{compact} and the boundedness of $T_1: \tilde Y\to \tilde X$ that
		\[T_1\tilde T_0: \tilde X\to\tilde X\quad \text{is a compact operator}.\]
		We claim that $I-T_1\tilde T_0: \tilde X\to\tilde X$ is an injection. Indeed, if $Q\in \tilde X$ satisfies $T_1\tilde T_0Q=Q$, then letting $\psi=\tilde T_0Q\in\tilde Y$ we have
		\begin{equation}\label{5.50}
			\begin{cases}
				\frac1\mu\p_\vp(\beta Q_\vp)+\frac{\mu Q_{\phi\phi}}{\beta}+\frac\gamma{2\mu}\psi_\phi=0,\\
				Q=\psi+\frac\beta{2\mu}\psi_\beta.
			\end{cases}
		\end{equation}
		Now, using the same ideas as in the proof of uniqueness part of Proposition \ref{linear_thm}: expanding $Q$ and $\psi$ in Fourier series (here the series is a finite summation) and then by Lemma \ref{uniqueness_ODE}, we can show that the system \eqref{5.50} only has the trivial solution $(Q, \psi)=(0,0)$ in $\tilde X\times\tilde Y$. Hence, $I-T_1\tilde T_0: \tilde X\to\tilde X$ is injective.
		
		By Fredholm's theory, $I-T_1\tilde T_0: \tilde X\to\tilde X$ is a bijection and has a bounded inverse. As a result,
		\[Q:=\left(I-T_1\tilde T_0\right)^{-1}(T_2H+H)-H\in \tilde X, \qquad \psi:=\tilde T_0(H+Q)\in\tilde Y\]
		solve the system \eqref{exist_2_Q} and
		\[\|Q\|_{\tilde X}+\|\psi\|_{\tilde Y}\lesssim \|H\|_{\tilde X_0}\lesssim\|H\|_{X_0}.\hypertarget{step2}{}\]
		
		\noindent\textbf{Step 2.} In this step, we show that the solution $Q\in \tilde X$ to \eqref{exist_2_Q} given
		\beno
		\displaystyle H=\sum_{2\leq |n|\leq N_0}\hat H_n(\beta)e^{\ii n(\beta+\phi)}\in X_0
		\eeno
		(so $H\in\tilde X$)  constructed in the previous step exactly belongs to $X_0$, the smaller space requiring the $C^\alpha$ regularity.

		By \eqref{f1} and $\partial_{\beta}Q=\partial_{\phi}Q-\partial_{\varphi}Q$, we have\begin{align}\label{f2}
			&\left\|\beta^{2\mu-1}Q\right\|_{C_\beta^\alpha}\lesssim \left\|\langle\beta\rangle^{\alpha}\beta^{2\mu-1}Q\right\|_{L^{\infty}}+
			\left\|\langle\beta\rangle^{\alpha-1}\beta\partial_{\beta}(\beta^{2\mu-1}Q)\right\|_{L^{\infty}}\\ \notag&\lesssim \left\|\langle\beta\rangle^{\alpha}\beta^{2\mu-1}Q\right\|_{L^{\infty}}+
			\left\|\langle\beta\rangle^{\alpha-1}\beta^{2\mu}\partial_{\beta}Q\right\|_{L^{\infty}}\\ \notag&\lesssim \left\|\langle\beta\rangle^{\alpha}\beta^{2\mu-1}Q\right\|_{L^{\infty}}+
			\left\|\langle\beta\rangle^{\alpha-1}\beta^{2\mu}\partial_{\phi}Q\right\|_{L^{\infty}}+
			\left\|\langle\beta\rangle^{\alpha-1}\beta^{2\mu}\partial_{\varphi}Q\right\|_{L^{\infty}}\\ \notag
			&\lesssim \left\|\langle\beta\rangle^{\alpha}\beta^{2\mu-1}Q\right\|_{L^{\infty}}+
			\left\|\langle\beta\rangle^{\alpha-1}\beta^{2\mu}\partial_{\varphi}Q\right\|_{L^{\infty}} \lesssim\|Q\|_{\tilde X}\lesssim\|H\|_{X_0}.
		\end{align}
		Here we have used \eqref{bernstein2} so that
		\beno
		\left\|\langle\beta\rangle^{\alpha-1}\beta^{2\mu}\partial_{\phi}Q\right\|_{L^{\infty}}\lesssim
		\left\|\langle\beta\rangle^{\alpha-1}\beta^{2\mu}Q\right\|_{L^{\infty}}\leq \left\|\langle\beta\rangle^{\alpha}\beta^{2\mu-1}Q\right\|_{L^{\infty}},
		\eeno
		and we also used $\langle\beta\rangle^{\alpha-1}\leq \langle\beta\rangle^{\alpha} $ and the definition of the $\tilde X$ norm.
		
		By \eqref{bernstein2}, we also have
		\beno
		\left\|\beta^{2\mu-1}\partial_{\phi}Q\right\|_{C_\beta^\alpha}\lesssim\left\|\beta^{2\mu-1}Q\right\|_{C_\beta^\alpha}\lesssim \|H\|_{X_0}.
		\eeno
		
		Similarly, by \eqref{f2} with $Q$ replaced by $\beta\p_\vp Q $, \eqref{exist_2_Q} and \eqref{bernstein2}, we have
		\begin{align*}
			&\left\|\beta^{2\mu}\p_\vp Q\right\|_{C_\beta^\alpha}\lesssim  \left\|\langle\beta\rangle^{\alpha}\beta^{2\mu}\p_\vp Q\right\|_{L^{\infty}}+
			\left\|\langle\beta\rangle^{\alpha-1}\beta^{2\mu}\partial_{\varphi}(\beta\p_\vp Q)\right\|_{L^{\infty}} \\
			&\lesssim \|Q\|_{\tilde X}+\left\|\langle\beta\rangle^{\alpha-1}\beta^{2\mu-1}\partial_{\phi}^2 Q\right\|_{L^{\infty}}+\left\|\langle\beta\rangle^{\alpha-1}\beta^{2\mu}\partial_{\phi}\psi\right\|_{L^{\infty}}\\
			&\lesssim \|Q\|_{\tilde X}+\left\|\langle\beta\rangle^{\alpha}\beta^{2\mu-1} Q\right\|_{L^{\infty}}+\left\|\langle\beta\rangle^{\alpha-1}\beta^{2\mu}\psi\right\|_{L^{\infty}}\lesssim\|Q\|_{\tilde X}+\|\psi\|_{\tilde Y}\lesssim\|H\|_{X_0}.
		\end{align*}
		Summing up, we have $Q\in X_0$ with
		\begin{align*}
			\|Q\|_{X_0}&=\left\|\beta^{\alpha+2\mu}\p_\vp Q\right\|_{C_\beta^\alpha}+\left\|\beta^{\alpha+2\mu-1}\p_\phi Q\right\|_{C_\beta^\alpha}+\left\|\beta^{\alpha+2\mu-1}Q\right\|_{C_\beta^\alpha}\lesssim\|H\|_{X_0}.
		\end{align*}
		
		Finally, Lemma \ref{F<H} implies that $\psi=\tilde T_0(H+Q)\in Y_0$ and $	\|\psi\|_{Y_0}\lesssim\|H+Q\|_{X_0}	\leq \|H\|_{X_0}+\|Q\|_{X_0}\lesssim\|H\|_{X_0}$. 	
		\end{proof}

	\section{Nonlinear problem in the original coordinates}\label{sec_original_coor}
	This section is devoted to recovering the solution in the physical variables $x\in \RR^2$. First of all, we study the invertibility of the change of variables $x\mapsto\bbb=(\beta,\phi)$, which is a nonlinear implicit change of variables. Later on, we check that the solution in the physical variables is a weak solution to Euler equations and finish the proof of our main theorem.  In this section and Section \ref{sec_main_cor}, since we are taking estimates on the solutions constructed before, the implicit constants in $\lesssim$, unless otherwise specified, all depend on the solutions and all parameters (including $m$), and they are independent of the variables $(t, y)$, $(r,\theta)$ and $(\beta, \phi)$.
		
	\subsection{Invertibility of the change of coordinates}\label{sec_invertibility}
	For $m\geq 2$, $\mu>1/2$ and $\alpha\in(0,\alpha_\mu)$ where $\alpha_\mu$ is given by \eqref{alpha_mu}, Theorem \ref{IFT} gives a $C^\infty$ map $\Xi: B_{\varepsilon_\Omega}^{(\WWW_m)}(\Omega_0)\to B_{\varepsilon_\psi}^{(Y_m)}(\psi_0)$, where $\varepsilon_\psi, \varepsilon_\Omega>0$ are independent of $m\geq2$, such that the unique solution of $\FF(\psi,\Omega)=0$ in $B_{\varepsilon_\psi}^{(Y_m)}(\psi_0)\times  B_{\varepsilon_\Omega}^{(\WWW_m)}(\Omega_0)$ is $\psi=\Xi(\Omega)$. Using the function $\psi$, we define the change of coordinates $\RR^2\ni x\mapsto\aaa=(r,\theta)\mapsto \bbb=(\beta,\phi)$ in the beginning of Section \ref{sec_formulation}. Now we check that this change of coordinates has all properties we expected and is invertible.
	
	By Lemma \ref{lem_nonlinear}, we have
	\[\beta^{2\mu-1}\psi\in \GG_m^-, \ \ \beta^{2\mu}\psi_\beta\in\GG_m^-,\ \  \beta^{2\mu-1}\psi_\phi\in\GG_m^0, \ \ \beta^{2\mu}\psi_{\beta\phi}\in\GG_m^0,\]
	\[\beta^{2\mu}\left(\psi_\vp+\frac\beta{2\mu}\psi_{\beta\vp}\right)\in\GG_m^0,\ \  \beta^{2\mu}\psi_\vp\in\GG_m, \ \ \beta^{2\mu+1}\psi_{\beta\vp}\in\GG_m.\]
	Due to $\p_\beta=\p_\phi-\p_\vp$, we obtain
	\[\psi_{\beta\beta}, \psi_{\beta\phi}, \psi_\phi\in C_{(\beta,\phi)}(\RR_+\times\TT)\ \ \Longrightarrow \ \ \psi_\beta, \psi\in C^1_{(\beta,\phi)}(\RR_+\times\TT).\]
	Here $C_{(\beta,\phi)}(\RR_+\times\TT)$ denotes the continuous function in terms of variable $(\beta,\phi)$.
	For the special solution $\psi_0=\frac{1}{2\mu-1}\beta^{1-2\mu}$, we have
	\[\beta^{2\mu}\p_\beta\psi_0=-1<0, \ \ \beta^{2\mu}\p_\vp\psi_0=1>0, \ \  \beta^{2\mu+1}\p_{\beta\vp}\psi_0=-2\mu<0.\]
	By the definitions of $\GG_m^0$, $\GG_m^-$ and $\GG_m$, the following embeddings are independent of $m\in\NN_+$:
	\[\GG_m^0, \GG_m^-, \GG_m\hookrightarrow C_b(\RR_+\times \TT), \qquad \langle\beta\rangle^\alpha \GG_m^0\hookrightarrow C_b(\RR_+\times \TT).\]
	Therefore, by adjusting $\varepsilon_\Om>0$ to a smaller number if necessary, for the solution $\psi=\Xi(\Omega)$ with $\Om\in B_{\varepsilon_\Omega}^{(\WWW_m)}(\Omega_0)$, we have
	\begin{equation}
		\beta^{2\mu}\psi_\beta<0,\ \ \beta^{2\mu}\psi_\vp>0, \ \ \beta^{2\mu+1}\psi_{\beta\vp}<0,
	\end{equation}
	and
	\begin{equation}\label{5.2}
		-\psi_\beta\sim \beta^{-2\mu},\ \ \psi_\vp\sim\beta^{-2\mu},\ \ -\psi_{\beta\vp}\sim \beta^{-2\mu-1}.
	\end{equation}
	It follows from \eqref{beta-x} that
	\begin{equation}\label{r_beta^-mu}
		r=\left(-\frac1\mu\psi_\beta\right)^{1/2}>0 \qquad\text{ and } \qquad r\sim \beta^{-\mu},
	\end{equation}
	thus the change of variable $\bbb=(\beta,\phi)\mapsto \aaa=(r,\theta)$ is $C^1_{(\beta,\phi)}(\RR_+\times\TT)$. By \eqref{beta_to_r}, the Jacobian of $\bbb\mapsto \aaa$ is
	\[\det\aaa_\bbb=\det\begin{pmatrix}
		r_\beta & r_\phi \\ \theta_\beta & \theta_\phi
	\end{pmatrix}=-r_\vp=\frac{\psi_{\beta\vp}}{2\mu r}<0,\]
	hence $\bbb\mapsto \aaa$  is a $C^1$ local diffeomorphism. As for the surjectivity of $\bbb\mapsto \aaa$, we refer  to section 7.4 in \cite{ELL2013}, since the proof is identical. Therefore, the map $(\beta,\phi)\mapsto (r,\theta)$ is a surjective, hence a $C^1$ diffeomorphism.  The transform $(r,\theta)\mapsto x\in\RR^2\setminus\{0\}$ is also a $C^1$ diffeomorphism modulo periodicity.

	\subsection{Properties of the solution in the physical variables}\label{sec_phycical_variable}
	
	In this subsection, we explore the properties of the solution constructed in Theorem \ref{IFT}  in the physical variables.\smallskip
	
	We start with the properties of $\psi=\psi(x)$ with the polar coordinates $x=(r\cos\theta, r\sin\theta)$. It follows from $\beta^{2\mu-1}\psi\in\GG_m^-$ and \eqref{r_beta^-mu} that $|\psi(r,\theta)|\lesssim r^{2-\frac1\mu}$. By \eqref{p_rtheta}, we have
	\begin{align*}
		\p_\theta \psi&=\psi_\phi-\frac{\psi_{\beta\phi}}{\psi_{\beta\vp}}\psi_\vp=\beta^{1-2\mu}\langle\beta\rangle^{-\alpha}
		\left(\langle\beta\rangle^{\alpha}\beta^{2\mu-1}\psi_\phi-
		\frac{\langle\beta\rangle^{\alpha}\beta^{2\mu}\psi_{\beta\phi}}{\beta^{2\mu+1}\psi_{\beta\vp}}\beta^{2\mu}\psi_\vp\right),\\
		r\p_r\psi&=\frac{2\psi_\beta}{\psi_{\beta\vp}}\psi_\vp=\frac{2\psi_\beta\left(\psi_\vp+\frac\beta{2\mu}\psi_{\beta\vp}\right)}{\psi_{\beta\vp}}-
		\frac\beta\mu\psi_\beta\\
		&=2\frac{\beta^{2\mu}\psi_\beta\cdot \langle\beta\rangle^{\alpha}\beta^{2\mu}\left(\psi_\vp+\frac\beta{2\mu}\psi_{\beta\vp}\right)}{\beta^{2\mu+1}\psi_{\beta\vp}}\beta^{1-2\mu}
		\langle\beta\rangle^{-\alpha}-
		\frac{\beta^{2\mu}\psi_\beta}\mu\beta^{1-2\mu}.
	\end{align*}
	Lemma \ref{lem_nonlinear}, \eqref{5.2} and \eqref{r_beta^-mu} imply that
	\begin{equation}\label{5.4}
		\left|\p_\theta\psi\right|\lesssim \min(r^{2+\frac{\alpha-1}{\mu}},r^{2-\frac{1}{\mu}}),\qquad \left|\p_r\psi\right|\lesssim r^{1-\frac1\mu}.
	\end{equation}
	Moreover, since $\beta^{2\mu}\psi_\beta\in\GG_m^-$, there exists a constant $c_0$ such that
	\[\left|r\p_r\psi(r,\theta)-c_0\beta(r,\theta)^{1-2\mu}\right| {{\lesssim\beta(r,\theta)^{1-\alpha-2\mu}}}\lesssim r^{2+\frac{\alpha-1}{\mu}}.\]
	We denote $A(r,\theta)=\beta(r,\theta)^{1-2\mu}$ and
	\begin{align*}
		A_0(r)&=P_0A(r):=\frac1{2\pi}\int_\TT A(r,\theta)\,d\theta, \\ A_{\neq}(r,\theta)&=A(r,\theta)-A_0(r)=\frac1{2\pi}\int_\TT [A(r,\theta)-A(r,\theta')]\,d\theta',
	\end{align*}
	then by \eqref{r_beta^-mu}, we have $|A_0(r)|\lesssim r^{2-\frac1\mu}$. Recalling \eqref{r_to_beta}, we compute
	\[\p_\theta A(r,\theta)=(1-2\mu)\beta^{-2\mu}\beta_\theta=(1-2\mu)\beta^{1-\alpha-2\mu}\frac{\beta^{\alpha+2\mu}\psi_{\beta\phi}}{\beta^{2\mu+1}\psi_{\beta\vp}},\]
	hence $\left|\p_\theta A(r,\theta)\right|\lesssim r^{2+\frac{\alpha-1}{\mu}}$ and then we have $|A_{\neq}(r,\theta)|\lesssim r^{2+\frac{\alpha-1}{\mu}}$. As a result, for $v=v^re_r+v^\theta e_\theta=\nabla^\perp \psi=-\frac1r\p_\theta\psi e_r+\p_r\psi e_\theta$, we have $v^\theta(r,\theta)=v^\theta_1(r,\theta)+v^\theta_2(r)$ with {($v^\theta_2(r)=c_0A_0(r)/r$)}
	\begin{equation}\label{v_bounds}
		|v(r,\theta)|\lesssim r^{1-\frac1\mu}, \qquad \left|v^r(r,\theta)\right|+\left|v^\theta_1(r,\theta)\right|\lesssim r^{1+\frac{\alpha-1}{\mu}}, \qquad \left|v^\theta_2(r)\right|\lesssim r^{1-\frac1\mu}.
	\end{equation}
	Since $\mu>\frac12$, we know that $v(x)\in L^2_{\text{loc}}(\RR^2)$.
	
	For the vorticity $\omega$, it follows from \eqref{omega_Omega} and \eqref{r_beta^-mu} that
	\[|\omega(r,\theta)|=\left(\beta^{2\mu}\psi_\vp(\beta,\phi)\right)^{-\frac1{2\mu}}\beta|\Omega(\phi)|\lesssim r^{-\frac1\mu}|\Omega(\phi)|.\]
	For any fixed $R>0$, using the change of variables $x\mapsto(r,\theta)\mapsto(\beta,\phi)$, we obtain
	\begin{equation}\label{5.6}
		\begin{aligned}
			&\int_{|x|\leq R}|\omega(x)|\,dx\lesssim\int_0^R\int_\TT r^{1-\frac1\mu}|\Omega(\phi(r,\theta))|\,dr\,d\theta\\
			&\lesssim\int_{(CR)^{-\frac1\mu}}^\infty\int_\TT \beta^{1-\mu}|\Omega(\phi)|\left|\det\aaa_\bbb\right|\,d\beta\,d\phi \lesssim \int_{(CR)^{-\frac1\mu}}^\infty\int_\TT \beta^{1-\mu}|\Omega(\phi)||\psi_{\beta\varphi}/r|\,d\beta\,d\phi
			\\
			&\lesssim\int_{(CR)^{-\frac1\mu}}^\infty(\beta^{1-\mu}\beta^{-2\mu-1}/\beta^{-\mu})\,d\beta\|\Omega\|_{L^1(\TT)}\lesssim R^{2-\frac1\mu}\|\Omega\|_{L^1(\TT)}.
		\end{aligned}
	\end{equation}
	Therefore, $\omega(x)\in L^1_{\text{loc}}(\RR^2)$.
	
	Recall the self-similar change of variables \eqref{self-similar}: with $x=t^{-\mu}y$
	\[\vvv(y,t)=t^{\mu-1}v(x),\qquad \www(y,t)=t^{-1}\omega(x).\]
	Hence, $\vvv\in C((0,\infty); L^2_{\text{loc}}(\RR^2; \RR^2))$ and $\www\in C((0,\infty); L^1_{\text{loc}}(\RR^2))$.
	
	\subsection{Weak solution of the Euler equations}
	We show that $\vvv(y,t)$ is actually a weak solution to the 2-D Euler equations \eqref{2DEuler}.\smallskip

	The same arguments as in section 4 in \cite{ELL2016} based on the change of variables show that $v$ solves the equation  weakly outside the origin:
	\begin{equation}\label{5.7}
		\nabla\times\Big((\mu-1)v+v\cdot\nabla v-\mu x\cdot\nabla v\Big)=0 \qquad \text{ in } \mathcal{D}'\left(\RR^2\setminus\{0\}\right),
	\end{equation}
	and $\vvv$ solves \eqref{2DEuler} weakly outside the space-time origin. It remains to show that $\vvv$ is a weak solution on $\RR^2\times[0,\infty)$ (including the origin).\smallskip
	
	We first show that the equation for $v$ holds weakly in the whole plane $\RR^2$ (including the origin). Recalling from \eqref{alpha_mu}, we have
	\[\alpha_\mu=\sqrt{4\mu^2-2\mu+1}-(2\mu-1)>1-\mu \text{ and } \alpha_\mu>\frac12>0,\qquad \text{ for  } \mu>\frac12.\]
	
	\begin{prop}\label{prop_v_weak}
		Assume that $\mu>\frac12$ and $\alpha\in(\max\{0, 1-\mu\}, \alpha_\mu)$. For any vector field $w\in C_c^\infty(\RR^2; \RR^2)$ with $\text{div }w=0$, there holds
		\begin{equation}\label{v_weaksol}
			\int_{\RR^2}(3\mu-1)v\cdot w-\left(v\otimes v\right): \nabla w+\mu v\cdot (x\cdot \nabla w)\,dx=0.
		\end{equation}
	\end{prop}
	\begin{proof}
		Since $\text{div }w=0$, there exists a scalar function $\eta\in C_c^\infty(\RR^2)$ such that $w=\nabla^\perp \eta$. Let $\rho\in C_c^\infty(\RR)$ be a smooth bump function satisfying $\rho|_{[-1,1]}\equiv 1$ and $\text{supp }\rho{\ \subset}(-2,2)$. For any $\delta>0$, we define $\rho_\delta\in C_c^\infty(\RR^2)$ by $\rho_\delta(x)=\rho\left(\frac{|x|}{\delta}\right)$ for $x\in\RR^2$. It follows from \eqref{5.7} that
		\begin{align*}
			\int_{\RR^2}(3\mu-1)v\cdot\nabla^\perp\left(\eta(1-\rho_\delta)\right)&-\left(v\otimes v\right): \nabla \nabla^\perp\left(\eta(1-\rho_\delta)\right)\\
			&+\mu v\cdot (x\cdot \nabla \nabla^\perp\left(\eta(1-\rho_\delta)\right))\,dx=0.
		\end{align*}
		As a consequence, to show \eqref{v_weaksol}, it suffices to show that
		\begin{equation}
			\begin{aligned}
				\lim_{\delta\to0+}\int_{\RR^2}(3\mu-1)v\cdot\nabla^\perp\left(\eta\rho_\delta\right)-\left(v\otimes v\right): \nabla \nabla^\perp\left(\eta\rho_\delta\right)+\mu v\cdot (x\cdot \nabla \nabla^\perp\left(\eta\rho_\delta\right))\,dx=0.
			\end{aligned}
		\end{equation}
		By \eqref{v_bounds}, as $\delta\to0+$, we have
		\begin{align*}
			\left|\int_{\RR^2}v\cdot\nabla^\perp\left(\eta\rho_\delta\right)\,dx\right|&\lesssim\int_0^{2\delta}r\cdot r^{1-\frac1\mu}\left(1+\frac1\delta\right)\,dr\lesssim\delta^{2-\frac1\mu}\to0,\\
			\left|\int_{\RR^2}v\cdot (x\cdot \nabla \nabla^\perp\left(\eta\rho_\delta\right))\,dx\right|&\lesssim\int_0^{2\delta} r\cdot r^{1-\frac1\mu}\cdot r\cdot\frac1{\delta^2}\,dr\lesssim \delta^{2-\frac1\mu}\to0.
		\end{align*}
		Here the implicit constants in all $\lesssim$  are independent of $\delta>0$. The remaining term is more delicate and we need to explore some cancellations. We decompose $\eta$ as $\eta(x)=\eta(0)+\p_1\eta(0)x_1+\p_2\eta(0)x_2+\eta_2(x)$. Then $\eta_2$ is smooth and
		\begin{equation*}
			\frac{|\eta_2(x)|}{|x|^2}+\frac{|\nabla\eta_2(x)|}{|x|}+\left|\nabla^2\eta_2(x)\right|\lesssim1,\qquad |x|\leq 1.
		\end{equation*}
		Hence,
		\begin{align*}
			\left|\int_{\RR^2}\left(v\otimes v\right): \nabla \nabla^\perp\left(\eta_2\rho_\delta\right)\,dx\right|\lesssim \int_0^{2\delta}r\cdot r^{2-\frac2\mu}\left(1+\frac r\delta+\frac{r^2}{\delta^2}\right)\,dr\lesssim\delta^{4-\frac2\mu}\to0.
		\end{align*}
		Now it remains to show that
		\begin{align}
			\lim_{\delta\to0+}\int_{\RR^2}\left(v\otimes v\right): \nabla \nabla^\perp\rho_\delta\,dx&=0, \label{5.10}\\
			\lim_{\delta\to0+}\int_{\RR^2}\left(v\otimes v\right): \nabla \nabla^\perp\left(x_i \rho_\delta\right)\,dx&=0 \ \ (i=1,2). \label{5.11}
		\end{align}
		
		\underline{Proof of \eqref{5.10}.} Direct computation gives
		$$\left(v\otimes v\right): \nabla \nabla^\perp\rho_\delta=\frac1{\delta^2}v^rv^\theta\left(\rho''\left(\frac r\delta\right)-\frac\delta r\rho'\left(\frac r\delta\right)\right).$$
		Note that $v^r=-\frac1r\p_\theta\psi$, so $\int_\TT v^r(r,\theta)\,d\theta=0$. By \eqref{v_bounds}, we obtain
		\begin{align*}
			\left|\int_{\RR^2}\left(v\otimes v\right): \nabla \nabla^\perp\rho_\delta\,dx\right|&=\frac1{\delta^2}\Big|\int_\delta^{2\delta}r\int_\TT v^r(r,\theta)\left(v^\theta_1(r,\theta)+v^\theta_2(r)\right)\,d\theta\\
			&\qquad\times \left(\rho''\left(\frac r\delta\right)-\frac\delta r\rho'\left(\frac r\delta\right)\right)\,dr\Big|\\
			&\lesssim \frac1{\delta^2}\int_\delta^{2\delta}r\int_\TT\left|v^r(r,\theta)v^\theta_1(r,\theta)\right|\,d\theta\,dr\\
			&\lesssim \frac1{\delta^2}\int_\delta^{2\delta}r\cdot r^{1+\frac{\alpha-1}{\mu}}\cdot r^{1+\frac{\alpha-1}{\mu}}\,dr\lesssim \delta^{2+\frac{2\alpha-2}{\mu}}\to0,
		\end{align*}
		due to $\alpha>1-\mu$.
		
		\underline{Proof of \eqref{5.11}.} We only prove the limit for $i=1$, since the proof of $i=2$ is the same. Direct computation gives
		\begin{align*}
			\left(v\otimes v\right): \nabla \nabla^\perp\left(x_1 \rho_\delta\right)&=\frac1\delta\left(|v^r|^2-|v^\theta|^2\right)\rho'\left(\frac r\delta\right)\sin\theta+\frac r{\delta^2}v^rv^\theta\left(\rho''\left(\frac r\delta\right)+\frac\delta r\rho'\left(\frac r\delta\right)\right)\cos\theta.
		\end{align*}
		We write $|v^\theta(r,\theta)|^2=v^\theta_1(r,\theta)\left(v^\theta_1(r,\theta)+2v^\theta_2(r)\right)+\left(v^\theta_2(r)\right)^2$. Since $\int_\TT\sin \theta\,d\theta=0$, the term $\left(v^\theta_2(r)\right)^2$ contributes nothing into the integral. As a consequence, we have
		\begin{align*}
			\left|\int_{\RR^2}\left(v\otimes v\right): \nabla \nabla^\perp\left(x_1 \rho_\delta\right)\,dx\right|&\lesssim \int_\delta^{2\delta}\left(\frac1\delta r^{3+\frac{2\alpha-2}{\mu}}+\frac1{\delta^2}r^{4+\frac{\alpha-2}{\mu}}+\frac1\delta r^{3+\frac{\alpha-2}{\mu}}\right)\,dr\\
			&\lesssim \delta^{3+\frac{2\alpha-2}{\mu}}+\delta^{3+\frac{\alpha-2}{\mu}}\to 0,\qquad \delta\to0+,
		\end{align*}
		since $\alpha>1-\mu>2-3\mu$.
		
		This concludes the proof of the proposition.
	\end{proof}

	We now discuss the initial data for $\vvv(t)$ and $\www(t)$.
	
	\begin{prop}\label{prop_data}
		Let $\mathring{\omega}(\theta)=\mu^{-\frac1{2\mu}}\Omega(\theta)$, then
		\begin{equation}
			\www(y,t)\xrightarrow{t\to0+}|y|^{-\frac1\mu}\ \mathring{\omega}(\theta)=:\www_0(y) \qquad \text{ in } L^1_{\text{loc}}(\RR^2).
		\end{equation}
	\end{prop}
	\begin{proof}
		By \eqref{self-similar} and \eqref{omega_Omega}, we have
		\begin{align*}
			\www(y,t)&=\frac1t\omega\left(\frac y{t^\mu}\right)=\frac{1}{t}\psi_\vp^{-\frac1{2\mu}}\left(\beta\left(\frac{|y|}{t^\mu},\theta\right), \phi\left(\frac{|y|}{t^\mu},\theta\right)\right)\Omega\left(\phi\left(\frac{|y|}{t^\mu},\theta\right)\right)\\
			&\xlongequal{r=t^{-\mu}|y|}|y|^{-\frac1\mu}r^{\frac1\mu}
			\left(\psi_\vp(\beta(r,\theta),\phi(r,\theta))\right)^{-\frac1{2\mu}}\Omega(\phi(r,\theta))\\
			&\xlongequal{\eqref{r_beta^-mu}} |y|^{-\frac1\mu}\left(\frac{-\psi_\beta}{\mu\psi_\vp}\right)^{\frac1{2\mu}}(\beta(r,\theta),\phi(r, \theta))\Omega(\phi(r,\theta)).
		\end{align*}
		Let $\bar w(\beta,\phi):=\left(\frac{-\psi_\beta}{\mu\psi_\vp}\right)^{\frac1{2\mu}}(\beta,\phi)$. Recalling $\p_\vp=\p_\phi-\p_\beta$, we have
		\[\frac{-\psi_\beta}{\mu\psi_\vp}=\frac1\mu-\frac{\psi_\phi}{\mu\psi_\vp}=\frac1\mu-\frac\beta\mu\frac{\beta^{2\mu-1}\psi_\phi}{\beta^{2\mu}\psi_\vp}.\]
		It follows from the boundedness of $\beta^{2\mu-1}\psi_\phi$ and \eqref{5.2} that $\lim\limits_{\beta\to0+} \bar w(\beta,\phi)=\mu^{-\frac1{2\mu}}$ uniformly. Now, a similar argument involving the change of variables as in \eqref{5.6} gives the $L^1_{\text{loc}}(\RR^2)$ convergence of $\www(t)$ as $t\to0+$.
	\end{proof}

	Now let us consider the initial data for $\vvv$. Let $\Psi_0\in L^\infty_{\text{loc}}(\RR^2)$ solve $\Delta_y\Psi_0=\www_0$ and we require that $\Psi_0$ is $m$-fold symmetric and it has the bound $|\Psi_0(y)|\lesssim |y|^{2-\frac1\mu}$. Indeed, there exists only one $\Psi_0$ satisfying our requirements: $\Psi_0(y)=|y|^{2-\frac1\mu}B(\theta)$, where $B(\theta)$ is the only $m$-fold function solving the ODE $\gamma^2 B(\theta)+B''(\theta)=\mathring{\omega}$, hence $\Psi_0\in C^1(\mathbb{R}^2\setminus\{0\})$. The uniqueness can be proved by using several methods: considering the Fourier coefficients as in Subsection \ref{Uniqueness}; or by Liouville's theorem, we know from $\Delta_y\Psi_0=0$, $|\Psi_0(y)|\lesssim |y|^{2-\frac1\mu}$ and $\mu>\frac12$ that $\Psi_0$ is an affine function, i.e., $\Psi_0(y)=(C_1, C_2)\cdot y$, now since $\Psi_0$ is $m$-fold symmetric and $m\geq2$, we obtain $\Psi_0=0$; the third method is applying the Poisson's representation formula for the Laplace operator and then using the $m$-fold symmetry of $\Psi_0$ to gain more decay in the formula, see Lemma 2.9 in \cite{Elgindi}.
	
	Recall that $\Psi_0\in C^1(\mathbb{R}^2\setminus\{0\})$, we can define $\vvv_0=\nabla_y^\perp\Psi_0$, then $|\vvv_0(y)|\lesssim |y|^{1-\frac1\mu}$, so $\vvv_0\in L^2_{\text{loc}}(\RR^2; \RR^2)$.
	
	\begin{prop}\label{v_initial}
		As $t\to0+$, we have
		\[\Psi(\cdot, t)\to \Psi_0\ \ \text{ in }\  L^\infty_{\text{loc}}(\RR^2),\qquad \vvv(\cdot, t)\to \vvv_0\ \ \text{ in } \ L^2_{\text{loc}}(\RR^2; \RR^2).\]
		In particular, $\vvv\in C([0,\infty); L_{\text{loc}}^2(\RR^2; \RR^2))$.
	\end{prop}
	\begin{proof}
		It follows from \eqref{self-similar}, $|\psi|\lesssim r^{2-\frac1\mu}$ and \eqref{5.4} that
		\begin{align}\label{psi1}\left|\Psi(y,t)\right|\lesssim |y|^{2-\frac1\mu},\qquad \left|\nabla_y\Psi(y,t)\right|\lesssim |y|^{1-\frac1\mu}.\end{align}
		By Arzel\`a-Ascoli lemma and Cantor's diagonal arguments, there exist a sequence $\{t_n\}$ with $\lim\limits_{n\to\infty}t_n=0$ and an $m$-fold function $\widetilde{\Psi_0}\in L^\infty_{\text{loc}}(\RR^2\setminus\{0\})$ such that $\Psi(\cdot, t_n)\to \widetilde{\Psi_0}$ in $L^\infty_{\text{loc}}(\RR^2\setminus\{0\})$, and moreover $\Psi$ and $\widetilde{\Psi_0}$ have the bound $\left|\Psi(y,t)\right|+\left|\widetilde{\Psi_0}(y)\right|\lesssim |y|^{2-\frac1\mu}\to0$ as $y\to0$, which implies that $\Psi(\cdot, t_n)\to \widetilde{\Psi_0}$ in $L^\infty_{\text{loc}}(\RR^2)$. By $\Delta_y \Psi=\www$ and Proposition \ref{prop_data}, we have $\Delta_y\widetilde{\Psi_0}=\www_0$ weakly. Now, the uniqueness of solution stated above implies that $\widetilde{\Psi_0}=\Psi_0$, which is independent of the sequence $\{t_n\}$. As a consequence, we can easily show that $\Psi(\cdot, t)\to \Psi_0$ in $L^\infty_{\text{loc}}(\RR^2)$ as $t\to0+$.
		
		It remains to show that $\vvv(\cdot, t)\to \vvv_0\text{ in }L^2_{\text{loc}}(\RR^2; \RR^2)$. For simplicity, we denote $\widetilde{\vvv}(y,t):=\vvv(y,t)-\vvv_0(y)$, $\widetilde{\Psi}(y,t):=\Psi(y,t)-\Psi_0(y)$ and $\widetilde{\www}(y,t):=\www(y,t)-\www_0(y)$. Fix an arbitrary $R>0$, we introduce a smooth bump function $\rho\in C_c^\infty(\RR^2; [0,1])$ such that $\rho|_{B_R}\equiv 1$ and $\text{supp }\rho\subset B_{R+1}$, where $B_R=\{y\in\RR^2: |y|\leq R\}$. Integration by parts gives
		\begin{align*}
			\left\|\vvv(\cdot, t)-\vvv_0\right\|_{L^2(B_R)}^2&\leq \int_{\RR^2}\rho(y)\left|\widetilde{\vvv}(y,t)\right|^2\,dy=\int_{\RR^2}\rho(y)\nabla_y\widetilde{\Psi}(y,t)\cdot \nabla_y\widetilde{\Psi}(y,t)\,dy\\
			&=-\int_{\RR^2}\widetilde{\Psi}(y,t)\text{ div}_y\left(\rho\nabla_y\widetilde{\Psi}\right)(y,t)\,dy\\
			 &=-\int_{\RR^2}\widetilde{\Psi}(y,t)\nabla_y\widetilde{\Psi}(y,t)\cdot\nabla_y\rho(y)\,dy-\int_{\RR^2}\rho(y)\widetilde{\Psi}(y,t)\Delta_y\widetilde{\Psi}(y,t)\,dy\\
			&=\frac12\int_{\RR^2}\left|\widetilde{\Psi}(y,t)\right|^2\Delta_y\rho(y)\,dy-\int_{\RR^2}\rho(y)\widetilde{\Psi}(y,t)\widetilde{\www}(y,t)\,dy\\
			&\lesssim\left\|\widetilde{\Psi}(\cdot, t)\right\|_{L^\infty(B_{R+1})}^2+\left\|\widetilde{\Psi}(\cdot, t)\right\|_{L^\infty(B_{R+1})}\left\|\widetilde{\www}(\cdot, t)\right\|_{L^1(B_{R+1})}\to0
		\end{align*}
		as $t\to0+$. This concludes the proof.
	\end{proof}

	
	\begin{prop}\label{prop_vv_weak}
		Assume that $\mu>\frac12$ and $\alpha\in(\max\{0, 1-\mu\}, \alpha_\mu)$, where $\alpha_\mu$ is given by \eqref{alpha_mu}. Then $\vvv=\vvv(y,t)$ is a weak solution of the 2-D Euler equation \eqref{2DEuler} on $\RR^2\times[0,\infty)$.
	\end{prop}
	\begin{proof}
		Let $\mathbf{w}\in C_c^\infty(\RR^2_y\times [0,\infty); \RR^2)$ be a divergence-free test function. We need to show that
		\begin{equation}\label{5.13}
			\int_{\RR^2}\vvv\cdot\mathbf w \,dy\Big|_{t=0}+\int_0^\infty\int_{\RR^2}\vvv\cdot \p_t\mathbf w +\left(\vvv\otimes\vvv\right): \nabla_y\mathbf w \,dy\,dt=0.
		\end{equation}
		Let $w(x,t)=t^{3\mu-1}\mathbf w(t^\mu x, t)$ for $x\in\RR^2$ and $t>0$, then $\mathbf w(y,t)=t^{1-3\mu}w\left(\frac{y}{t^\mu},t\right)$, thus
		\begin{align*}
			\p_t\mathbf w(y,t)&=t^{-3\mu}\left((1-3\mu)w\left(\frac{y}{t^\mu},t\right)-\mu\frac{y}{t^\mu}\cdot\nabla_{x}w
			\left(\frac{y}{t^\mu},t\right)+t\p_tw\left(\frac{y}{t^\mu},t\right)\right),\\
			\nabla_y\mathbf w (y,t)&=t^{1-4\mu}\nabla_{x}w\left(\frac{y}{t^\mu},t\right).
		\end{align*}
		For each $t>0$, we have $w(\cdot, t)\in C_c^\infty(\RR^2;\RR^2)$ and $\text{ div}_x w(\cdot, t)=0$, by $\vvv(y,t)=t^{\mu-1}v\left(\frac{y}{t^\mu}\right)$ and Proposition \ref{prop_v_weak},
		\begin{align*}
			&\ \ \ \int_{\RR^2}\vvv\cdot \p_t\mathbf w +\left(\vvv\otimes\vvv\right): \nabla_y\mathbf w \,dy\\
			&=t^{-1}\int_{\RR^2}v(x)\cdot t\p_tw(x,t)+(1-3\mu)v(x)\cdot w(x,t)-\mu v(x)\cdot \left(x\cdot\nabla w(x,t)\right)\\
			&\qquad\qquad+\left(v(x)\otimes v(x)\right):\nabla w(x,t)\,dx\\
			&=\int_{\RR^2}v(x)\cdot\p_tw(x,t)\,dx=\frac{d}{dt}\int_{\RR^2}v(x)\cdot w(x,t)\,dx;\\
			&\ \ \ \int_{\RR^2}\vvv(y,t)\cdot \mathbf w (y,t)\,dy=\int_{\RR^2}t^{-2\mu}v\left(\frac{y}{t^\mu}\right)\cdot w\left(\frac{y}{t^\mu},t\right)\,dy=\int_{\RR^2}v(x)\cdot w(x,t)\,dx.
		\end{align*}
		Therefore, for any small $\tau>0$ we have
		\[\int_{\RR^2}\vvv\cdot\mathbf w \,dy\Big|_{t=\tau}+\int_\tau^\infty\int_{\RR^2}\vvv\cdot \p_t\mathbf w +\left(\vvv\otimes\vvv\right): \nabla_y\mathbf w \,dy\,dt=0.\]
		Letting $\tau\to0+$ and using $\vvv\in C([0,\infty); L_{\text{loc}}^2(\RR^2; \RR^2))$ gives \eqref{5.13}.
	\end{proof}
	
	Finally, let us prove the main theorem.
	
	\begin{proof}[Proof of Theorem \ref{mainthm}]
		Assume that $\mu>\frac12$ and $m\ge2$. By Theorem \ref{IFT}, Proposition \ref{prop_data} and Proposition \ref{prop_vv_weak}, there exists a small positive number $\varepsilon_\Omega>0$ such that for all $\Omega\in L^1(\TT)$ with $\|\Omega-\gamma\|_{\WWW_m}<\varepsilon_\Omega$, we can find a weak solution $\vvv$ to the 2-D Euler equation \eqref{2DEuler} with the initial data \eqref{initial_data}, where $\mathring{\omega}=\mu^{-\frac1{2\mu}}\Omega$. Note that the condition $\|\Omega-\gamma\|_{\WWW_m}<\varepsilon_\Omega$ is equivalent to $\|\mathring{\omega}-\gamma\mu^{-\frac1{2\mu}}\|_{\WWW_m}<\mu^{-\frac1{2\mu}}\varepsilon_\Omega $. Choose $\varepsilon\in\left(0,\frac1\gamma\varepsilon_\Omega\right)$ and now we verify Theorem \ref{mainthm}. Let $\mathring{\omega}\in L^1(\TT)$ with \eqref{1.5}, i.e., $\|P_{\neq}\mathring{\omega}\|_{L^1(\TT)}\leq \varepsilon m^{\frac12}|P_0\mathring{\omega}|$. We may assume that $P_0\mathring{\omega}\neq0$, otherwise $\mathring{\omega}=0$ and things are trivial. Denote $c_0=\gamma\mu^{-\frac1{2\mu}}>0$ and $A=P_0\mathring{\omega}\neq0$. Let $\mathring{\omega}_1:=\frac{c_0}{A}\mathring{\omega}$, then
		\begin{align*}
			\|\mathring{\omega}_1-c_0\|_{\WWW_m}=\frac{c_0}{|A|}\|\mathring{\omega}-A\|_{\WWW_m}=
			\frac{c_0}{|A|}m^{-\frac12}\left\|P_{\neq}\mathring{\omega}\right\|_{L^1(\TT)}\leq c_0\varepsilon< \mu^{-\frac1{2\mu}}\varepsilon_\Omega.
		\end{align*}
		Hence, we can find a weak solution $\vvv_1$ to the 2-D Euler equation with the initial data $\www_1|_{t=0}=|y|^{-\frac1\mu}\mathring{\omega}_1(\theta)$. Finally, using the well-known Euler scaling property, $\vvv(y,t):=\frac{A}{c_0}\vvv_1\left(y,\frac A{c_0}t\right)$ is a weak solution of \eqref{2DEuler} with the initial data \eqref{initial_data}.
	\end{proof}

	\section{Existence of weak solution with Radon measure}\label{sec_main_cor}
	
	In this section, we prove  Corollary \ref{maincor}. As we mentioned in the introduction, we regularize $\mathring{\omega}\in\mathcal{M}(\TT)$ by using $F_N=N\chi_{(0,1/N)}$ to get $\mathring{\omega}_N=F_N*\mathring{\omega}\in L_m^1(\TT)$ which satisfies \eqref{1.5} for each $\mathring{\omega}_N (N\in\NN_+)$. In this section, the implicit constants in all $\lesssim$ are independent of $N$.

	As in the proof of Theorem \ref{mainthm},  we only need to consider $\mathring{\omega}\in\mathcal{M}(\TT)$ such that $\Omega=\mu^{\frac1{2\mu}}\mathring{\omega}$ satisfies $\|\Omega-\gamma\|_{\widetilde{\WWW}_m}<\varepsilon_\Omega\ll 1$, where $\widetilde{\WWW}_m$ is the trivial extension of $\WWW_m$ into the corresponding subspace of $\mathcal{M}(\TT)$. In this case, we have $\Omega_N\in L^1(\TT)$ and $\|\Omega_N-\gamma\|_{\WWW_m}<\varepsilon_\Omega$.
	By Theorem \ref{IFT} and the arguments in Subsection \ref{sec_phycical_variable}, for each $N\in\NN_+$, there exists $\psi_N\in C^1(\RR^2)$ such that $v_N=\nabla_x^\perp \psi_N\in L^2_{\text{loc}}(\RR^2)$ satisfies \eqref{v_weaksol}: for $w\in C_c(\RR^2; \RR^2)$ with $\text{ div } w=0$, there holds
	\begin{equation}\label{5.8'}
		\int_{\RR^2}(3\mu-1)v_N\cdot w-(v_N\otimes v_N):\nabla w+\mu v_N\cdot(x\cdot\nabla w)\,dx=0,
	\end{equation}
	and $\omega_N\in L^1_{\text{loc}}(\RR^2)$ with the uniform bounds
	\begin{equation}
		|\psi_N(x)|\lesssim |x|^{2-\frac1\mu},\qquad |\nabla_x\psi_N(x)|\lesssim |x|^{1-\frac1\mu}, \qquad \int_{|x|\leq R}|\omega_N(x)|\,dx\lesssim_R1.
	\end{equation}
	
	Similar to the proof of Proposition \ref{v_initial}, by Arzel\`a-Ascoli lemma and Cantor's diagonal arguments, we can find a subsequence of $\{\psi_N\}$, which is still denoted by $\{\psi_N\}$ and $\psi\in C(\RR^2)$ such that $\psi_N\rightarrow\psi$ in $L^\infty_{\text{loc}}(\RR^2)$ as $N\to\infty$, moreover $\psi$ has the bounds $|\psi(x)|\lesssim  |x|^{2-\frac1\mu}$ and $|\nabla_x\psi(x)|\lesssim |x|^{1-\frac1\mu}$. Since $v_N$ is uniformly bounded in $L^2_{\text{loc}}(\RR^2)$, using again Cantor's diagonal arguments, up to subsequence we have $v_N\rightharpoonup v$ in $L^2(B_R)$ for all $R>0$ and some $v\in L^2_{\text{loc}}(\RR^2)$. A standard argument gives that $v=\nabla_x^\perp \psi$ weakly.
	
	Now we claim that $v_N\rightarrow v$ in $L^2_{\text{loc}}(\RR^2)$. The proof is similar to the corresponding step in the proof of Proposition \ref{v_initial}. Fix an arbitrary $R>0$. Let $\rho\in C_c^\infty(\RR^2; [0,1])$ be a smooth bump function such that $\rho|_{B_R}\equiv1$ and $\text{supp }\rho\subset B_{R+1}$. Then we have
	\begin{align*}
		\int_{B_R}|v_N-v|^2\,dx&\leq \int_{\RR^2}\rho(x)|v_N(x)-v(x)|^2\,dx=\int_{\RR^2}\rho(x)\left|\nabla_x(\psi_N-\psi)\right|^2\,dx\\
		&=\int_{\RR^2}\rho(x)\nabla_x\psi_N\cdot \nabla_x(\psi_N-\psi)\,dx-\int_{\RR^2}\rho(x)v(x)\cdot(v_N-v)\,dx.
	\end{align*}
	It follows from $v_N\rightharpoonup v$ in $L^2(B_{R+1})$ that $\lim\limits_{N\to\infty}\int_{\RR^2}\rho(x)v(x)\cdot(v_N-v)\,dx=0$. For the other integral, we use integration by parts to obtain
	\begin{align*}
		&\left|\int_{\RR^2}\rho(x)\nabla_x\psi_N\cdot \nabla_x(\psi_N-\psi)\,dx\right|=\left|\int_{\RR^2}(\psi_N-\psi)\text{ div}(\rho\nabla_x\psi_N)\,dx\right|\\
		&\leq \|\psi_N-\psi\|_{L^\infty(B_{R+1})}\int_{\RR^2}\left|\text{div}(\rho\nabla\psi_N)\right|\,dx\lesssim_R\|\psi_N-\psi\|_{L^\infty(B_{R+1})}\to0, \qquad N\to\infty,
	\end{align*}
	where we have used the uniform estimate
	\begin{align*}
		\int_{\RR^2}\left|\text{div}(\rho\nabla\psi_N)\right|\,dx&\lesssim\int_{\RR^2}|\nabla \rho||\nabla\psi_N|\,dx+\int_{\RR^2}\rho|\Delta\psi_N|\,dx\\
		&\lesssim_R\|v_N\|_{L^2(B_{R+1})}+\|\omega_N\|_{L^1(B_{R+1})}\lesssim_R1.
	\end{align*}
	This proves $v_N\rightarrow v$ in $L^2_{\text{loc}}(\RR^2)$. Letting $N\to\infty$ in \eqref{5.8'}, we obtain \eqref{v_weaksol} for $v\in L^2_{\text{loc}}(\RR^2)$. Hence, $v$ solves the equation \eqref{5.7} weakly.
	
	Next we recover the velocity field $\vvv(y,t)=t^{\mu-1}v\left(\frac y{t^\mu}\right)$, $\vvv_N(y,t)=t^{\mu-1}v_N\left(\frac y{t^\mu}\right)$. Then $\vvv_N,\ \vvv\in C(\RR_+; L^2_{\text{loc}}(\RR^2))$ and $\vvv_N\to\vvv $ in $L^2_{\text{loc}}(\RR^2\times\RR_+)$. We also have $| v_N(x)|\lesssim |x|^{1-\frac1\mu}$,
	$| v(x)|\lesssim |x|^{1-\frac1\mu}$, $| \vvv_N(y,t)|\lesssim |y|^{1-\frac1\mu}$,
	$| \vvv(y,t)|\lesssim |y|^{1-\frac1\mu}$, then $\vvv_N,\ \vvv$ are uniformly bounded in $ L^{\infty}(\RR_+; L^2(B_R))$, and by the dominated convergence theorem, we have $\vvv_N\to\vvv$ in $ L^2(B_R\times(0,T))$ for any $R,T\in\RR_+.$ By the proof of Proposition \ref{prop_vv_weak}, we have
	\[\int_{\RR^2}\vvv_N\cdot\mathbf w \,dy\Big|_{t=\tau}+\int_\tau^\infty\int_{\RR^2}\vvv_N\cdot \p_t\mathbf w +\left(\vvv_N\otimes\vvv_N\right): \nabla_y\mathbf w \,dy\,dt=0,\]
	for every $\tau\geq0 $ and every $\mathbf{w}\in C_c^\infty(\RR^2_y\times [0,\infty); \RR^2)$ with $\text{div}\,\mathbf{w}=0$.
	Letting $N\to+\infty$, we have (for $ \tau>0$)
	\begin{align}\label{vvv}\int_{\RR^2}\vvv\cdot\mathbf w \,dy\Big|_{t=\tau}+\int_\tau^\infty\int_{\RR^2}\vvv\cdot \p_t\mathbf w +\left(\vvv\otimes\vvv\right): \nabla_y\mathbf w \,dy\,dt=0.
	\end{align}
	
By Proposition \ref{v_initial}, we know that $\lim\limits_{t\to0+}\left\|\vvv_N(\cdot, t)-\vvv_N|_{t=0}\right\|_{L^2(B_R)}=0$, if we define $\vvv_N|_{t=0}=\nabla_y^\perp \Psi_{0,N}$ such that $\Psi_{0,N}(y)=|y|^{2-\frac{1}{\mu}}B_N(\theta)$, $ \gamma^2B_N+B_N''=\mathring{\omega}_N $.
	Let $\vvv_0=\nabla_y^\perp \Psi_{0}$, $\Psi_{0}(y)=|y|^{2-\frac{1}{\mu}}B(\theta)$,	$ \gamma^2B+B''=\mathring{\omega}$. Then we have $B_N=F_N*B$ and  $\lim\limits_{N\to\infty}\left\|\vvv_N|_{t=0}-\vvv_0\right\|_{L^2(B_R)}=0$. Thus, \eqref{vvv} is still true for $ \tau=0$ if we define $\vvv|_{t=0}=\vvv_0 $.
	
	Finally, we prove $\vvv\in C([0,\infty); L^2_{\text{loc}}(\RR^2))$.  Now by \eqref{vvv}, $\int_{\RR^2}\vvv\cdot\mathbf w \,dy\Big|_{t=\tau}$ is continuous at $ \tau=0$, thus $\int_{\RR^2}\vvv(y,t)\cdot w(y) \,dy$ is continuous at $ t=0$ for all $w\in C_c^\infty(\RR^2; \RR^2)$, div\,$w=0$.
	
	We define $\Psi(y, t)=t^{2\mu-1}\psi\left(\frac y{t^\mu}\right) $, then $\vvv(y,t)=\nabla_y^\perp \Psi(y, t)$, and \eqref{psi1} is true. Similar to the proof of Proposition \ref{v_initial}, there exists $t_n\to0+$ so that $\Psi(\cdot, t_n)\to \widetilde{\Psi_0}$ in $L^\infty_{\text{loc}}(\RR^2)$, and $|\widetilde{{\Psi_{0}}}(y)|\lesssim |y|^{2-\frac1\mu}$. Let $\widetilde{\vvv}_0=\nabla_y^\perp \widetilde{{\Psi_{0}}}$. Then $ \int_{\RR^2}\vvv_0\cdot w \,dy=\lim\limits_{n\to\infty}\int_{\RR^2}\vvv(y,t_n)\cdot w(y) \,dy=-\lim\limits_{n\to\infty}\int_{\RR^2}\Psi(y,t_n)(\nabla\times w)(y) \,dy=-\int_{\RR^2}\widetilde{{\Psi_{0}}}(y)(\nabla\times w)(y) \,dy=\int_{\RR^2}\widetilde{\vvv}_0\cdot w \,dy$ for all $w\in C_c^\infty(\RR^2; \RR^2)$, div\,$w=0$. Then $\Delta_y(\widetilde{{\Psi_{0}}}-{\Psi_{0}})=0$, which gives $\widetilde{{\Psi_{0}}}={\Psi_{0}}$ as $|\widetilde{{\Psi_{0}}}|+| \Psi_0(y)|\lesssim |y|^{2-\frac1\mu}$. This limit is independent of the sequence $\{t_n\}$. As a consequence, $\Psi(\cdot, t)\to \Psi_0$ in $L^\infty_{\text{loc}}(\RR^2)$ and  $\vvv(\cdot, t)\rightharpoonup \vvv_0$ in $L^2_{\text{loc}}(\RR^2)$ as $t\to0+$. Next we prove that $\vvv(\cdot, t)\to \vvv_0$ in $L^2_{\text{loc}}(\RR^2)$ as $t\to0+$. Similar to the proof of $v_N\rightarrow v$ in $L^2_{\text{loc}}(\RR^2)$, we need to prove the uniform boundedness of $\text{div}(\rho\nabla\Psi(\cdot,t))$ in $\mathcal{M}(B_{R+1}) $ norm. This follows from the uniform boundedness of $\|\text{div}(\rho\nabla\Psi_N(\cdot,t))\|_{L^1(B_{R+1})}$ (here $\Psi_N(y, t)=t^{2\mu-1}\psi_N\left(\frac y{t^\mu}\right) $).
	
	This proves $\vvv\in C([0,\infty); L^2_{\text{loc}}(\RR^2))$, and by $\www=\nabla_y\times\vvv$, we have $\www\in C([0,\infty); \mathcal{D}'(\RR^2))$.

	\appendix

	\section{Some ODE lemmas}\label{AppendixB}
	\begin{lem}\label{prop_B1}
		Let $q\in C^\infty([0,\infty); (0,\infty))$ be  such that $q(x)=q_1^2$ for $x\in[0,1]$ and $q(x)=q_2^2$ for $x\geq 2$, where $q_1$ and $q_2$ are positive real numbers. Consider the second order linear differential operator $L$ defined by
		\begin{equation}\label{B.1}
			(Ly)(x)=x^2y''(x)+xy'(x)-q(x)y(x),\qquad x>0.
		\end{equation}
		Then $Ly=0$ has a fundamental system of solutions $\{y_1, y_2\}$ given by
		\begin{equation}\label{B.2}
			y_1(x)\begin{cases}
				=x^{q_1} & x\in(0,1)\\
				\sim x^{q_1} & x\in[1,2]\\
				=C_1x^{q_2}+C_2x^{-q_2} & x>2,
			\end{cases}\quad\text{and}\quad y_2(x)\begin{cases}
				=C_3x^{-q_1}+C_4x^{q_1} & x\in(0,1)\\
				\sim x^{-q_2} & x\in [1,2]\\
				=x^{-q_2} & x>2,
			\end{cases}
		\end{equation}
		where $C_1, C_2, C_3, C_4$ are real constants and $C_1\neq 0, C_3\neq 0$.
	\end{lem}
	In this appendix, all implicit constants in $\lesssim$ and $\sim$ depend only on $q$ and parameters $\alpha,\mu$ (see Lemma \ref{prop_B2} below).
	\begin{proof}
		We denote
		\[L_1y=x^2y''+xy'-q_1^2y,\qquad L_2y=x^2y''+xy'-q_2^2y,\]
		then $\{x^{q_1}, x^{-q_1}\}$ is a fundamental system of solutions of $L_1y=0$, and $\{x^{q_2}, x^{-q_2}\}$ is a fundamental system of solutions of $L_2y=0$. By standard ODE theory, we know that $Ly=0$ has two smooth solutions $y_1, y_2$ of the form \eqref{B.2}. We need to check that $C_1\neq 0$ and $C_3\neq 0$, so that $y_1$ and $y_2$ are linearly independent, and thus they form a fundamental system of solutions to $Ly=0$.
		
		We look at $y_2$ first. Note that from \eqref{B.1}, we have $(xy_2y_2')'=x(y_2')^2+y_2(xy_2''+y_2')=x(y_2')^2+y_2^2q(x)/x\geq0$, and then $xy_2y_2'$ is increasing on
		$(0,\infty)$. We also have $y_2>0$, $y_2'<0$, $xy_2y_2'<0$ on $(2,\infty)$. Then $xy_2y_2'<0$ on $(0,\infty)$, $y_2'$ does not change sign on $(0,\infty)$,
		$y_2'<0$ on $(0,\infty)$.
		So, $y_2$ is decreasing on the whole interval $(0,\infty)$, which implies that $C_3>0$.
		
		Finally, we prove that $C_1\neq 0$. We consider the Wronskian
		\begin{equation}\label{Wronskian}
			W(x)=y_1(x)y_2'(x)-y_2(x)y_1'(x), \qquad x>0.
		\end{equation}
		Direct computation gives the equation for $W$: $xW'(x)+W(x)=0$, hence $W(x)=\frac Cx$ for some constant $C$. It follows from the definition that $W(x)=-2q_1C_3x^{-1}$ for $x\in(0,1)$, hence $C=-2q_1C_3<0$. As a consequence, $C_1\neq 0$, because $W(x)=-2q_2C_1x^{-1}\neq0$ for $x>2$.
	\end{proof}
	
	\begin{lem}\label{prop_B2}
		Let $q\in C^\infty([0,\infty); (0,\infty))$ be  as in Lemma \ref{prop_B1}. Assume that $f: (0,\infty)\to\CC$ satisfies $\langle x\rangle^\alpha x^{2\mu-1}f\in L^\infty$ for $\alpha\in(0,1)$, $\mu>1/2$ and
		\begin{equation}\label{alpha<1/2}
			q_1-2\mu+1>0,\qquad q_2-2\mu-\alpha+1>0,
		\end{equation}
		then the function
		\begin{equation}\label{B.4}
			y(x)=u_1(x)y_1(x)+u_2(x)y_2(x),\qquad x>0
		\end{equation}
		is a solution to  $Ly=f$, where $L$ is given by  \eqref{B.1}, $y_1, y_2$ are given by \eqref{B.2} and $u_1, u_2$ are given by
		\begin{equation}\label{B.5}
			u_1(x)=\int_x^\infty\frac{y_2(t)f(t)}{t^2W(t)}\,dt,\qquad u_2(x)=\int_0^x\frac{y_1(t)f(t)}{t^2W(t)}\,dt,
		\end{equation}
		here $W$ is the Wronskian defined in \eqref{Wronskian}. Moreover, we have
		\begin{equation}\label{B.6}
			\left\|\langle x\rangle^\alpha x^{2\mu}y'\right\|_{L^\infty}+\left\|\langle x\rangle^\alpha x^{2\mu-1}y\right\|_{L^\infty}\lesssim\left\|\langle x\rangle^\alpha x^{2\mu-1}f\right\|_{L^\infty}.
		\end{equation}
	\end{lem}
	\begin{rmk}
		The solution \eqref{B.4} to  $Ly=f$ is found by using the fundamental system of solutions $\{y_1, y_2\}$ and the method of variation of constants. Here we check that the solution \eqref{B.4} is really a solution with the bound \eqref{B.6}.
	\end{rmk}
	\begin{proof}[Proof of Lemma \ref{prop_B2}]
		We start with the estimates on $u_1$ and $u_2$, which not only show that the integrals in \eqref{B.5} are absolutely convergent, but also are crucial in the proof of the estimate \eqref{B.6}. From the proof of Lemma \ref{prop_B1}, we know that $W(x)=\frac Cx$ for some non-zero constant $C\in\RR$. For simplicity, we assume that $\left\|\langle x\rangle^\alpha x^{2\mu-1}f\right\|_{L^\infty}=1$. If $x\geq1$, then
		\[|u_1(x)|\lesssim \int_x^\infty t^{-q_2-1}|f(t)|\,dt\lesssim\int_x^\infty t^{-q_2-1}t^{1-2\mu-\alpha}\,dt\lesssim x^{-q_2+1-2\mu-\alpha};\]
		If $x\in(0,1)$, then
		\begin{align*}
			|u_1(x)|\lesssim \int_x^1 t^{-q_1-1}t^{1-2\mu}\,dt+\int_1^\infty t^{-q_2-1}t^{1-2\mu}\,dt\lesssim x^{-q_1+1-2\mu}+1\lesssim x^{-q_1+1-2\mu},
		\end{align*}
		hence we can conclude that
		\begin{equation}\label{u1}
			|u_1(x)|\lesssim \begin{cases}
				\langle x\rangle^{-\alpha}x^{-q_1+1-2\mu} & x\in(0,1)\\
				\langle x\rangle^{-\alpha}x^{-q_2+1-2\mu} & x\geq 1.
			\end{cases}
		\end{equation}
		Similarly, we estimate $u_2$ by using \eqref{alpha<1/2}. If $x\in(0,1)$, then
		\[|u_2(x)|\lesssim\int_0^x t^{q_1-1}t^{1-2\mu}\,dt\lesssim x^{q_1+1-2\mu};\]
		If $x\geq 1$, then
		\[|u_2(x)|\lesssim\int_0^1 t^{q_1-1}t^{1-2\mu}\,dt+\int_1^x t^{q_2-1}t^{1-2\mu-\alpha}\,dt\lesssim x^{q_2+1-2\mu-\alpha},\]
		hence we conclude that
		\begin{equation}\label{u2}
			|u_2(x)|\lesssim \begin{cases}
				\langle x\rangle^{-\alpha}x^{q_1+1-2\mu} & x\in(0,1)\\
				\langle x\rangle^{-\alpha}x^{q_2+1-2\mu} & x\geq 1.
			\end{cases}
		\end{equation}
		Now \eqref{u1} and \eqref{u2} imply that $u_1$ and $u_2$ are well-defined.
		
		Next we check that the function $y$ given by \eqref{B.4} satisfies $Ly=f$. Direct computation gives that $y'=u_1'y_1+u_2'y_2+u_1y_1'+u_2y_2'=u_1y_1'+u_2y_2'$ and $y''=u_1'y_1'+u_2'y_2'+u_1y_1''+u_2y_2''$, hence by the definition of Wronskian,
		\begin{align*}
			Ly=x^2\left(u_1'y_1'+u_2'y_2'\right)+u_1Ly_1+u_2Ly_2=x^2\left(-\frac{y_2f}{x^2W}y_1'+\frac{y_1f}{x^2W}y_2'\right)=f.
		\end{align*}
		
		Finally, the estimate \eqref{B.6} follows directly from \eqref{B.2}, \eqref{u1} and \eqref{u2}, noting that $y'=u_1y_1'+u_2y_2'$. \end{proof}
	
	\begin{lem}\label{prop_B3}
		Let $q\in C^\infty([0,\infty); (0,\infty))$ be as in Lemma \ref{prop_B1}. Assume that $f: (0,\infty)\to\CC$ satisfies $\langle x\rangle^\alpha x^{2\mu-1}f\in L^\infty$ for $\alpha\in(0,1)$, $\mu>1/2$ and \eqref{alpha<1/2} holds.
		Then the function
		\begin{equation*}
			y(x)=u_1(x)y_1(x)+u_2(x)y_2(x),\qquad x>0
		\end{equation*}
		is a weak solution to $Ly(x)=xf'(x)$, where $L$ is given by \eqref{B.1}, $y_1, y_2$ are given by \eqref{B.2} and $u_1, u_2$ are given by
		\begin{equation*}
			u_1(x)=-\frac1C\int_x^\infty y_2'(t)f(t)\,dt,\qquad u_2(x)=-\frac1C\int_0^xy_1'(t)f(t)\,dt,
		\end{equation*}
		here $C=-2q_1C_3<0$ is a constant. Moreover, we have
		\begin{equation*}
			\left\|\langle x\rangle^\alpha x^{2\mu}y'\right\|_{L^\infty}+\left\|\langle x\rangle^\alpha x^{2\mu-1}y\right\|_{L^\infty}\lesssim\left\|\langle x\rangle^\alpha x^{2\mu-1}f\right\|_{L^\infty}.
		\end{equation*}
	\end{lem}
	
	We omit the proof of Lemma \ref{prop_B3}, because it is very similar to the proof of Lemma \ref{prop_B2}.
	
	\section{Proof of Proposition \ref{linear_basic_estimate1}}\label{AppendixB1}
	
The main idea of the proof is to write down explicitly the formula for the solution of the equation \eqref{linear_basic_eq} and then make the estimates directly. Although the integral representation \eqref{Q_expression} of the solution $Q$ is not singular, the singularity appears when we take the derivative of $Q$.
Because of the special shape of our new coordinates $(\beta, \phi)$, we will introduce a partition of unity to take full advantage of the properties of the coordinates. Thus,  we write $Q=\sum_{k=0}^\infty Q_k$ with $Q_k$ given by \eqref{Q_k_expression}. The major contribution to $Q_k(\beta,\phi)$ in the integral \eqref{Q_k_expression} comes from two parts:

 \begin{itemize}
 \item when $s$ is  small or large, we need to use the cancellation given by the condition $\hat G_{\pm1}=0,\ \widehat{G}_{0}=0$ to extract the leading order of the integrand;

 \item when $s$ is near $k$, we need to make full advantage of the $C_\beta^\al$ regularity of $G$ in order to compensate the smallness of the denominator.
  \end{itemize}

The most difficult part is the estimate of $Q_0$, especially for $\beta$ small, where two main contributions ($s$ small and $s$ near $k=0$) are mixed up.
In this case, we must perform a more refined decomposition; see Lemma \ref{lem5}. 

In this appendix, all implicit constants in $\lesssim$ and $\sim$
depend only on the parameters $\alpha, \mu$ and the bump functions $\eta,\rho$ we introduced in the proof, unless otherwise specified. And we assume that $\alpha\in(0, 1)$, $\mu>1/2$ throughout this appendix. Without loss of generality, we assume that $\|\beta^{2\mu-1}G\|_{C_\beta^\alpha}=1$.

		\subsection{Representation formula of the solution.}
				
		 Let
		\begin{equation}\label{QnGn}
			{\hat{Q}_n(\beta):=\frac{1}{2\pi}}\int_\TT Q(\beta,\phi)e^{-\ii n(\beta+\phi)}\,d\phi,\qquad \hat{G}_n(\beta):=\frac{1}{2\pi}\int_\TT G(\beta,\phi)e^{-\ii n(\beta+\phi)}\,d\phi.
		\end{equation}
		If $Q$ solves \eqref{linear_basic_eq}, then  \eqref{linear_basic_eq} is converted into a system of equations for modes $\hat Q_n$: $(\beta\p_\beta+\mu n)\hat Q_n=\hat G_n$, which is equivalent to
$\p_\beta\left(\beta^{\mu n}\hat Q_n\right)=\beta^{\mu n-1}\hat G_n$. The solution is given by
\begin{align}\label{Qn1}
&\hat Q_n(\beta)=\beta^{-\mu n}\int_0^\beta s^{\mu n-1}\hat G_n(s)\,ds=\int_0^\beta \left(\frac s\beta\right)^{\mu n}\frac{\hat G_n(s)}{s}\,ds\quad\text{for}\ n\geq 1,\\
\label{Qn0}&\hat Q_n(\beta)=-\beta^{-\mu n}\int_\beta^\infty s^{\mu n-1}\hat G_n(s)\,ds=-\int_\beta^\infty \left(\frac s\beta\right)^{\mu n}\frac{\hat G_n(s)}{s}\,ds\quad\text{for}\ n\leq 0{.}
\end{align}		
		Thanks to $\mu>1/2$, $\beta^{2\mu-1}G\in C_{\beta}^\alpha$ and {$\hat G_1=0$}, the above two integral formulas are absolutely convergent,\footnote{Note that $\hat G_{-1}=\hat G_0=0$ is not needed here for $\hat Q_n$ to be convergent:  for $n\leq0$, by $\mu>1/2$ and
			$\|\beta^{2\mu-1}\hat G_n\|_{L^{\infty}}\leq \|\beta^{2\mu-1}G\|_{C_\beta^\alpha}$, we have $\int_\beta^\infty\left|(s/\beta)^{n\mu}\frac{\hat G_n(s)}{s}\right|\,ds\leq
			\|\beta^{2\mu-1}G\|_{C_\beta^\alpha}\int_\beta^\infty(s/\beta)^{n\mu}\frac{s^{1-2\mu}}{s}\,ds<+\infty$.}  and if $\hat G_n=0$ then $\hat Q_n=0$.
To sum them up, we use two elementary identities: for $\theta\in\TT$,
\begin{equation}\label{rnsn}
\sum_{n=1}^{+\infty} r^n e^{\ii n\theta}=\frac{re^{\ii\theta}}{1-re^{\ii\theta}},\ \forall\ r\in(0,1),\qquad \sum_{n=-\infty}^0s^ne^{\ii n\theta}=
\frac{se^{\ii\theta}}{se^{\ii\theta}-1},\ \forall\ s\in (1,+\infty).\end{equation}
Then we can obtain
		\begin{align}\label{Q_expression}
				Q(\beta,\phi)&=\sum_{n\in\ZZ}\hat Q_n(\beta)e^{\ii n(\beta+\phi)}=\frac1{2\pi}\int_0^\infty\tilde Q(\beta, \phi, s)\,ds,\quad \beta>0, \phi\in\TT,
		\end{align}
		where for $\beta>0, s>0, \beta\neq s, \phi\in\TT$ we define \footnote{It is not necessary to give a precise definition of $\tilde Q(\beta,\phi,s)$ for $\beta=s$.  For simplicity we can define $\tilde Q(\beta,\phi,\beta)=0$. Note that the integral with respect to $s$ in \eqref{Q_expression} does not changed if we only change the  value of $\tilde Q(\beta,\phi,\beta)$.}
		\begin{align}
\label{Qt}\tilde Q(\beta, \phi, s)&:=\int_\TT \frac{\left( s/\beta\right)^\mu e^{\ii(\beta+\phi-s-\Phi)}}{1-\left(s/\beta\right)^\mu e^{\ii(\beta+\phi-s-\Phi)}}\frac{G(s, \Phi)}{s}\,d\Phi\\
				\notag&=\int_\TT\frac{s^\mu e^{-\ii(s+\Phi)}}{\beta^\mu e^{-\ii(\beta+\phi)}-s^\mu e^{-\ii(s+\Phi)}}\frac{G(s, \Phi)}{s}\,d\Phi.
		\end{align}
In details, we need to prove that
\begin{align}\label{a1}
			&\int_0^\beta \sum_{n=1}^{+\infty}\left(\frac s\beta\right)^{\mu n}\frac{|\hat G_n(s)|}{s}\,ds+\int_\beta^\infty\sum_{n=-\infty}^0\left(\frac s\beta\right)^{\mu n}\frac{|\hat G_n(s)|}{s}\,ds<+\infty,
\quad \forall\ \beta>0,
\\ \label{a2}&\sum_{n=1}^{+\infty}\left(\frac s\beta\right)^{\mu n}\frac{\hat G_n(s)}{s}e^{\ii n(\beta+\phi)}=\frac1{2\pi}\tilde Q(\beta, \phi, s),\quad \forall\ 0<s<\beta, \forall\ \phi\in\TT,\\
\label{a3}&-\sum_{n=-\infty}^0\left(\frac s\beta\right)^{\mu n}\frac{\hat G_n(s)}{s}e^{\ii n(\beta+\phi)}=\frac1{2\pi}\tilde Q(\beta, \phi, s),\quad \forall\ s>\beta>0, \forall\ \phi\in\TT.
		\end{align}
		
In fact, \eqref{a2} follows from the following facts.
\begin{itemize}
	\item By \eqref{QnGn}, we have
	$$\left(\frac s\beta\right)^{\mu n}\frac{\hat G_n(s)}{s}e^{\ii n(\beta+\phi)}=\frac{1}{2\pi}\int_\TT \left(\frac s\beta\right)^{\mu n}\frac{G(s,\Phi)}{s}e^{\ii n(\beta+\phi-s-\Phi)}\,d\Phi; $$
	\item By \eqref{rnsn}, we have (as $0<s<\beta$, $\left(s/\beta\right)^\mu\in(0,1) $)
	$$\sum_{n=1}^{+\infty} \left(\frac s\beta\right)^{\mu n}\frac{G(s,\Phi)}{s}e^{\ii n(\beta+\phi-s-\Phi)}=
	\frac{\left( s/\beta\right)^\mu e^{\ii(\beta+\phi-s-\Phi)}}{1-\left(s/\beta\right)^\mu e^{\ii(\beta+\phi-s-\Phi)}}\frac{G(s, \Phi)}{s};$$
	\item $ \int_\TT \sum_{n=1}^{+\infty}|\left(s/\beta\right)^{\mu n}\frac{G(s,\Phi)}{s}e^{\ii n(\beta+\phi-s-\Phi)}|\,d\Phi=
	\int_\TT \frac{\left( s/\beta\right)^\mu}{1-\left(s/\beta\right)^\mu }\frac{|G(s, \Phi)|}{s}\,d\Phi<+\infty$; \footnote{Here we only need to prove \eqref{a2} and \eqref{a3} for fixed $s\neq \beta$. It is not necessary to consider the limit of \eqref{a2} and \eqref{a3} as $s\to \beta$.}
	\item The definition of $ \tilde Q(\beta, \phi, s)$ in \eqref{Qt} and Fubini's theorem.
\end{itemize}
The proof of \eqref{a3} is similar.

\begin{proof}[Proof of \eqref{a1}]
Thanks to $\hat G_{\pm1}=0$, $\widehat{G}_{0}=0$, it is enough to prove that (for $N\in\ZZ$, $N\geq 2$, $ \beta>0$; the constant of $\lesssim$ in \eqref{a4}
depends only on $\mu$)
\begin{align}\label{a4}
			&\int_0^\beta \sum_{n=N}^{2N-1}\left(\frac s\beta\right)^{\mu n}\frac{|\hat G_n(s)|}{s}\,ds+\int_\beta^\infty\sum_{n=1-2N}^{-N}\left(\frac s\beta\right)^{\mu n}\frac{|\hat G_n(s)|}{s}\,ds\lesssim
N^{-1/2}\beta^{1-2\mu}.
		\end{align}
Then \eqref{a1} follows by taking $N=2^k$ in \eqref{a4} and summing over $k\in \ZZ_+$ ($\ZZ_+:=\ZZ\cap(0,+\infty)$).

As $\|\beta^{2\mu-1}G\|_{C_{\beta}^{\alpha}}=1$ we have $\beta^{2\mu-1}|G(\beta,\phi)|\leq 1 $. Then by \eqref{QnGn} and Parseval's identity, we have
$ \sum_{n\in\ZZ}|\hat G_n(s)|^2=\frac{1}{2\pi}\int_\TT|G(s,\Phi)|^2\,d\Phi\leq s^{2(1-2\mu)}$. Thus $\sum_{n=N}^{2N-1}|\hat G_n(s)|\leq N^{1/2}s^{1-2\mu} $, {$\sum_{n=1-2N}^{-N}|\hat G_n(s)|\leq N^{1/2}s^{1-2\mu} $. For} $0<s<\beta$ we have
$$\sum_{n=N}^{2N-1}\left(\frac s\beta\right)^{\mu n}\frac{|\hat G_n(s)|}{s}\leq \sum_{n=N}^{2N-1}\left(\frac s\beta\right)^{\mu N}\frac{|\hat G_n(s)|}{s}\leq \left(\frac s\beta\right)^{\mu N}N^{1/2}s^{-2\mu}. $$
{For} $s>\beta>0$ we have $$\sum_{n=1-2N}^{-N}\left(\frac s\beta\right)^{\mu n}\frac{|\hat G_n(s)|}{s}\leq \sum_{n=1-2N}^{-N}\left(\frac s\beta\right)^{-\mu N}\frac{|\hat G_n(s)|}{s}\leq \left(\frac s\beta\right)^{-\mu N}N^{1/2}s^{-2\mu}. $$
Then (using $N\geq2$, $\mu>1/2$)\begin{align*}
	&\int_0^\beta \sum_{n=N}^{2N-1}\left(\frac s\beta\right)^{\mu n}\frac{|\hat G_n(s)|}{s}\,ds+\int_\beta^\infty\sum_{n=1-2N}^{-N}\left(\frac s\beta\right)^{\mu n}\frac{|\hat G_n(s)|}{s}\,ds\\
	\leq&\int_0^\beta \left(\frac s\beta\right)^{\mu N}N^{1/2}s^{-2\mu}\,ds+\int_\beta^\infty\left(\frac s\beta\right)^{-\mu N}N^{1/2}s^{-2\mu}\,ds\\
	=&\frac{\beta^{1-2\mu}N^{1/2}}{\mu N-2\mu+1}+\frac{\beta^{1-2\mu}N^{1/2}}{\mu N+2\mu-1}\lesssim
	N^{-1/2}\beta^{1-2\mu}.
\end{align*}
This completes the proof of \eqref{a4} and hence \eqref{a1}.
\end{proof}	
By \eqref{a1}--\eqref{a3}, we have $\frac1{2\pi}\int_0^\infty|\tilde Q(\beta, \phi, s)|\,ds<+\infty$, $\forall \beta>0,\phi\in\TT$, even if	
  \begin{equation*}
   \int_{0}^\infty \int_\TT\left|\frac{s^\mu e^{-\ii(s+\Phi)}}{\beta^\mu e^{-\ii(\beta+\phi)}-s^\mu e^{-\ii(s+\Phi)}}
   \frac{G(s, \Phi)}{s}\right|\,d\Phi\,ds=+\infty.
  \end{equation*}
  By \eqref{a1}, \eqref{Qn1} and \eqref{Qn0}, we have $\sum_{n\in\ZZ}|\hat Q_n(\beta)|<+\infty$ and
  $Q(\beta,\phi)=\sum_{n\in\ZZ}\hat Q_n(\beta)e^{\ii n(\beta+\phi)} $ is absolutely convergent.
  Moreover, by \eqref{a4}, \eqref{Qn1} and \eqref{Qn0}, we have
  $$\sum_{N\leq |n|<2N}|\hat Q_n(\beta)|\lesssim
N^{-1/2}\beta^{1-2\mu},\quad \sum_{|n|\geq N}|\hat Q_n(\beta)|\lesssim
N^{-1/2}\beta^{1-2\mu},\quad \forall\ N\in\ZZ\cap[2,+\infty).$$
Thus, $Q(\beta,\phi)=\sum_{n\in\ZZ}\hat Q_n(\beta)e^{\ii n(\beta+\phi)} $ is locally uniformly convergent in $ (0,+\infty)\times\TT$ (which implies that $Q(\beta,\phi)$ is continuous as $\hat Q_n(\beta)$ is).
By \eqref{QnGn}, we have  $G(\beta,\phi)=\sum_{n\in\ZZ}\hat G_n(\beta)e^{\ii n(\beta+\phi)}  $ in $L_{\text{loc}}^2((0,+\infty)\times\TT)$.
We also have (using \eqref{Qn1}, \eqref{Qn0} and $\p_\vp=\p_\phi-\p_\beta$) $$ (\beta\p_\vp+\ii\mu\p_\phi)\left[\hat Q_n(\beta)e^{\ii n(\beta+\phi)}\right]=\left[-(\beta\p_\beta+\mu n)\hat Q_n(\beta)\right]e^{\ii n(\beta+\phi)}=-\hat G_n(\beta)e^{\ii n(\beta+\phi)},\quad  \forall\ n\in\ZZ.$$
   Thus, the function $Q$ defined in \eqref{Q_expression}
 satisfies the equation \eqref{linear_basic_eq} in the sense
   of distribution.\smallskip		

\subsection{Proof of \eqref{linear_basic_est}}
Next we prove \eqref{linear_basic_est}. Recall that we assume  $\|\beta^{2\mu-1}G\|_{C_\beta^\alpha}=1$.

We  first make several reductions. Recalling the equation \eqref{linear_basic_eq}, it suffices to show that
		\[\|\beta^{2\mu-1}\p_\phi Q\|_{C_\beta^\alpha}+\|\beta^{2\mu-1}Q\|_{C_\beta^\alpha}\lesssim1.\]
		{As $ \hat{G}_0=\hat G_{\pm1}=0,$ $ \hat{Q}_0=\hat Q_{\pm1}=0,$ \eqref{reverse_bernstein2} implies $\|\beta^{2\mu-1}Q\|_{C_\beta^\alpha} \lesssim
\|\beta^{2\mu-1}\p_\phi Q\|_{C_\beta^\alpha}$,} then it suffices to show that $\|\beta^{2\mu-1}\p_\phi Q\|_{C_\beta^\alpha}\lesssim1$.
		Equation \eqref{linear_basic_eq} and $\p_\vp=\p_\phi-\p_\beta$ also imply that
		\[\beta^{2\mu-1}\p_\phi Q=\beta^{2\mu-1}\frac{\beta\p_\beta Q}{\beta+\ii\mu}-\frac1{\beta+\ii\mu}\beta^{2\mu-1}G.\]
		As a consequence, it suffices to show that (also using $\frac1{\beta+\ii\mu}\in C_\beta^\alpha$ and Lemma \ref{AppendixA_Algebra})
		\begin{equation}\label{Q3}\left\|\beta^{2\mu-1}\frac{\beta\p_\beta Q}{\beta+\ii\mu}\right\|_{C_\beta^\alpha}\lesssim1.	\end{equation}
		
		Let $\eta\in C_c^\infty(\RR; [0,1])$ be such that
		$$\text{supp }\eta\subset\left[-\frac34,\frac34\right],\qquad \eta\Big|_{\left[-\frac14,\frac14\right]}\equiv 1,\ \ \text{ and }\ \ \sum_{k=0}^\infty \eta(\beta-k)=1,\ \ {\forall\ }\beta>0.$$
		For each $k\in\ZZ_{\geq0}$, we define $G_k(\beta,\phi):=G(\beta,\phi)\eta(\beta-k)$, then $G=\sum\limits_{k=0}^\infty G_k$.
We define
		\begin{equation}\label{Q_k_expression}
			Q_k(\beta,\phi):=\frac1{2\pi}\int_0^\infty\left(\int_\TT\frac{s^\mu e^{-\ii(s+\Phi)}}{\beta^\mu e^{-\ii(\beta+\phi)}-s^\mu e^{-\ii(s+\Phi)}}
\frac{G_k(s, \Phi)}{s}\,d\Phi\right)\,ds.
		\end{equation}
		By \eqref{Q_expression} and \eqref{Qt}, we have\footnote{We emphasize here that \eqref{Q_k_expression} and \eqref{Q_expression_1} should be viewed as iterated integrals, although we will see in \eqref{Q_k>0} below that for $k\in\ZZ_+$ we can drop the bracket in \eqref{Q_k_expression} and view it as a Lebesgue integral (or a double integral) with the integral domain $(0,+\infty)\times\TT$.}
		\begin{equation}\label{Q_expression_1}
			Q(\beta,\phi)=\frac1{2\pi}\int_0^\infty\left(\int_\TT\frac{s^\mu e^{-\ii(s+\Phi)}}{\beta^\mu e^{-\ii(\beta+\phi)}-s^\mu e^{-\ii(s+\Phi)}}
\frac{G(s, \Phi)}{s}\,d\Phi\right)\,ds.
		\end{equation}We also have (using $\beta^{2\mu-1}G\in C_\beta^\alpha$ and $1-\eta(s)=0 $ for $s\leq 1/4$)
		\begin{align*}
			&\int_0^\infty\int_\TT\sum\limits_{k=1}^\infty\left|\frac{s^\mu e^{-\ii(s+\Phi)}}{\beta^\mu e^{-\ii(\beta+\phi)}-s^\mu e^{-\ii(s+\Phi)}}
\frac{G_k(s, \Phi)}{s}\right|\,d\Phi\,ds\\
&=\int_0^\infty\int_\TT\left|\frac{s^\mu e^{-\ii(s+\Phi)}}{\beta^\mu e^{-\ii(\beta+\phi)}-s^\mu e^{-\ii(s+\Phi)}}
\frac{G(s, \Phi)}{s}\right|(1-\eta(s))\,d\Phi\,ds<+\infty
		\end{align*}
		for each $\beta>0$ and $\phi\in\TT$.
		Then {by Fubini's theorem}, we have $Q=\sum\limits_{k=0}^\infty Q_k$ and (by \eqref{Q_k_expression})\begin{equation}\label{Q_k>0}
			Q_k(\beta,\phi)=\frac1{2\pi}\int_0^\infty\int_\TT\frac{s^\mu e^{-\ii(s+\Phi)}}{\beta^\mu e^{-\ii(\beta+\phi)}-s^\mu e^{-\ii(s+\Phi)}}
\frac{G_k(s, \Phi)}{s}\,d\Phi\,ds,\quad \forall\ k\in\ZZ_+.
		\end{equation}
		Moreover,
$\text{supp }G_k(\cdot, \phi)\subset \left[k-\frac34,k+\frac34\right]$ for $k\geq 1$ and $\text{supp }G_0(\cdot, \phi)\subset \left[0,\frac34\right]$.
Thanks to $\|\beta^{2\mu-1}G\|_{C_\beta^\alpha}=1$, we have $\|G_k\|_{L^\infty}\lesssim k^{1-\alpha-2\mu}$ for $k\geq 1$
and $	\|\beta^{2\mu-1}G_0\|_{L^\infty}\lesssim1$.

For each $k\geq0$, we claim that $\pa_\beta Q_k\in C((0,+\infty)\times\TT)$ and we define $Q_k^{(1)}(\beta, \phi):=\frac\beta{\beta+\ii\mu}\p_\beta Q_k(\beta,\phi) $. Our desired inequality \eqref{Q3} will follow from the following lemma.

\begin{lem}\label{lem1} The following key estimates hold.
	\begin{enumerate}[(i)]
		\item If $\beta>0$, then $\left|Q_0^{(1)}(\beta,\phi)\right|\lesssim \beta^{1-2\mu}/(1+\beta)$.
		\item If $\beta>0$, $k\in\ZZ_+$ then  $\left|Q_k^{(1)}(\beta,\phi)\right|\lesssim(|\beta-k|+1)^{\alpha-1}(\beta+k)^{1-\alpha-2\mu}k^{-\alpha}$.
	\end{enumerate}	
	
	For any fixed $\beta_1,\beta_2>0,\phi\in\TT$ and $k\in\ZZ_{\geq0}$, let
	\begin{equation}\label{Ik}
		I_k=I_k(\beta_1,\beta_2,\phi):=\left|\beta_1^{2\mu-1} Q_k^{(1)}(\beta_1,\phi)-\beta_2^{2\mu-1} Q_k^{(1)}(\beta_2,\phi)\right|.
	\end{equation}
	\begin{enumerate}
		\item[(iii)]  If $0<\beta_1/2<\beta_2<\beta_1<2$, then
		$I_k\lesssim\left|\frac{\beta_1-\beta_2}{\beta_1+\beta_2}\right|^\alpha(1+k)^{-2\mu-\alpha}$.
		
		\item[(iv)] If $k_0\in\ZZ$, $k_0\geq 2$, $\beta_1,\beta_2\in(k_0-1,k_0+1)$, $0<\beta_1-\beta_2<1 $, then
		$I_k\lesssim\left|\frac{\beta_1-\beta_2}{\beta_1+\beta_2}\right|^\alpha$ for $k\in\{0,k_0-1,k_0,k_0+1\}$, and
		$$I_k\lesssim\begin{cases}
			|k-k_0|^{\alpha-1}k_0^{-2\alpha}, & \text{if }2\leq |k-k_0|\leq|\beta_1-\beta_2|k_0/2,\\
			|\beta_1-\beta_2|k_0^{2\mu-1}|k_0-k|^{\alpha-2}(k_0+k)^{2-\alpha-2\mu}k^{-\alpha}, &\text{if }|k-k_0|\geq\max\{2,|\beta_1-\beta_2|k_0/2\}
		\end{cases}$$ for $k\in\ZZ_+\setminus[k_0-1,k_0+1]$.
	\end{enumerate}
	In (i) and (ii) the implicit constants in $\lesssim$ are independent of $\beta,\phi,k$, and  in (iii) and (iv) the implicit constants in $\lesssim$ are independent of $\beta_1, \beta_2,\phi$ and $k, k_0$, and as usual, they only depend on $\alpha,\mu$ and the bump function $\eta$ and $\rho$ (introduced in the proof of Lemma \ref{lem5}).
	\end{lem}	
	
Assuming Lemma \ref{lem1} for the moment,  we prove \eqref{Q3}.

\begin{proof}[Proof of \eqref{Q3}]

	\begin{enumerate}[(1)]
		\item \underline{$\left\|\langle\beta\rangle^\alpha\beta^{2\mu-1}\frac{\beta\p_\beta Q}{\beta+\ii\mu}\right\|_{L^\infty}\lesssim1$.} By  Lemma \ref{lem1} (i) (ii),
		{$\mu>1/2$ and $\alpha\in(0,1)$}, we know that $\sum_{k=0}^\infty|Q_k^{(1)}(\beta,\phi)|$ is convergent locally uniformly in $(0,+\infty)\times\TT$, hence
		$\p_\beta Q\in C((0,+\infty)\times\TT)$, and {(recall that $Q=\sum\limits_{k=0}^\infty Q_k$, $Q_k^{(1)}(\beta, \phi)=\frac\beta{\beta+\ii\mu}\p_\beta Q_k(\beta,\phi) $)}
		\begin{align}
			\label{Qk1}\frac{\beta\p_\beta Q(\beta,\phi)}{\beta+\ii\mu}&=\sum_{k=0}^\infty Q_k^{(1)}(\beta,\phi),\\
			\label{Qk2}\left|\langle\beta\rangle^\alpha\beta^{2\mu-1}\frac{\beta\p_\beta Q(\beta,\phi)}{\beta+\ii\mu}\right|&{\leq}
			\langle\beta\rangle^\alpha\beta^{2\mu-1}\sum_{k=0}^\infty \left|Q_k^{(1)}(\beta,\phi)\right|.\end{align}
		
		{\bf Case I. $\beta\in(0,2)$.} By Lemma \ref{lem1} (i), we have $|Q_0^{(1)}(\beta,\phi)|\lesssim \beta^{1-2\mu}$;
		by Lemma \ref{lem1} (ii), we have $\left|Q_k^{(1)}(\beta,\phi)\right|\lesssim k^{-2\mu-\alpha}$ for $k\geq 1$.
		Recalling that $\mu>1/2$ and $\alpha\in(0,1)$, we obtain (using \eqref{Qk2})
		\[\left|\langle\beta\rangle^\alpha\beta^{2\mu-1}\frac{\beta\p_\beta Q(\beta,\phi)}{\beta+\ii\mu}\right|\lesssim\langle\beta\rangle^\alpha\beta^{2\mu-1}
		\left(\beta^{1-2\mu}+1\right)\lesssim \langle\beta\rangle^\alpha\lesssim 1.\]
		
		{\bf Case II. $\beta\geq2$.} By Lemma \ref{lem1} (ii), we have
		\begin{align*}
			\sum_{k=1}^\infty \left|Q_k^{(1)}(\beta,\phi)\right|&\lesssim\sum_{k=1}^\infty (|\beta-k|+1)^{\alpha-1}(\beta+k)^{1-\alpha-2\mu}k^{-\alpha}\\
			&\lesssim  \beta^{1-\alpha-2\mu}\sum_{1\leq k\leq\frac\beta2}\beta^{\alpha-1}k^{-\alpha}+\beta^{1-\alpha-2\mu}\sum_{\frac\beta2\leq k\leq2\beta}
			(|\beta-k|+1)^{\alpha-1}\beta^{-\alpha}\\
			&\qquad\qquad\qquad+\sum_{k\geq 2\beta}k^{\alpha-1}k^{1-\alpha-2\mu}k^{-\alpha}\\
			&\lesssim  \beta^{1-\alpha-2\mu}.
		\end{align*}
		Combining this with Lemma \ref{lem1} (i) {($|Q_0^{(1)}(\beta,\phi)|\lesssim \beta^{-2\mu}$ as $\beta\geq2$) and \eqref{Qk2}}, we get
		\[\left|\langle\beta\rangle^\alpha\beta^{2\mu-1}\frac{\beta\p_\beta Q(\beta,\phi)}{\beta+\ii\mu}\right|\lesssim\langle\beta\rangle^\alpha\beta^{2\mu-1}
		\left(\beta^{-2\mu}+\beta^{1-\alpha-2\mu}\right)\lesssim 1.\]
		
		\item \underline{$\left\|\beta^{2\mu-1}\frac{\beta\p_\beta Q}{\beta+\ii\mu}\right\|_{C_\beta^\alpha}\lesssim1$.} Now we estimate the H\"older norm. Fix
		$0<\beta_2<\beta_1<2\beta_2$ with $\beta_1-\beta_2<1$. We start with {(using \eqref{Qk1} and \eqref{Ik})}
		\begin{align*}
			&\left|\beta_1^{2\mu-1}\frac{\beta_1\p_\beta Q(\beta_1,\phi)}{\beta_1+\ii\mu}-\beta_2^{2\mu-1}\frac{\beta_2\p_\beta Q(\beta_2,\phi)}{\beta_2+\ii\mu}\right|\\
			&{\leq}\sum_{k=0}^\infty \left|\beta_1^{2\mu-1} Q_k^{(1)}(\beta_1,\phi)-\beta_2^{2\mu-1} Q_k^{(1)}(\beta_2,\phi)\right|=\sum_{k=0}^\infty I_k.
		\end{align*}
		
		{\bf Case I. $\beta_2<\beta_1<2$.} By Lemma \ref{lem1} (iii) and $\mu>1/2,\alpha>0$, we have
		\[\sum_{k=0}^\infty I_k\lesssim\sum_{k=0}^\infty \left|\frac{\beta_1-\beta_2}{\beta_1+\beta_2}\right|^\alpha (1+k)^{-2\mu-\alpha}\lesssim \left|\frac{\beta_1-\beta_2}{\beta_1+\beta_2}\right|^\alpha.\]
		
		{\bf Case II. $1<\beta_2<\beta_1$.}  Fix an integer $k_0\geq 2$ such that $\beta_1,\beta_2\in(k_0-1, k_0+1)$. By Lemma \ref{lem1} (iv), we have
		\begin{align*}
			\sum_{|k-k_0|\geq 2 \atop k\in \ZZ_+}I_k&=\sum_{2\leq |k-k_0|\leq|\beta_1-\beta_2|k_0/2 \atop k\in \ZZ_+}I_k+\sum_{|k-k_0|\geq \max\{2,|\beta_1-\beta_2|k_0/2\}
				\atop k\in \ZZ_+}I_k\\
			&\lesssim k_0^{-2\alpha}\sum_{2\leq |k-k_0|\leq|\beta_1-\beta_2|k_0/2 \atop k\in \ZZ_+}|k-k_0|^{\alpha-1}\\
			&\qquad\qquad+|\beta_1-\beta_2|k_0^{2\mu-1}\sum_{|k-k_0|\geq|\beta_1-\beta_2|k_0/2\atop k\in\ZZ_+}|k-k_0|^{\alpha-2}(k_0+k)^{2-\alpha-2\mu}k^{-\alpha}.
		\end{align*}
		Since
		\begin{align*}&\sum_{|k-k_0|\geq|\beta_1-\beta_2|k_0/2\atop k\in\ZZ_+}|k_0-k|^{\alpha-2}(k_0+k)^{2-\alpha-2\mu}k^{-\alpha}\\
			&\lesssim\sum_{|k-k_0|\geq|\beta_1-\beta_2|k_0/2\atop k_0/2\leq k\leq 2k_0}|k_0-k|^{\alpha-2}k_0^{2-\alpha-2\mu}k_0^{-\alpha}+
			\sum_{1\leq k<k_0/2}k_0^{\alpha-2}k_0^{2-\alpha-2\mu}k^{-\alpha}+\sum_{k\geq 2k_0} k^{-2\mu-\alpha}\\
			&\lesssim (k_0|\beta_1-\beta_2|)^{\alpha-1}k_0^{2-2\alpha-2\mu}+k_0^{\alpha-2}k_0^{2-\alpha-2\mu}k_0^{1-\alpha}+k_0^{1-\alpha-2\mu}\\
			&=|\beta_1-\beta_2|^{\alpha-1}k_0^{1-\alpha-2\mu}+k_0^{1-\alpha-2\mu}\lesssim |\beta_1-\beta_2|^{\alpha-1}k_0^{1-\alpha-2\mu},\end{align*}
		we have
		\[\sum_{|k-k_0|\geq 2 \atop k\in \ZZ_+}I_k{\lesssim k_0^{-2\alpha}(k_0|\beta_1-\beta_2|)^{\alpha}+|\beta_1-\beta_2|k_0^{2\mu-1}|\beta_1-\beta_2|^{\alpha-1}k_0^{1-\alpha-2\mu}}\lesssim|\beta_1-\beta_2|^\alpha k_0^{-\alpha}.\]
		Using  Lemma \ref{lem1} (iv){(for $k\in\{0,k_0-1,k_0,k_0+1\}$)} again, we obtain
		\[\sum_{k=0}^\infty I_k{\lesssim\left|\frac{\beta_1-\beta_2}{\beta_1+\beta_2}\right|^\alpha+|\beta_1-\beta_2|^\alpha k_0^{-\alpha}}\lesssim\left|\frac{\beta_1-\beta_2}{\beta_1+\beta_2}\right|^\alpha.\]
	\end{enumerate}
	Putting everything all together, we have shown the validity of \eqref{Q3}.
\end{proof}

\subsection{Proof of Lemma \ref{lem1}}
We first make some preparations. Let \begin{align}\label{z1}z:=\beta^\mu e^{-\ii(\beta+\phi)},\quad \gamma(s,\Phi):=s^\mu e^{-\ii(s+\Phi)}.\end{align}

\begin{lem}\label{lem2}
\begin{enumerate}[(i)]
	\item If $0<\beta\leq \frac15k$, $k\in\ZZ_+$, then $|Q_k^{(1)}(\beta,\phi)|+|\partial_{\beta}Q_k^{(1)}(\beta,\phi)|\lesssim k^{-2\mu-\alpha}$.
	\item If $\beta\geq2k$, $k\in\ZZ_+$, then $|Q_k^{(1)}(\beta,\phi)|+|\partial_{\beta}Q_k^{(1)}(\beta,\phi)|\lesssim\beta^{-2\mu} k^{-\alpha}$.
\end{enumerate}
\end{lem}

\begin{proof}
	Recalling that $\text{supp }G_k(\cdot, \phi)\subset \left[k-\frac34,k+\frac34\right]$, {by \eqref{z1}}, \eqref{Q_k>0} can be rewritten as
	\begin{align}\label{Qk3}Q_k(\beta,\phi)=\frac1{2\pi}\int_{k-\frac34}^{k+\frac34}\int_\TT\frac{\gamma(s,\Phi)}{z-\gamma(s,\Phi)}\frac{G_k(s,\Phi)}{s}\,d\Phi\,ds.\end{align}
	
	\begin{enumerate}[(i)]
		\item $\beta\leq \frac15k$.
	For $s\in[k-3/4,k+3/4]$, it follows from $\beta\leq\frac15k\leq\frac45(k-3/4)$ that $\beta\leq\frac45s$, and thus $|z-\gamma(s,\Phi)|\geq |\gamma(s,\Phi)|-|z|=s^\mu-\beta^\mu\gtrsim s^\mu$, {and we can take the derivative under the integral \eqref{Qk3}}. Recall that $\hat G_{-1}=0$, hence $\int_\TT\frac1{\gamma(s,\Phi)}G_k(s,\Phi)\,d\Phi=0$ and then by $\frac{\beta}{\beta+\ii\mu}\p_\beta z=-\ii z$,
	\begin{equation}\label{5.14}
		\begin{aligned}			Q_k^{(1)}(\beta,\phi)&=\frac{\beta\p_\beta Q_k(\beta,\phi)}{\beta+\ii\mu}=\frac{\ii z}{2\pi}\int_{k-\frac34}^{k+\frac34}\int_\TT\frac{\gamma(s,\Phi)}{\left(z-\gamma(s,\Phi)\right)^2}\frac{G_k(s,\Phi)}{s}\,d\Phi\,ds\\
			&=\frac{\ii z}{2\pi}\int_{k-\frac34}^{k+\frac34}\int_\TT\left(\frac{\gamma(s,\Phi)}{\left(z-\gamma(s,\Phi)\right)^2}-\frac1{\gamma(s,\Phi)}\right)\frac{G_k(s,\Phi)}{s}\,d\Phi\,ds\\
			&=\frac{\ii z^2}{2\pi} \int_{k-\frac34}^{k+\frac34}\int_\TT\frac{2\gamma(s,\Phi)-z}{\gamma(s,\Phi)\left(z-\gamma(s,\Phi)\right)^2}\frac{G_k(s,\Phi)}{s}\,d\Phi\,ds.
		\end{aligned}
	\end{equation}
This along with $\|G_k\|_{L^\infty}\lesssim k^{1-\alpha-2\mu}$ { and $|z|=\beta^\mu$} gives
	\[\left|Q_k^{(1)}(\beta,\phi)\right|\lesssim\beta^{2\mu}\frac{k^\mu}{k^{3\mu}}k^{-\alpha-2\mu}\lesssim \beta^{2\mu}k^{-\alpha-4\mu}\lesssim k^{-\alpha-2\mu}.\]
	Using $\p_\beta z=\frac{\mu-\ii\beta}{\beta}z$, we infer from \eqref{5.14} that
	\begin{align*}
		\p_\beta Q_k^{(1)}(\beta,\phi)=\frac{\mu-\ii\beta}{\beta}Q_k^{(1)}(\beta,\phi)-2\ii \frac{\mu-\ii\beta}{2\pi\beta}z^2\int_{k-\frac34}^{k+\frac34}\int_\TT\frac{\gamma(s,\Phi)}{\left(z-\gamma(s,\Phi)\right)^3}\frac{G_k(s,\Phi)}{s}\,d\Phi\,ds,
	\end{align*}
	and then (recalling $\mu>1/2$)
	\[\left|\p_\beta Q_k^{(1)}(\beta,\phi)\right|\lesssim\frac k\beta \beta^{2\mu}k^{-\alpha-4\mu}+\frac k\beta\beta^{2\mu}\frac{k^\mu}{k^{3\mu}}k^{-\alpha-2\mu}\lesssim \beta^{2\mu-1}k^{1-\alpha-4\mu}\lesssim k^{-\alpha-2\mu}.\]\smallskip
	
	\item $\beta\geq2k$. In this case, for $s\in[k-3/4, k+3/4]$, one has $\beta\geq 2k\geq \frac87(k+3/4)\geq \frac87s$, hence
	$|z-\gamma(s,\Phi)|\geq \beta^\mu-s^\mu\gtrsim \beta^\mu$, {thus we can take the derivative under the integral \eqref{Qk3}}. Recall that $\hat G_1=0$, hence $\int_\TT\gamma(s,\Phi)G_k(s,\Phi)\,d\Phi=0$ and then
	using $\frac{\beta}{\beta+\ii\mu}\p_\beta z=-\ii z$, we have
	\begin{align*}
		Q_k^{(1)}(\beta,\phi)&=\frac{\beta\p_\beta Q_k(\beta,\phi)}{\beta+\ii\mu}=\frac{\ii z}{2\pi}\int_{k-\frac34}^{k+\frac34}\int_\TT\frac{\gamma(s,\Phi)}{\left(z-\gamma(s,\Phi)\right)^2}\frac{G_k(s,\Phi)}{s}\,d\Phi\,ds\\&=\frac{\ii z}{2\pi}\int_{k-\frac34}^{k+\frac34}\int_\TT\gamma(s,\Phi)\left(\frac{1}{\left(z-\gamma(s,\Phi)\right)^2}-\frac1{z^2}\right)\frac{G_k(s,\Phi)}{s}\,d\Phi\,ds\\
		&=\frac{\ii z}{2\pi}\int_{k-\frac34}^{k+\frac34}\int_\TT\frac{\left(\gamma(s,\Phi)\right)^2\left(2z-\gamma(s,\Phi)\right)}{z^2\left(z-\gamma(s,\Phi)\right)^2}
		\frac{G_k(s,\Phi)}{s}\,d\Phi\,ds
	\end{align*}
	with the bound
	\[\left|Q_k^{(1)}(\beta,\phi)\right|\lesssim\beta^\mu \frac{k^{2\mu}\ \beta^\mu}{\beta^{2\mu}\ \beta^{2\mu}}k^{-\alpha-2\mu}\lesssim \beta^{-2\mu} k^{-\alpha}.\]
	And by $\p_\beta z=\frac{\mu-\ii\beta}{\beta}z$,
	\begin{align*}
		&\p_\beta Q_k^{(1)}(\beta,\phi)-\frac{\mu-\ii\beta}{\beta}Q_k^{(1)}(\beta,\phi)\\
		=&-2\ii \frac{\mu-\ii\beta}{2\pi\beta}z^2\int_{k-\frac34}^{k+\frac34}\int_\TT\gamma(s,\Phi)\left(\frac{1}{\left(z-\gamma(s,\Phi)\right)^3}-\frac1{z^3}\right)
		\frac{G_k(s,\Phi)}{s}\,d\Phi\,ds\\
		=&-2\ii \frac{\mu-\ii\beta}{2\pi\beta}z^2\int_{k-\frac34}^{k+\frac34}\frac{\left(\gamma(s,\Phi)\right)^2\left(3z^2-3z\gamma(s,\Phi)+
			\left(\gamma(s,\Phi)\right)^2\right)}{z^3\left(z-\gamma(s,\Phi)\right)^3}\frac{G_k(s,\Phi)}{s}\,ds,
	\end{align*}
	hence,
	\[\left|\p_\beta Q_k^{(1)}(\beta,\phi)\right|\lesssim\beta^{-2\mu} k^{-\alpha}+\beta^{2\mu} \frac{k^{2\mu}\ \beta^{2\mu}}{\beta^{3\mu}\ \beta^{3\mu}}k^{-\alpha-2\mu}\lesssim\beta^{-2\mu} k^{-\alpha}.\]
\end{enumerate}

This completes the proof.
\end{proof}

\begin{lem}\label{lem3}
Assume that $k\in\ZZ_+$, $\frac1{10}k<\beta<3k$.  We define {(see \eqref{z1})}
	\begin{equation}\label{Kk1}
			\widetilde{K}_k(\beta,\phi,\Phi)=\int_{k-\frac34}^{k+\frac34}\frac{\gamma(s,\Phi)}{z-\gamma(s,\Phi)}\frac{G_k(s,\Phi)}{s}\,ds,\quad \forall\ \phi\neq\Phi,
		\end{equation}and $K_k(\beta,\phi,\Phi)=\frac{\beta}{\beta+\ii\mu}\p_\beta\tilde K_k(\beta,\phi,\Phi)$ {for $\phi\neq\Phi\in\TT$}.
\begin{enumerate}[(i)]
	\item If $|\beta-k|\geq 1$, then $|\p_\beta^j K_k(\beta, \phi, \Phi)|\lesssim \left|\frac{\beta-k}{k}\right|^{\alpha-1-j}k^{-2\mu-\alpha}$ for $j=0,1$.
	\item \footnote{We define $\|a\|_{\mathbb{T}}=\inf_{k\in\mathbb{Z}}|a-2\pi k|$ for $a\in\mathbb{R}$.
		For $\phi\in\mathbb{T}$ we define $\|\phi\|_{\mathbb{T}}=\|a\|_{\mathbb{T}}$ for $a\in\mathbb{R}$ such that $p_0(a)=\phi$, where
		$p_0:\mathbb{R}\to\mathbb{T}=\mathbb{R}/(2\pi\mathbb{Z})$ is the quotient map. It is well defined as $a\mapsto\|a\|_{\mathbb{T}}$ is a $2\pi$-periodic function for
		$a\in\mathbb{R}$. We also have $\|a+b\|_\TT\leq \|a\|_\TT+\|b\|_\TT$ for all $a,b\in\RR$.} $|\p_\beta^j K_k(\beta, \phi, \Phi)|\lesssim\|\phi-\Phi\|_{\TT}^{\alpha-1-j}k^{1+j-2\alpha-2\mu}$ for $j=0,1$, $\phi\neq\Phi$.
	\item We have\footnote{In view of (i) and (ii), we know that $\int_\TT K_k(\beta,\phi,\Phi)\,d\Phi$ is absolutely convergent.} $Q_k(\beta, \phi)=\frac1{2\pi}\int_\TT \widetilde{K}_k(\beta, \phi, \Phi)\,d\Phi $, $Q_k^{(1)}(\beta, \phi)=\frac1{2\pi}\int_\TT K_k(\beta, \phi, \Phi)\,d\Phi $.
\end{enumerate}
	\end{lem}
	
\begin{proof}
	It follows from {\eqref{Qk3} and \eqref{Kk1}} 
that
	\begin{equation}\label{Qk4}Q_k(\beta,\phi)=\frac1{2\pi}\int_\TT\wt K_k(\beta,\phi,\Phi)\,d\Phi,\qquad k\in\ZZ_+, \beta>0, \phi\in\TT.\end{equation}
	Define 
	 $\tilde G_k(s,\Phi):=\frac{\gamma(s,\Phi)}{\p_s\gamma(s,\Phi)}\frac{G_k(s,\Phi)}s=\frac{G_k(s,\Phi)}{\mu-\ii s}$ {(see \eqref{z1})},
	then
	\begin{equation}\label{tilde_K_k}
		\tilde K_k(\beta,\phi,\Phi)=\int_{k-\frac34}^{k+\frac34}\frac{\p_s\gamma(s,\Phi)}{z-\gamma(s,\Phi)}\tilde G_k(s,\Phi)\,ds.
	\end{equation}
	By the definition, $\tilde G_k(s,\Phi)=\frac{\eta(s-k)}{s^{2\mu-1}(\mu-\ii s)}\cdot s^{2\mu-1}G(s,\Phi)$. Using \eqref{f1} and $\text{supp }\eta\subset[-3/4, 3/4]$, one can easily show that $\left\|\langle\beta\rangle^{-\alpha}\frac{\eta(\beta-k)}{\beta^{2\mu-1}(\mu-\ii \beta)}\right\|_{C_\beta^\alpha}\lesssim k^{-2\mu}$ for $k\in\ZZ_+$. It follows from the algebra property (Lemma \ref{AppendixA_Algebra}) and $\|\beta^{2\mu-1}G\|_{C_\beta^\alpha}=1$ that
	\begin{equation}\label{tilde_G_k_est}
		\left\|\tilde G_k\right\|_{C_\beta^\alpha}\lesssim k^{-2\mu},\qquad k\geq1.
	\end{equation}
	It also follows from $\text{supp }G_k(\cdot, \phi)\subset \left[k-\frac34,k+\frac34\right]$ that we can change the integral domain in \eqref{tilde_K_k} into any interval containing $\left[k-\frac34,k+\frac34\right]$. Due to technical reasons, the proof for this case involves careful choices of the larger integral intervals, depending on the range of the parameters $\beta,\phi$ and $\Phi$. \smallskip
	
	\begin{enumerate}[(i)]
		\item $|\beta-k|\geq 1$. Pick $s_0\in\left[k-\frac67, k-\frac67+2\pi\right]$ such that $\beta+\phi-(s_0+\Phi)\in2\pi\ZZ$, so $z$ and $\gamma(s_0,\Phi)$ have the same arguments. For $s\in\left[k-\frac56, k+\frac56\right]$, we have {(using $\frac1{10}k<\beta<3k$)}
		\begin{equation}\label{z-gamma_1}
			\begin{aligned}
				\left|z-\gamma(s,\Phi)\right|^2&=\left|\beta^\mu e^{-\ii(s_0+\Phi)}-s^\mu e^{-\ii(s+\Phi)}\right|^2=\beta^{2\mu}\left|1-\left(\frac s\beta\right)^\mu e^{-\ii(s-s_0)}\right|^2\\
				&=\beta^{2\mu}\left(\left|1-\left(\frac s\beta\right)^\mu\right|^2+\left(\frac s\beta\right)^\mu\left(1-\cos(s-s_0)\right)\right)\\
				&\sim k^{2\mu}\left(\left|\frac{\beta-s}{k}\right|^2+\|s-s_0\|_\TT^2\right)\sim k^{2\mu}\left(\left|\frac{\beta-k}{k}\right|^2+|s-s_0|^2\right).
			\end{aligned}
		\end{equation}
		Since $|z-\gamma(s,\Phi)|^{2}\gtrsim k^{2\mu}|(\beta-k)/k|^2\gtrsim k^{2\mu-2}$, we obtain (using $\frac{\beta}{\beta+\ii\mu}\p_\beta z=-\ii z$)
		\begin{align}
			K_k(\beta,\phi, \Phi)&=\frac{\beta}{\beta+\ii\mu}\p_\beta\tilde K_k(\beta,\phi,\Phi)=\ii z\int_{k-\frac34}^{k+\frac34}\frac{\p_s\gamma(s,\Phi)}{\left(z-\gamma(s,\Phi)\right)^2}\tilde G_k(s,\Phi)\,ds,\label{Eq.K_k}\\
			\p_\beta K_k(\beta,\phi, \Phi)&=-2\ii \frac{\mu-\ii\beta}{\beta} z^2 \int_{k-\frac34}^{k+\frac34}\frac{\p_s\gamma(s,\Phi)}{\left(z-\gamma(s,\Phi)\right)^3}\tilde G_k(s,\Phi)\,ds+\frac{\mu-\ii\beta}{\beta}K_k(\beta,\phi,\Phi).\label{Eq.p_beta_K_k}
		\end{align}
		
		{\bf Case I. $s_0\in[k-4/5, k+4/5]$.} Due to $\operatorname{supp}\tilde G_k(\cdot,\phi)\subset\left[k-\frac34,k+\frac34\right]$, we have
		\begin{align*}
			K_k(\beta,\phi, \Phi)&=\ii z\int_{k-\frac56}^{k+\frac56}\frac{\p_s\gamma(s,\Phi)}{\left(z-\gamma(s,\Phi)\right)^2}\tilde G_k(s,\Phi)\,ds.
		\end{align*}
		To explore the effects of the $C_\beta^\alpha$ regularity of $\tilde G_k$, we rewrite
		\begin{align*}
			K_k(\beta,\phi, \Phi)&=\ii z\int_{k-\frac56}^{k+\frac56}\frac{\p_s\gamma(s,\Phi)}{\left(z-\gamma(s,\Phi)\right)^2}\left(\tilde G_k(s,\Phi)-\tilde G_k(s_0,\Phi)\right)\,ds\\
			&\qquad+\ii z\ \tilde G_k(s_0,\Phi)\int_{k-\frac56}^{k+\frac56}\frac{\p_s\gamma(s,\Phi)}{\left(z-\gamma(s,\Phi)\right)^2}\,ds.
		\end{align*}
		Now we are handling the case where $s\in\left[k-\frac56,k+\frac56\right], s_0\in\left[k-\frac45, k+\frac45\right], |\beta-k|\geq1, \frac1{10}k<\beta<3k$. By \eqref{z-gamma_1}, we have
		\begin{align*}
			\left|z-\gamma(s,\Phi)\right|&\sim k^\mu\left(\left|\frac{\beta-k}{k}\right|+|s-s_0|\right),\\
			\left|z-\gamma\left(k\pm\frac56,\Phi\right)\right|&\sim k^\mu\left(\left|\frac{\beta-k}{k}\right|+\left|k\pm\frac56-s_0\right|\right)\sim k^\mu.
		\end{align*}
		It also follows from Lemma \ref{equvi_Holder_norm} and \eqref{tilde_G_k_est} that
		\[\left|\tilde G_k(s,\Phi)-\tilde G_k(s_0,\Phi)\right|\lesssim k^{-2\mu}\frac{|s-s_0|^\alpha}{|s+s_0|^\alpha}\lesssim |s-s_0|^\alpha k^{-\alpha-2\mu}.\]
		Hence,
		\begin{align*}
			\left|K_k(\beta,\phi,\Phi)\right|&\lesssim k^\mu \int_{k-\frac56}^{k+\frac56}\frac{k^\mu|s-s_0|^\alpha k^{-\alpha-2\mu}}{k^{2\mu}\left(\left|\frac{\beta-k}{k}\right|^2+|s-s_0|^2\right)}\,ds\\
			&\qquad+k^\mu\cdot k^{-\alpha-2\mu}\left(\frac1{\left|z-\gamma\left(k+\frac56, \Phi\right)\right|}+\frac1{\left|z-\gamma\left(k-\frac56, \Phi\right)\right|}\right)\\
			&\lesssim \left|\frac{\beta-k}{k}\right|^{\alpha-1}k^{-\alpha-2\mu}+k^{-\alpha-2\mu}\lesssim\left|\frac{\beta-k}{k}\right|^{\alpha-1}k^{-\alpha-2\mu},
		\end{align*}
		where we used $\int_\RR \frac{|x|^\alpha}{x^2+A^2}\,dx\lesssim |A|^{\alpha-1}$ for $0<\alpha<1$. Similarly, we rewrite
		\begin{align*}
			&\p_\beta K_k(\beta,\phi,\Phi)=-2\ii \frac{\mu-\ii\beta}{\beta} z^2 \int_{k-\frac56}^{k+\frac56}\frac{\p_s\gamma(s,\Phi)}{\left(z-\gamma(s,\Phi)\right)^3}\left(\tilde G_k(s,\Phi)-\tilde G_k(s_0,\Phi)\right)\,ds\\
			&\qquad-2\ii\frac{\mu-\ii\beta}{\beta} z^2\ \tilde G_k(s_0,\Phi)\int_{k-\frac56}^{k+\frac56}\frac{\p_s\gamma(s,\Phi)}{\left(z-\gamma(s,\Phi)\right)^3}\,ds+\frac{\mu-\ii\beta}{\beta}K_k(\beta,\phi,\Phi)
		\end{align*}
		with the bound
		\begin{align*}
			\left|\p_\beta K_k(\beta,\phi,\Phi)\right|&\lesssim\left|\frac{\beta-k}{k}\right|^{\alpha-1}k^{-\alpha-2\mu}+k^{2\mu}\int_{k-\frac56}^{k+\frac56}\frac{k^\mu|s-s_0|^\alpha k^{-\alpha-2\mu}}{k^{3\mu}\left(\left|\frac{\beta-k}{k}\right|^3+|s-s_0|^3\right)}\,ds\\
			&\qquad+k^{2\mu}\cdot k^{-\alpha-2\mu}\left(\frac1{\left|z-\gamma\left(k+\frac56, \Phi\right)\right|^2}+\frac1{\left|z-\gamma\left(k-\frac56, \Phi\right)\right|^2}\right)\\
			&\lesssim\left|\frac{\beta-k}{k}\right|^{\alpha-1}k^{-\alpha-2\mu}+\left|\frac{\beta-k}{k}\right|^{\alpha-2}k^{-\alpha-2\mu}\lesssim
			\left|\frac{\beta-k}{k}\right|^{\alpha-2}k^{-\alpha-2\mu}.
		\end{align*}
		
		{\bf Case II. $s_0\in[k-6/7,k-6/7+2\pi]\setminus[k-4/5, k+4/5]$.} In this case, for any $s\in\left[k-\frac34,k+\frac34\right]$, we have $\left|\frac{\beta-k}k\right|\lesssim1$ and $|s-s_0|\sim 1$, hence by \eqref{z-gamma_1}, $$|z-\gamma(s,\Phi)|\sim k^\mu\left(\left|\frac{\beta-k}{k}\right|+|s-s_0|\right)\sim k^\mu.$$
		As a consequence, we get by \eqref{Eq.K_k}  that
		\[\left|K_k(\beta,\phi,\Phi)\right|\lesssim \beta^\mu\frac{k^\mu}{k^{2\mu}}k^{-\alpha-2\mu}\lesssim k^{-\alpha-2\mu}.\]
		Similarly, we have $\left|\p_\beta K_k(\beta,\phi,\Phi)\right|\lesssim k^{-\alpha-2\mu}.$
		
		{In summary}, for $|\beta-k|\geq 1$ and $\frac1{10}k<\beta<3k$, we arrive at
		\[\left|K_k(\beta, \phi,\Phi)\right|\lesssim \left|\frac{\beta-k}{k}\right|^{\alpha-1}k^{-\alpha-2\mu},\qquad \left|\p_\beta K_k(\beta, \phi,\Phi)\right|\lesssim\left|\frac{\beta-k}{k}\right|^{\alpha-2}k^{-\alpha-2\mu}.\]
		
		\item {If $|\beta-k|\geq 1$, by (i) we have $|\p_\beta^j K_k(\beta, \phi, \Phi)|\lesssim |\frac{\beta-k}{k}|^{\alpha-1-j}k^{-2\mu-\alpha}\leq |\frac{1}{k}|^{\alpha-1-j}k^{-2\mu-\alpha}=k^{1+j-2\alpha-2\mu}\lesssim \|\phi-\Phi\|_{\TT}^{\alpha-1-j}k^{1+j-2\alpha-2\mu}$ for $j=0,1$.
Now we assume $|\beta-k|\leq 2$.} The proof for $|\beta-k|\leq 2$ is similar to (i).
		Pick $s_0\in\left[k-\frac{13}5, k-\frac{13}5+2\pi\right]$ for $k\geq3$, $s_0\in\left[\frac{1}{10}, \frac{1}{10}+2\pi\right]$ for $k=1,2$ such that
		$\beta+\phi-(s_0+\Phi)\in2\pi\ZZ$, so $z$ and $\gamma(s_0,\Phi)$ have the same arguments. For
		$s\in\left[\max(k-\frac52,\frac{1}{10}), k+\frac52\right]=:[k_1, k_2]$, as in \eqref{z-gamma_1}
		we have
		\begin{equation}\label{z-gamma_2}
			\begin{aligned}
				\left|z-\gamma(s,\Phi)\right|&\sim k^\mu\left(\frac{|\beta-s|}{k}+|s-s_0|\right)\sim k^\mu\left(\frac{|\beta-s_0|}{k}+|s-s_0|\right)\\
				&\sim k^\mu\left(\frac{\|\phi-\Phi\|_\TT}{k}+|s-s_0|\right).
			\end{aligned}
		\end{equation}
		Since $\phi\neq\Phi$, we have $|z-\gamma(s,\Phi)|\gtrsim k^{\mu-1}\|\phi-\Phi\|_\TT$, hence we can take the derivative under the integral, then \eqref{Eq.K_k} and \eqref{Eq.p_beta_K_k} hold as well in this case. \smallskip
		
		{\bf Case I. $s_0\in[\max(k-2,1/5), k+2]$.} Similar to {\bf Case I} in (i), to make full use of the $C^\alpha_\beta$ regularity of $\tilde G_k$, we rewrite \eqref{Eq.K_k} in the following form
		\begin{align*}
			K_k(\beta,\phi,\Phi)&=\ii z\int_{\max(k-\frac52,\frac{1}{10})}^{k+\frac52}\frac{\p_s\gamma(s,\Phi)}{\left(z-\gamma(s,\Phi)\right)^2}\tilde G_k(s,\Phi)\,ds\\
			&=\ii z\int_{\max(k-\frac52,\frac{1}{10})}^{k+\frac52}\frac{\p_s\gamma(s,\Phi)}{\left(z-\gamma(s,\Phi)\right)^2}\left(\tilde G_k(s,\Phi)-\tilde G_k(s_0,\Phi)\right)\,ds\\
			&\qquad\qquad\qquad+\ii z\ \tilde G_k(s_0,\Phi)\int_{\max(k-\frac52,\frac{1}{10})}^{k+\frac52}\frac{\p_s\gamma(s,\Phi)}{\left(z-\gamma(s,\Phi)\right)^2}\,ds.
		\end{align*}
		By \eqref{z-gamma_2}, for $\left[\max(k-\frac52,\frac{1}{10}), k+\frac52\right]=[k_1,k_2]$, $s_0\in[\max(k-2,\frac{1}{5}), k+2]$, we have
		\[\left|z-\gamma\left(k_j,\Phi\right)\right|\sim k^\mu\left(\frac{\|\phi-\Phi\|_\TT}{k}+\left|k_j-s_0\right|\right)\sim k^\mu,\quad j=1,2.\]
		Following exactly the same proof in {\bf Case I} of (i), we obtain
		\[\left|K_k(\beta,\phi,\Phi)\right|\lesssim \left(\frac{\|\phi-\Phi\|_\TT}{k}\right)^{\alpha-1}k^{-\alpha-2\mu}\lesssim \|\phi-\Phi\|_\TT^{\alpha-1}k^{1-2\alpha-2\mu},\]
		and
		\[\left|\p_\beta K_k(\beta,\phi,\Phi)\right|\lesssim \left(\frac{\|\phi-\Phi\|_\TT}{k}\right)^{\alpha-2}k^{-\alpha-2\mu}\lesssim \|\phi-\Phi\|_\TT^{\alpha-2}k^{2-2\alpha-2\mu}.\]
		
		{\bf Case II}. $s_0\in[k-13/5,k-13/5+2\pi]\setminus[k-2, k+2]$ for $k\geq 3$ or $s_0\in[1/10,1/10+2\pi]\setminus[1/5, k+2]$ for $k=1,2$.  In this case, for any $s\in\left[k-\frac34,k+\frac34\right]$, we have $\frac{\|\phi-\Phi\|_\TT}k\lesssim1$ and $|s-s_0|\sim 1$, hence by \eqref{z-gamma_2}, $$|z-\gamma(s,\Phi)|\sim k^\mu\left(\frac{\|\phi-\Phi\|_\TT}{k}+|s-s_0|\right)\sim k^\mu.$$
		Along the same way as in {\bf Case II} of (i), one shows
		\[\left|K_k(\beta,\phi,\Phi)\right|+\left|\p_\beta K_k(\beta,\phi,\Phi)\right|\lesssim k^{-\alpha-2\mu}.\]
		
		{In summary}, for $k\geq 1$ and $\frac1{10}k<\beta<3k$, we obtain
		\[\left|\p_\beta^jK_k(\beta,\phi,\Phi)\right|\lesssim \|\phi-\Phi\|_\TT^{\alpha-1-j}k^{1+j-2\alpha-2\mu},\qquad j\in\{0,1\}, { \phi\neq\Phi}.\]
		
		\item Since $Q_k(\beta,\phi)=\frac1{2\pi}\int_\TT{\widetilde K_k}(\beta,\phi,\Phi)\,d\Phi$ (i.e. \eqref{Qk4}), it remains to prove $Q_k^{(1)}(\beta,\phi)=\frac1{2\pi}\int_\TT K_k(\beta,\phi,\Phi)\,d\Phi$ for $k\in\ZZ_+$.
By (ii) 
we have
$$\int_\TT \sup_{\frac1{10}k<\beta<3k}\left|K_k(\beta,\phi,\Phi)\right|\,d\Phi\lesssim \int_\TT\|\phi-\Phi\|_\TT^{\alpha-1}k^{1-2\alpha-2\mu}\,d\Phi\lesssim
			k^{1-2\alpha-2\mu}<+\infty.$$
		Then by $K_k=\frac{\beta}{\beta+\ii\mu}\p_\beta\wt K_k$, $Q_k^{(1)}=\frac{\beta}{\beta+\ii\mu}\p_\beta Q_k$, \eqref{Qk4} and the dominated convergence theorem, we have $Q_k^{(1)}(\beta,\phi)=\frac1{2\pi}\int_\TT K_k(\beta,\phi,\Phi)\,d\Phi$.
	\end{enumerate}
	
	This completes the proof.
\end{proof}	
	
\begin{lem}\label{lem4}
 If $\beta\geq1$, then $|Q_0^{(1)}(\beta,\phi)|+|\partial_{\beta}Q_0^{(1)}(\beta,\phi)|\lesssim\beta^{-2\mu}$.
	\end{lem}

\begin{proof}
We rewrite \eqref{Q_k_expression} for $k=0$ as {(see \eqref{z1})}
	\[Q_0(\beta,\phi)=\frac1{2\pi}\int_0^\infty\left(\int_\TT\frac{\gamma(s,\Phi)}{z-\gamma(s,\Phi)}\frac{G_0(s,\Phi)}{s}\,d\Phi\right)\,ds.\]
	
	Recall from $\text{supp }G_0(\cdot, \phi)\subset \left[0,\frac34\right]$ and $\int_\TT\gamma(s,\Phi)G_0(s,\Phi)\,d\Phi=0$ {(as $\hat{G}_{1}=0$)} that
	\begin{equation}\label{Q_0_expression_1}
		Q_0(\beta,\phi)=\frac1{2\pi}\int_0^{\frac34}\left(\int_\TT\gamma(s,\Phi)\left(\frac1{z-\gamma(s,\Phi)}-\frac1z\right)\frac{G_0(s,\Phi)}{s}\,d\Phi\right)\,ds.
	\end{equation}
	For $\beta\geq {\frac78}, s\in\left[0,\frac34\right]$, we have $|z-\gamma(s,\Phi)|\geq |z|-|\gamma(s,\Phi)|\geq \beta^\mu-s^\mu \gtrsim \beta^\mu$. Thus, for $\beta\geq 1$ and $\phi\in\TT$ we have (recalling $\|\beta^{2\mu-1}G_0\|_{L^\infty}\lesssim1$)
	\begin{align*}
		&\int_0^{\frac34}\int_\TT\left|\gamma(s,\Phi)\left(\frac1{z-\gamma(s,\Phi)}-\frac1z\right)\frac{G_0(s,\Phi)}{s}\right|\,d\Phi\,ds\\
		=&\int_0^{\frac34}\int_\TT\left|\frac{\gamma(s,\Phi)^2}{z(z-\gamma(s,\Phi))}\frac{G_0(s,\Phi)}{s}\right|\,d\Phi\,ds\lesssim \int_0^{\frac34}\frac{s^{2\mu}}{\beta^{2\mu}}\frac{s^{1-2\mu}}{s}\,ds\lesssim\beta^{-2\mu}<+\infty,
	\end{align*}
	hence we can drop the bracket outside $\int_\TT$ in \eqref{Q_0_expression_1} to get
	\begin{equation}\label{Q_0_expression_2}
		Q_0(\beta,\phi)=\frac1{2\pi}\int_0^{\frac34}\int_\TT\gamma(s,\Phi)\left(\frac1{z-\gamma(s,\Phi)}-\frac1z\right)\frac{G_0(s,\Phi)}{s}\,d\Phi\,ds.
	\end{equation}
	
	Using $\frac{\beta}{\beta+\ii\mu}\p_\beta z=-\ii z$, we can take the derivative under the integral to obtain
	\begin{align*}
		Q_0^{(1)}(\beta,\phi)&=\frac{\beta\p_\beta Q_0(\beta,\phi)}{\beta+\ii\mu}=\frac{\ii z}{2\pi}\int_{0}^{\frac34}\int_\TT\gamma(s,\Phi)\left(\frac{1}{\left(z-\gamma(s,\Phi)\right)^2}-\frac1{z^2}\right)\frac{G_0(s,\Phi)}{s}\,d\Phi\,ds\\
		&=\frac{\ii z}{2\pi}\int_{0}^{\frac34}\int_\TT\frac{\left(\gamma(s,\Phi)\right)^2\left(2z-\gamma(s,\Phi)\right)}{z^2\left(z-\gamma(s,\Phi)\right)^2}\frac{G_0(s,\Phi)}{s}\,d\Phi\,ds.
	\end{align*}
	Indeed, this follows from
	\begin{align*}
		&\int_{0}^{\frac34}\int_\TT\sup_{\beta>\frac78}\left|\gamma(s,\Phi)\left(\frac{1}{\left(z-\gamma(s,\Phi)\right)^2}-\frac1{z^2}\right)\frac{G_0(s,\Phi)}{s}\right|\,d\Phi\,ds\\
		=&\int_{0}^{\frac34}\int_\TT\sup_{\beta>\frac78}\left|\frac{\left(\gamma(s,\Phi)\right)^2\left(2z-\gamma(s,\Phi)\right)}{z^2\left(z-\gamma(s,\Phi)\right)^2}\frac{G_0(s,\Phi)}{s}\right|\,d\Phi\,ds\lesssim \int_{0}^{\frac34}\sup_{\beta>\frac78}\frac{s^{2\mu}\beta^\mu}{\beta^{2\mu}\beta^{2\mu}}\frac{s^{1-2\mu}}{s}\,ds<+\infty,
	\end{align*}
{and the dominated convergence theorem.}	Hence,
	\[\left|Q_0^{(1)}(\beta, \phi)\right|\lesssim\beta^\mu \int_0^{\frac34}\frac{s^{2\mu}\ \beta^\mu}{\beta^{2\mu}\ \beta^{2\mu}}\frac{s^{1-2\mu}}{s}\,ds\lesssim \beta^{-2\mu}.\]
	\if Furthermore, by $\frac{\beta}{\beta+\ii\mu}\p_\beta z=-\ii z$, we have
	\begin{align*}
		K_0(\beta,\phi,\Phi)&=\frac{\beta}{\beta+\ii\mu}\p_\beta\tilde K_0(\beta,\phi,\Phi)=\ii z\int_{0}^{\frac34}\gamma(s,\Phi)\left(\frac{1}{\left(z-\gamma(s,\Phi)\right)^2}-\frac1{z^2}\right)\frac{G_0(s,\Phi)}{s}\,ds\\
		&=\ii z\int_{0}^{\frac34}\frac{\left(\gamma(s,\Phi)\right)^2\left(2z-\gamma(s,\Phi)\right)}{z^2\left(z-\gamma(s,\Phi)\right)^2}\frac{G_0(s,\Phi)}{s}\,ds
	\end{align*}
	with the bound
	\[\left|K_0(\beta,\phi,\Phi)\right|\lesssim\beta^\mu \int_0^{\frac34}\frac{s^{2\mu}\ \beta^\mu}{\beta^{2\mu}\ \beta^{2\mu}}\frac{s^{1-2\mu}}{s}\,ds\lesssim \beta^{-2\mu}.\]\fi
	Recalling $\p_\beta z=\frac{\mu-\ii\beta}{\beta}z$, we get
	\begin{align*}
		&\p_\beta Q_0^{(1)}(\beta,\phi)-\frac{\mu-\ii\beta}{\beta}Q_0^{(1)}(\beta,\phi)\\
		&\qquad=-2\ii \frac{\mu-\ii\beta}{2\pi\beta}z^2\int_{0}^{\frac34}\int_\TT\gamma(s,\Phi)\left(\frac{1}{\left(z-\gamma(s,\Phi)\right)^3}-
		\frac1{z^3}\right)\frac{G_0(s,\Phi)}{s}\,d\Phi\,ds\\
		&\qquad=-2\ii \frac{\mu-\ii\beta}{2\pi\beta}z^2
		\int_{0}^{\frac34}\int_\TT\frac{\left(\gamma(s,\Phi)\right)^2\left(3z^2-3z\gamma(s,\Phi)+
			\left(\gamma(s,\Phi)\right)^2\right)}{z^3\left(z-\gamma(s,\Phi)\right)^3}\frac{G_0(s,\Phi)}{s}\,d\Phi\,ds.
	\end{align*}
	Here we can take the derivative under the integral based on the similar reason as above, and we have
	\[\left|\p_\beta Q_0^{(1)}(\beta,\phi)\right|\lesssim\beta^{-2\mu} +\beta^{2\mu}\int_0^{\frac34} \frac{s^{2\mu}\ \beta^{2\mu}}{\beta^{3\mu}\ \beta^{3\mu}}\frac{s^{1-2\mu}}{s}\,ds\lesssim\beta^{-2\mu}.\]
	
	This completes the proof.
\end{proof}

\begin{lem}\label{lem5}
Assume that  $\beta\in(0,2)$, $a\in(\beta/3,2\beta/3)$, $b\in (3\beta/2,5\beta/2)$.\footnote{It is equivalent to $0<a<b<5$ 
and $  \beta\in(3a/2,3a)\cap(2b/5,2b/3)\cap(0,2)=:I_{a,b}$. Then $a<\beta<b$.\label{I_a,b}}
Define
	\begin{align*}
			&\widetilde{K}_{0,a,b}(\beta,\phi,\Phi):=
\int_{0}^{a}\left(\frac{\gamma(s,\Phi)}{z-\gamma(s,\Phi)}-\frac{\gamma(s,\Phi)}{z}\right)\frac{G_0(s,\Phi)}{s}\,ds\\
&+\int_{a}^{b}\frac{\gamma(s,\Phi)}{z-\gamma(s,\Phi)}\frac{G_0(s,\Phi)}{s}\,ds+\int_{b}^{5}\left(\frac{\gamma(s,\Phi)}{z-\gamma(s,\Phi)}+\frac{z}{\gamma(s,\Phi)}\right)\frac{G_0(s,\Phi)}{s}\,ds,\quad\forall\ \phi\neq\Phi,\\
&K_{0,a,b}(\beta,\phi,\Phi):=\frac{\beta}{\beta+\ii\mu}\p_\beta\tilde K_{0,a,b}(\beta,\phi,\Phi),
\quad K_{0}(\beta,\phi,\Phi):=K_{0,\beta/2,2\beta}(\beta,\phi,\Phi),\quad \forall\ \phi\neq\Phi.
		\end{align*}
{Here $z:=\beta^\mu e^{-\ii(\beta+\phi)},$ $\gamma(s,\Phi):=s^\mu e^{-\ii(s+\Phi)}$ (i.e. \eqref{z1})}. Then it holds that
\begin{enumerate}[(i)]
	\item  $Q_0(\beta, \phi)=\frac1{2\pi}\int_\TT \widetilde{K}_{0,a,b}(\beta, \phi, \Phi)\,d\Phi $.
	\item $|K_{0,a,b}(\beta,\phi,\Phi)|\lesssim\|\phi-\Phi\|_{\TT}^{\alpha-1}\beta^{1-2\mu} $ for  $\phi\neq\Phi$.
	\item  $Q_0^{(1)}(\beta, \phi)=\frac1{2\pi}\int_\TT K_0(\beta, \phi, \Phi)\,d\Phi $, $|K_{0}(\beta,\phi,\Phi)|\lesssim\|\phi-\Phi\|_{\TT}^{\alpha-1}\beta^{1-2\mu} $ for $\phi\neq\Phi$.
	\item $|\p_\beta K_0(\beta, \phi, \Phi)|\lesssim\|\phi-\Phi\|_{\TT}^{\alpha-2}\beta^{-2\mu} $ for $\phi\neq\Phi$.
\end{enumerate}
	\end{lem}
\begin{proof}
	\begin{enumerate}[(i)]
		\item 
Recall that (using \eqref{Q_k_expression} for $k=0$, \eqref{z1} and $\text{supp }G_0(\cdot, \phi)\subset \left[0,\frac34\right]$)
		\[Q_0(\beta,\phi)=\frac1{2\pi}\int_0^5\left(\int_\TT\frac{\gamma(s,\Phi)}{z-\gamma(s,\Phi)}\frac{G_0(s,\Phi)}{s}\,d\Phi\right)\,ds.\]
	Since $\int_\TT \gamma(s,\Phi)G_0(s,\Phi)\,d\Phi=0=\int_\TT \gamma(s,\Phi)^{-1}G_0(s,\Phi)\,d\Phi$ (as $\hat{G}_{1}=\hat{G}_{-1}=0$), we have
	\begin{align*}
		Q_0(\beta,\phi)&=\frac1{2\pi}\int_0^a\left(\int_\TT\left(\frac{\gamma(s,\Phi)}{z-\gamma(s,\Phi)}-\frac{\gamma(s,\Phi)}{z}\right)\frac{G_0(s,\Phi)}{s}\,d\Phi\right)\,ds\\
		&\qquad+\frac1{2\pi}\int_a^b\left(\int_\TT\frac{\gamma(s,\Phi)}{z-\gamma(s,\Phi)}\frac{G_0(s,\Phi)}{s}\,d\Phi\right)\,ds\\
		&\qquad+\frac1{2\pi}\int_b^5\left(\int_\TT\left(\frac{\gamma(s,\Phi)}{z-\gamma(s,\Phi)}+\frac z{\gamma(s,\Phi)}\right)\frac{G_0(s,\Phi)}{s}\,d\Phi\right)\,ds.
	\end{align*}
	For $s\in [0,a]$, we have $|z-\gamma(s,\Phi)|\geq \beta^\mu-s^\mu\geq \beta^\mu-a^\mu>0$, hence by $\|\beta^{2\mu-1}G_0\|_{L^\infty}\lesssim1$,
	\begin{align*}
		&\int_0^a\int_\TT\left|\left(\frac{\gamma(s,\Phi)}{z-\gamma(s,\Phi)}-\frac{\gamma(s,\Phi)}{z}\right)\frac{G_0(s,\Phi)}{s}\right|\,d\Phi\,ds\\
		&=\int_0^a\int_\TT\left|\frac{\gamma(s,\Phi)^2}{z(z-\gamma(s,\Phi))}\frac{G_0(s,\Phi)}{s}\right|\,d\Phi\,ds\lesssim \int_0^a\int_\TT \frac{s^{2\mu}}{\beta^\mu(\beta^\mu-a^\mu)}\frac{s^{1-2\mu}}{s}\,d\Phi\,ds<+\infty,
	\end{align*}
	thus by Fubini's theorem,
	\begin{align*}
		&\int_0^a\left(\int_\TT\left(\frac{\gamma(s,\Phi)}{z-\gamma(s,\Phi)}-\frac{\gamma(s,\Phi)}{z}\right)\frac{G_0(s,\Phi)}{s}\,d\Phi\right)\,ds\\
		&=\int_\TT\left(\int_0^a\left(\frac{\gamma(s,\Phi)}{z-\gamma(s,\Phi)}-\frac{\gamma(s,\Phi)}{z}\right)\frac{G_0(s,\Phi)}{s}\,ds\right)\,d\Phi.
	\end{align*}
	For $s\in[a,b]\subset[\beta/3,5\beta/2]$, we have
	\begin{align*}
		\left|z-\gamma(s,\Phi)\right|^2&=\left|\beta^\mu e^{-\ii(\beta+\phi)}-s^\mu e^{-\ii(s+\Phi)}\right|^2=\beta^{2\mu}\left|1-\left(\frac{s}{\beta}\right)^\mu e^{-\ii(s+\Phi-\beta-\phi)}\right|^2\\
		&=\beta^{2\mu}\left(\left|1-\left(\frac s\beta\right)^\mu\right|^2+\left(\frac s\beta\right)^\mu\left(1-\cos(s-\beta+\Phi-\phi)\right)\right)\\
		&\sim \beta^{2\mu}\left(\left|\frac{\beta-s}{\beta}\right|^2+\|s-\beta+\Phi-\phi\|_\TT^2\right);
	\end{align*}then using $|\beta-s|\leq 3\beta/2<3<\pi$, $|\beta-s|=\|\beta-s\|_\TT$ and
	\[\frac{2\|\beta-s\|_\TT}{\beta}+\|s-\beta+\Phi-\phi\|_\TT\geq\|\beta-s\|_\TT+\|s-\beta+\Phi-\phi\|_\TT\geq \|\phi-\Phi\|_\TT,\]
	we obtain
	\begin{equation}\label{B.29}
			\left|z-\gamma(s,\Phi)\right|\sim \beta^{\mu}\left(\frac{|\beta-s|}{\beta}+\|s-\beta+\Phi-\phi\|_\TT\right)
			\sim \beta^{\mu}\left(\frac{|\beta-s|}{\beta}+\|\phi-\Phi\|_\TT\right),
	\end{equation}
	As a consequence, we obtain
	\begin{align*}
		&\int_a^b\int_\TT\left|\frac{\gamma(s,\Phi)}{z-\gamma(s,\Phi)}\frac{G_0(s,\Phi)}{s}\right|\,d\Phi\,ds\lesssim\int_a^b\int_\TT \frac{(s/\beta)^\mu}{|s-\beta|/\beta+\|\phi-\Phi\|_\TT}\frac{s^{1-2\mu}}{s}\,d\Phi\,ds\\
		\lesssim& \int_a^b\int_\TT \frac{(s\beta)^{-\mu}}{|s-\beta|/\beta+\|\phi-\Phi\|_\TT}\,d\Phi\,ds\lesssim\int_a^b(s\beta)^{-\mu}\left(\ln\frac{\beta}{|s-\beta|}+2\right)\,ds<+\infty,
	\end{align*}
	thus by Fubini's theorem,
	\begin{align*}
		\int_a^b\left(\int_\TT\frac{\gamma(s,\Phi)}{z-\gamma(s,\Phi)}\frac{G_0(s,\Phi)}{s}\,d\Phi\right)\,ds
		&=\int_\TT\left(\int_a^b\frac{\gamma(s,\Phi)}{z-\gamma(s,\Phi)}\frac{G_0(s,\Phi)}{s}\,ds\right)\,d\Phi.
	\end{align*}
	Finally, for $s\in[b,5]$ we have $|z-\gamma(s,\Phi)|\geq s^\mu-\beta^\mu\geq b^\mu-\beta^\mu>0$, then
	\begin{align*}
		&\int_b^5\int_\TT\left|\left(\frac{\gamma(s,\Phi)}{z-\gamma(s,\Phi)}+\frac z{\gamma(s,\Phi)}\right)\frac{G_0(s,\Phi)}{s}\right|\,d\Phi\,ds\\
		\lesssim&\int_b^5\int_\TT\left(\frac{s^\mu}{b^\mu-\beta^\mu}+\frac{\beta^\mu}{s^\mu}\right)\frac{s^{1-2\mu}}{s}\,d\Phi\,ds<+\infty,
	\end{align*}
	hence by Fubini's theorem,
	\begin{align*}
		&\int_b^5\left(\int_\TT\left(\frac{\gamma(s,\Phi)}{z-\gamma(s,\Phi)}+\frac z{\gamma(s,\Phi)}\right)\frac{G_0(s,\Phi)}{s}\,d\Phi\right)\,ds\\
		=&\int_\TT\left(\int_b^5\left(\frac{\gamma(s,\Phi)}{z-\gamma(s,\Phi)}+\frac z{\gamma(s,\Phi)}\right)\frac{G_0(s,\Phi)}{s}\,ds\right)\,d\Phi.
	\end{align*}
	Therefore, by the definition of $\widetilde{K}_{0,a,b}$, we have $Q_0(\beta, \phi)=\frac1{2\pi}\int_\TT \widetilde{K}_{0,a,b}(\beta, \phi, \Phi)\,d\Phi $.\smallskip
	
	\item We first claim that for $\beta\in(0,2)$, $\phi\neq\Phi\in\TT$, $a\in(\beta/3,2\beta/3)$ and $b\in(3\beta/2,5\beta/2)$, we have (recalling $\frac{\beta}{\beta+\ii\mu}\p_\beta z=-\ii z$) (i.e., we can take the derivative under the integral)
	\begin{align}
		&K_{0,a,b}(\beta,\phi,\Phi)=\frac{\beta\p_\beta\wt K_{0,a,b}(\beta,\phi,\Phi)}{\beta+\ii\mu}=\ii z\int_{0}^{a}\left(\frac{\gamma(s,\Phi)}{(z-\gamma(s,\Phi))^2}-\frac{\gamma(s,\Phi)}{z^2}\right)\frac{G_0(s,\Phi)}{s}\,ds\label{K_0,a,b}\\
		&+\ii z\int_{a}^{b}\frac{\gamma(s,\Phi)}{(z-\gamma(s,\Phi))^2}\frac{G_0(s,\Phi)}{s}\,ds+\ii z\int_{b}^{5}\left(\frac{\gamma(s,\Phi)}{(z-\gamma(s,\Phi))^2}-\frac{1}{\gamma(s,\Phi)}\right)\frac{G_0(s,\Phi)}{s}\,ds\nonumber\\
		&=:K_{0,a,b}^{(1)}(\beta,\phi,\Phi)+K_{0,a,b}^{(2)}(\beta,\phi,\Phi)+K_{0,a,b}^{(3)}(\beta,\phi,\Phi).\nonumber
	\end{align}
	Indeed, we only need to show that (``A" stands for ``absolute value") (see footnote \ref{I_a,b})
	\begin{align}
		A_{0,1}:&=\int_{0}^{a}{\sup_{\beta\in I_{a,b}}}\left|\left(\frac{\gamma(s,\Phi)}{(z-\gamma(s,\Phi))^2}-\frac{\gamma(s,\Phi)}{z^2}\right)\frac{G_0(s,\Phi)}{s}\right|\,ds<{+\infty},\label{A_0,1}\\
		A_{0,2}:&=\int_{a}^{b}\sup_{\beta\in I_{a,b}}\left|\frac{\gamma(s,\Phi)}{(z-\gamma(s,\Phi))^2}\frac{G_0(s,\Phi)}{s}\right|\,ds<+\infty,\label{A_0,2}\\
		A_{0,3}:&=\int_{b}^{5}\sup_{\beta\in I_{a,b}}\left|\left(\frac{\gamma(s,\Phi)}{(z-\gamma(s,\Phi))^2}-\frac{1}{\gamma(s,\Phi)}\right)\frac{G_0(s,\Phi)}{s}\right|\,ds<+\infty.\label{A_0,3}
	\end{align}
	for any fixed $0<a<b<5$ {such  that $I_{a,b}\neq \emptyset$} and $\phi\neq\Phi\in\TT$. For $A_{0,1}$, we have $|z-\gamma(s,\Phi)|\geq \beta^\mu-s^\mu\geq \beta^\mu-(2\beta/3)^\mu\gtrsim\beta^\mu$ if $s\in[0,a]$, then
	\begin{align*}
		A_{0,1}&=\int_{0}^{a}\sup_{\beta\in I_{a,b}}\left|\frac{\left(\gamma(s,\Phi)\right)^2\left(2z-\gamma(s,\Phi)\right)}{z^2\left(z-\gamma(s,\Phi)\right)^2}\frac{G_0(s,\Phi)}{s}\right|\,ds
		\lesssim \int_0^{a}\sup_{\beta\in I_{a,b}}\frac{s^{2\mu}\ \beta^\mu}{\beta^{2\mu}\ \beta^{2\mu}}\frac{s^{1-2\mu}}{s}\,ds<+\infty.
	\end{align*}
	For $A_{0,2}$,
	if $\beta\in I_{a,b}$ and $s\in[a,b]\subset [\beta/3,5\beta/2]$, by \eqref{B.29}
	we have
	\begin{align*}
		A_{0,2}&\lesssim \int_a^b \sup_{\beta\in I_{a,b}}\frac{s^{\mu}}{\beta^{2\mu}(|s-\beta|^2/\beta^2+\|\phi-\Phi\|_\TT^2)}\frac{s^{1-2\mu}}{s}\,ds
		\lesssim \int_a^b \sup_{\beta\in I_{a,b}}\frac{s^{-\mu}}{\beta^{2\mu}\|\phi-\Phi\|_\TT^2}\,ds<+\infty.
	\end{align*}
	For $A_{0,3}$, we have $|z-\gamma(s,\Phi)|\geq s^\mu-\beta^\mu\geq s^\mu-(2s/3)^\mu\gtrsim s^\mu$ if $s\in[b,5]\subset[3\beta/2,5]$, and then (recalling $\mu>1/2$)
	\begin{align*}
		A_{0,3}&=\int_{b}^{5}\sup_{\beta\in I_{a,b}}\left|\frac{z\left(2\gamma(s,\Phi)-z\right)}{\left(z-\gamma(s,\Phi)\right)^2\gamma(s,\Phi)}\frac{G_0(s,\Phi)}{s}\right|\,ds\lesssim \int_b^5\sup_{\beta\in I_{a,b}}\frac{\beta^\mu\ s^\mu}{s^{2\mu}\ s^\mu}\frac{s^{1-2\mu}}{s}\,ds<+\infty.
	\end{align*}
	Therefore, we have checked \eqref{A_0,1}, \eqref{A_0,2} and \eqref{A_0,3}, hence \eqref{K_0,a,b} holds.
	
	Now we prove the bound $\left|K_{0,a,b}(\beta,\phi,\Phi)\right|\lesssim \|\phi-\Phi\|_\TT^{\alpha-1}\beta^{1-2\mu}$. By the proof of \eqref{A_0,1}, \eqref{A_0,3} and $|z|=\beta^{\mu}$ in \eqref{z1}, we have
	\begin{align*}
		\left|K_{0,a,b}^{(1)}(\beta,\phi,\Phi)\right|&\lesssim \beta^{\mu}\int_0^{\frac{2\beta}3}\frac{s^{2\mu}\ \beta^\mu}{\beta^{2\mu}\ \beta^{2\mu}}\frac{s^{1-2\mu}}{s}\,ds\lesssim \beta^{1-2\mu},\\
		\left|K_{0,a,b}^{(3)}(\beta,\phi,\Phi)\right|&\lesssim \beta^{\mu}\int_{\frac{3\beta}{2}}^\infty \frac{\beta^\mu\ s^\mu}{s^{2\mu}\ s^\mu}\frac{s^{1-2\mu}}{s}\,ds\lesssim\beta^{1-2\mu}.
	\end{align*}
	As for $K_{0,a,b}^{(2)}$, we need to invoke the $C^\alpha$ regularity. Define
	$$\tilde G_0(s,\Phi)=\frac{\gamma(s,\Phi)}{\p_s\gamma(s,\Phi)}\frac{G_0(s,\Phi)}s=\frac{G_0(s,\Phi)}{\mu-\ii s}=\frac{\eta(s)}{s^{2\mu-1}(\mu-\ii s)}\cdot s^{2\mu-1}G(s, \Phi),$$
	then
	\begin{align}\label{K_0^2}
		&K_{0,a,b}^{(2)}(\beta,\phi,\Phi)=\ii z\int_{a}^{b}\frac{\p_s\gamma(s,\Phi)}{\left(z-\gamma(s,\Phi)\right)^2}\tilde G_0(s, \Phi)\,ds\\
		\notag			=&\ii z\int_{a}^{b}\frac{\p_s\gamma(s,\Phi)}{\left(z-\gamma(s,\Phi)\right)^2}\left(\tilde G_0(s,\Phi)-\tilde G_0(\beta,\Phi)\right)\,ds+\ii z\ \tilde G_0(\beta, \Phi)\int_{a}^{b}\frac{\p_s\gamma(s,\Phi)}{\left(z-\gamma(s,\Phi)\right)^2}ds.
	\end{align}
	By \eqref{B.29}, for $a\in(\beta/3,2\beta/3)$ and $b\in(3\beta/2,5\beta/2)$ we have
	\begin{equation}\label{5.23}
		\begin{aligned}
			\left|z-\gamma\left(a,\Phi\right)\right|\sim \beta^\mu\left(1+\|\phi-\Phi\|_\TT\right)\sim \beta^\mu,\quad \left|z-\gamma\left(b,\Phi\right)\right|\sim \beta^\mu\left(1+\|\phi-\Phi\|_\TT\right)\sim \beta^\mu.
		\end{aligned}
	\end{equation}
	Also, for $s\in[a,b]\subset \left[\beta/3, {5\beta/2}\right]$ and $0<\beta<2$ we have
	\begin{equation}\label{5.21}
		|\p_s\gamma(s,\Phi)|\lesssim s^{\mu-1}\lesssim \beta^{\mu-1}.
	\end{equation}
	It remains to estimate $\left|\tilde G_0(s,\Phi)-\tilde G_0(\beta,\Phi)\right|$ and $\left|\tilde G_0(\beta,\Phi)\right|$ for $s\in\left[\beta/3, {5\beta/2}\right]$ and $0<\beta<2$. Let $\rho\in C^\infty(\RR; [0,1])$ be a smooth bump function such that $\rho|_{(-\infty,1/2)}\equiv 0$, $\rho|_{(1,\infty)}\equiv1$. For each $r>0$, we define
	\[\tilde G_{0,r}(s, \Phi)=\rho\left(\frac sr\right)\tilde G_0(s, \Phi)=\rho\left(\frac sr\right)\frac{\eta(s)}{s^{2\mu-1}(\mu-\ii s)}\cdot s^{2\mu-1}G(s, \Phi).\]
	Using \eqref{f1} and $\text{supp }\eta\subset[-3/4,3/4]$, one can easily show that $\left\|\rho\left(\frac \beta r\right)\frac{\eta(\beta)}{\beta^{2\mu-1}(\mu-\ii \beta)}\right\|_{C_\beta^\alpha}\lesssim r^{1-2\mu}$ for $r\leq 1$. It follows from the algebra property (Lemma \ref{AppendixA_Algebra}) and $\|\beta^{2\mu-1}G\|_{C_\beta^\alpha}=1$ that
	\begin{equation*}
		\left\|\tilde G_{0,r}\right\|_{C_\beta^\alpha}\lesssim r^{1-2\mu},\qquad r\leq1.
	\end{equation*}
	Hence, for $s\in\left[\beta/3, {5\beta/2}\right]$ and $0<\beta<2$, taking $r=\beta/3$ we have (recall Lemma \ref{equvi_Holder_norm})
	\begin{equation}\label{tilde_G_0_est}
		\left|\tilde G_0(s,\Phi)-\tilde G_0(\beta,\Phi)\right|=\left|\tilde G_{0,r}(s,\Phi)-\tilde G_{0,r}(\beta,\Phi)\right|\lesssim r^{1-2\mu}\frac{|s-\beta|^\alpha}{|s+\beta|^\alpha}\lesssim \frac{|s-\beta|^\alpha}{\beta^{\alpha+2\mu-1}},
	\end{equation}
	and
	\begin{equation}\label{tilde_G_0_est1}
		\left|\tilde G_0(\beta,\Phi)\right|=\left|\tilde G_{0,r}(\beta,\Phi)\right|\lesssim r^{1-2\mu}\lesssim\beta^{1-2\mu}.
	\end{equation}
	Substituting \eqref{B.29}, \eqref{5.23}$\sim$\eqref{tilde_G_0_est1} into \eqref{K_0^2} yields
	\begin{align*}
		\left|K_{0,a,b}^{(2)}(\beta,\phi,\Phi)\right|&\lesssim \beta^\mu\int_{a}^{b}\frac{\beta^{\mu-1}}{\beta^{2\mu}\left(|s-\beta|^2/\beta^2+\|\phi-\Phi\|_\TT^2\right)}
		\frac{|s-\beta|^\alpha}{\beta^{\alpha+2\mu-1}}\,ds\\
		&\qquad\qquad+\beta^\mu\beta^{1-2\mu}\left(\frac1{\left|z-\gamma\left(a,\Phi\right)\right|}+\frac1{\left|z-\gamma\left(b,\Phi\right)\right|}\right)\\
		&\lesssim \|\phi-\Phi\|_\TT^{\alpha-1}\beta^{1-2\mu},
	\end{align*}
	where we have used the fact that $\int_{\RR} \frac{|x|^\alpha}{x^2+A^2}\,dx\lesssim_\alpha A^{\alpha-1}$ for $\alpha\in(0,1)$ and $A>0$ {(and take $A=\beta\|\phi-\Phi\|_\TT$)}.
	Therefore, we arrive at
	$\left|K_{0,a,b}(\beta,\phi,\Phi)\right|\lesssim \|\phi-\Phi\|_\TT^{\alpha-1}\beta^{1-2\mu}$, where the implicit constant in $\lesssim$ depends only on $\alpha,\mu$ and it is independent of $\beta\in(0,2), \phi\neq\Phi\in\TT$, $a\in(\beta/3,2\beta/3)$ and $b\in(3\beta/2,5\beta/2)$.
	
	\item By (i), we have $Q_0(\beta, \phi)=\frac1{2\pi}\int_\TT \widetilde{K}_{0,a,b}(\beta, \phi, \Phi)\,d\Phi $; by the definition, we have
	\[Q_0^{(1)}=\frac{\beta}{\beta+\ii\mu}\p_\beta Q_0,\qquad K_{0,a,b}=\frac{\beta}{\beta+\ii\mu}\p_\beta\wt K_{0,a,b};\]
	by (ii), we have (note that $\alpha\in(0,1)$, and $ \beta\sim a\sim b$ for ${\beta}\in I_{a,b} $, see footnote \ref{I_a,b})
	 $$\int_\TT\sup_{{\beta}\in I_{a,b}}|K_{0,a,b}({\beta},\phi,\Phi)|\,d\Phi\lesssim\int_\TT\sup_{{\beta}\in I_{a,b}}\|\phi-\Phi\|_{\TT}^{\alpha-1}\beta^{1-2\mu}\,d\Phi\lesssim a^{1-2\mu}$$
for $0<a<b<5$ such  that $I_{a,b}\neq \emptyset$.
	Hence, by the dominated convergence theorem,
	$$Q_0^{(1)}(\beta, \phi)=\frac1{2\pi}\int_\TT K_{0,a,b}(\beta, \phi, \Phi)\,d\Phi $$ for  $0<a<b<5$, ${\beta}\in I_{a,b}$. Taking $a=\beta/2, b=2\beta$ gives
	\[Q_0^{(1)}(\beta, \phi)=\frac1{2\pi}\int_\TT K_{0}(\beta, \phi, \Phi)\,d\Phi, \qquad \left|K_{0}(\beta,\phi,\Phi)\right| \lesssim\|\phi-\Phi\|_{\TT}^{\alpha-1}\beta^{1-2\mu}.\]
	
	\item Finally, we prove the bound for $\p_\beta K_0$. Using the fact that $\p_\beta z=\frac{\mu-\ii \beta}{\beta}z$, a direct computation yields {(using \eqref{K_0^2} for $a=\beta/2, b=2\beta$)}
	\begin{align*}
		\p_\beta K_0(\beta, \phi,\Phi)&=\frac{\mu-\ii \beta}{\beta}K_0(\beta,\phi,\Phi)-K_0^{(4)}(\beta,\phi,\Phi)-K_0^{(5)}(\beta,\phi,\Phi)-K_0^{(6)}(\beta,\phi,\Phi)\\
		&\qquad+\frac\ii\beta\left(\frac{zG_0(2\beta,\Phi)}{\gamma(2\beta,\Phi)}-\frac{\gamma\left(\beta/2,\Phi\right)
			G_0\left({\beta/2},\Phi\right)}{z}\right),
	\end{align*}
	where
	\begin{align*}
		K_0^{(4)}(\beta,\phi,\Phi)&=2\ii \frac{\mu-\ii \beta}{\beta}z^2\int_0^{\frac\beta2}\gamma(s,\Phi)\left(\frac{1}{\left(z-\gamma(s,\Phi)\right)^3}-\frac1{z^3}\right)\frac{G_0(s,\Phi)}{s}\,ds\\
		&=2\ii \frac{\mu-\ii \beta}{\beta}z^2\int_0^{\frac\beta2}\frac{\left(\gamma(s,\Phi)\right)^2\left(3z^2-3z\gamma(s,\Phi)+
			\left(\gamma(s,\Phi)\right)^2\right)}{z^3\left(z-\gamma(s,\Phi)\right)^3}\frac{G_0(s,\Phi)}{s}\,ds,
	\end{align*}
	\begin{align*}
		K_0^{(5)}(\beta,\phi,\Phi)&=2\ii \frac{\mu-\ii \beta}{\beta}z^2\int_{\frac\beta2}^{2\beta}\frac{\gamma(s,\Phi)}{\left(z-\gamma(s,\Phi)\right)^3}\frac{G_0(s,\Phi)}{s}\,ds\\
		&=2\ii \frac{\mu-\ii \beta}{\beta}z^2\int_{\frac\beta2}^{2\beta}\frac{\p_s\gamma(s,\Phi)}{\left(z-\gamma(s,\Phi)\right)^3}\left(\tilde G_0(s,\Phi)-\tilde G_0(\beta,\Phi)\right)\,ds\\
		&\qquad\qquad\qquad+2\ii \frac{\mu-\ii \beta}{\beta}z^2\ \tilde G_0(\beta, \Phi)\int_{\frac\beta2}^{2\beta}\frac{\p_s\gamma(s,\Phi)}{\left(z-\gamma(s,\Phi)\right)^3}ds,\\
		K_0^{(6)}(\beta,\phi,\Phi)&=2\ii\frac{\mu-\ii \beta}{\beta}z^2\int_{2\beta}^{5}\frac{\gamma(s,\Phi)}{\left(z-\gamma(s,\Phi)\right)^3}\frac{G_0(s,\Phi)}{s}\,ds.
	\end{align*}
	Here we take the derivative under the integral, which can be checked using an argument similar with (ii).
	Since $\left|\frac{\mu-\ii \beta}{\beta}\right|\lesssim\beta^{-1}$, we have
	\begin{align*}
		&\left|K_0^{(4)}(\beta,\phi,\Phi)\right|\lesssim \beta^{-1}\beta^{2\mu} \int_0^{\frac\beta2}\frac{s^{2\mu}\ \beta^{2\mu}}{\beta^{3\mu}\ \beta^{3\mu}}\frac{s^{1-2\mu}}{s}\,ds\lesssim \beta^{-2\mu},\\
		\left|K_0^{(5)}(\beta,\phi,\Phi)\right|&\lesssim \beta^{-1}\beta^{2\mu}\int_{\frac\beta2}^{2\beta}\frac{\beta^{\mu-1}}{\beta^{3\mu}\left(|s-\beta|^3/\beta^3+\|\phi-\Phi\|_\TT^3\right)}
		\frac{|s-\beta|^\alpha}{\beta^{\alpha+2\mu-1}}\,ds\\
		&+\beta^{-1}\beta^{2\mu}\beta^{1-2\mu}\left(\frac1{\left|z-\gamma\left({\beta/2},\Phi\right)\right|^2}+
		\frac1{\left|z-\gamma\left(2\beta,\Phi\right)\right|^2}\right)
		\lesssim \|\phi-\Phi\|_\TT^{\alpha-2}\beta^{-2\mu},\\
		&\left|K_0^{(6)}(\beta,\phi,\Phi)\right|\lesssim \beta^{-1}\beta^{2\mu} \int_{2\beta}^{5}\frac{s^\mu}{s^{3\mu}}\frac{s^{1-2\mu}}{s}\,ds\lesssim \beta^{-2\mu}.
	\end{align*}
	Therefore, we arrive at
	\begin{align*}
		\left|\p_\beta K_0(\beta, \phi,\Phi)\right|&\lesssim \|\phi-\Phi\|_\TT^{\alpha-2}\beta^{-2\mu}+\beta^{-1}\left(\frac{\beta^\mu\ (2\beta)^{1-2\mu}}{(2\beta)^\mu}+\frac{\left({\beta/2}\right)^\mu\ \left({\beta/2}\right)^{1-2\mu}}{\beta^\mu}\right)\\
		&\lesssim\|\phi-\Phi\|_\TT^{\alpha-2}\beta^{-2\mu}+\beta^{-2\mu}\lesssim\|\phi-\Phi\|_\TT^{\alpha-2}\beta^{-2\mu}.
	\end{align*}
\end{enumerate}

This completes the proof.
\end{proof}

Finally, we prove Lemma \ref{lem1}.
\begin{proof}[Proof of Lemma \ref{lem1}]
	\begin{enumerate}[(i)]
		\item If $\beta\geq 1$, then by Lemma \ref{lem4}, we have $|Q_0^{(1)}(\beta,\phi)|\lesssim\beta^{-2\mu}\lesssim\beta^{1-2\mu}/(1+\beta)$. If $\beta\in(0,2)$ , then we get by Lemma \ref{lem5} (iii) that {(as $\alpha\in(0,1)$)}
		\[\left|Q_0^{(1)}(\beta,\phi)\right|\lesssim\int_\TT{\left|K_0(\beta,\phi,\Phi)\right|\,d\Phi}\lesssim \int_\TT \|\phi-\Phi\|_\TT^{\alpha-1}\beta^{1-2\mu}\,d\Phi\lesssim\beta^{1-2\mu}\lesssim\beta^{1-2\mu}/(1+\beta).\]
		
		\item Assume that  $k\in\ZZ_+$. If $0<\beta\leq k/5$, we get by Lemma \ref{lem2} (i) that $$\left|Q_k^{(1)}(\beta,\phi)\right|\lesssim k^{-2\mu-\alpha}\lesssim k^{\alpha-1}k^{1-\alpha-2\mu}k^{-\alpha}\lesssim (|\beta-k|+1)^{\alpha-1}(\beta+k)^{1-\alpha-2\mu}k^{-\alpha},$$ since $\alpha\in(0,1), \mu>1/2$ and $|\beta-k|+1\leq k+1\leq2k$, $\beta+k\leq\frac{6}5k$. If $\beta\geq2k$, we have
		\[\left|Q_k^{(1)}(\beta,\phi)\right|\lesssim \beta^{-2\mu}k^{-\alpha}\lesssim \beta^{\alpha-1}\beta^{1-\alpha-2\mu}k^{-\alpha}\lesssim (|\beta-k|+1)^{\alpha-1}(\beta+k)^{1-\alpha-2\mu}k^{-\alpha},\]
		{by using Lemma \ref{lem2} (ii) and} $\alpha\in(0,1), \mu>1/2$, $|\beta-k|+1\leq \beta$, $\beta+k\leq \frac32\beta$. Now we assume that $k/10<\beta<3k$. If $|\beta-k|\geq 1$, then we get by Lemma \ref{lem3} (i) (iii) that
		\[\left|Q_k^{(1)}(\beta,\phi)\right|\lesssim\int_\TT\left|K_k(\beta,\phi,\Phi)\right|\,d\Phi\lesssim \left|\frac{\beta-k}{k}\right|^{\alpha-1}k^{-2\mu-\alpha}\lesssim(|\beta-k|+1)^{\alpha-1}(\beta+k)^{1-\alpha-2\mu}k^{-\alpha},\]
		since $\alpha\in(0,1), \mu>1/2>0$ and $|\beta-k|+1\leq 2|\beta-k|, \beta+k<4k$. If $|\beta-k|<2$, then we get by Lemma \ref{lem3}  (ii) (iii) that
		\begin{align*}
			\left|Q_k^{(1)}(\beta,\phi)\right|&\lesssim\int_\TT\left|K_k(\beta,\phi,\Phi)\right|\,d\Phi\lesssim\int_\TT \|\phi-\Phi\|_\TT^{\alpha-1}k^{1-2\alpha-2\mu}\,d\Phi\\
			&\lesssim k^{1-2\alpha-2\mu}\lesssim 1\cdot k^{1-\alpha-2\mu}k^{-\alpha}\lesssim(|\beta-k|+1)^{\alpha-1}(\beta+k)^{1-\alpha-2\mu}k^{-\alpha},
		\end{align*}
		since $\alpha\in(0,1), \mu>1/2$ and $|\beta-k|+1<3$, $\beta+k<4k$.
		
		\item Assume that $0<\beta_1/2<\beta_2<\beta_1<2$. For $k=0$, by Lemma \ref{lem5} (iii) (iv), {\eqref{Ik}} and $0<\beta_2<\beta_1<2\beta_2$, we have
		\begin{align*}
			I_0&\leq \int_{\|\phi-\Phi\|_\TT\leq \frac{\beta_1-\beta_2}{\beta_1}}\left[\left|\beta_1^{2\mu-1} K_{0}(\beta_1,\phi,\Phi)\right|+
			\left|\beta_2^{2\mu-1} K_{0}(\beta_2,\phi,\Phi)\right|\right]d\Phi\\
			&\qquad\qquad+\int_{\|\phi-\Phi\|_\TT> \frac{\beta_1-\beta_2}{\beta_1}}\int_{\beta_2}^{\beta_1}\left|\p_\beta\left(\beta^{2\mu-1
			}K_0\right)(\beta,\phi,\Phi)\right|\,d\beta\,d\Phi\\
			&\lesssim\int_{\|\phi-\Phi\|_\TT\leq \frac{\beta_1-\beta_2}{\beta_1}}\|\phi-\Phi\|_{\TT}^{\alpha-1}\,d\Phi\\
			&\qquad\qquad +\int_{\beta_2}^{\beta_1}\int_{\|\phi-\Phi\|_\TT> \frac{\beta_1-\beta_2}{\beta_1}}[\beta^{-1}\|\phi-\Phi\|_\TT^{\alpha-1}+
			\beta^{-1}\|\phi-\Phi\|_\TT^{\alpha-2}]\,d\Phi\,d\beta\\
			&\lesssim  \left|\frac{\beta_1-\beta_2}{\beta_1}\right|^\alpha+\left|\frac{\beta_1-\beta_2}{\beta_2}\right|+
			\left|\frac{\beta_1-\beta_2}{\beta_2}\right|\left|\frac{\beta_1-\beta_2}{\beta_1}\right|^{\alpha-1}\lesssim\left|\frac{\beta_1-\beta_2}{\beta_1}\right|^\alpha
			\lesssim\left|\frac{\beta_1-\beta_2}{\beta_1+\beta_2}\right|^\alpha.
		\end{align*}
		{For $k\in\ZZ_+$,  we consider two cases $\beta_1\leq\frac k5$ and $\beta_1>\frac k5$ respectively.}
If $\beta_1\leq\frac k5$, then $0<\beta_1/2<\beta_2<\beta_1\leq\frac k5$, and  by Lemma \ref{lem2} (i), {\eqref{Ik} and $\mu>1/2$}, we have\begin{align*}
			I_k&\leq\int_{\beta_2}^{\beta_1}\left|\p_\beta\left(\beta^{2\mu-1
			}Q_k^{(1)}\right)(\beta,\phi)\right|\,d\beta\lesssim \int_{\beta_2}^{\beta_1}\left(\beta^{2\mu-2}k^{-2\mu-\alpha}+
			\beta^{2\mu-1}k^{-2\mu-\alpha}\right)\,d\beta\\
			&\lesssim
			\int_{\beta_2}^{\beta_1}\beta^{-1}k^{-2\mu-\alpha}\,d\beta\lesssim\left|\frac{\beta_1-\beta_2}{\beta_1+\beta_2}\right|^\alpha(1+k)^{-2\mu-\alpha}.
		\end{align*}
Here we used $k\sim k+1$ and $\int_{\beta_2}^{\beta_1}\beta^{-1}\,d\beta\leq|\frac{\beta_1-\beta_2}{\beta_2}|
			\leq|\frac{\beta_1-\beta_2}{\beta_2}|^\alpha
			\lesssim|\frac{\beta_1-\beta_2}{\beta_1+\beta_2}|^\alpha $.

If $\beta_1>\frac k5$, then $\beta_2>\beta_1/2>\frac k{10}$,
 $\frac k{10}<\beta_2<\beta_1<2<3k$, $ \beta_2\sim\beta_1\sim 1\sim k\sim k+1$, and by Lemma \ref{lem3} (ii) (iii), {\eqref{Ik} and $\mu>1/2$}, we get
		\begin{align*}
			I_k&\leq \int_{\|\phi-\Phi\|_\TT\leq |\beta_1-\beta_2|}\left[\left|\beta_1^{2\mu-1} K_{k}(\beta_1,\phi,\Phi)\right|+
			\left|\beta_2^{2\mu-1} K_{k}(\beta_2,\phi,\Phi)\right|\right]d\Phi\\
			&\qquad\qquad+\int_{\|\phi-\Phi\|_\TT> |\beta_1-\beta_2|}\int_{\beta_2}^{\beta_1}\left|\p_\beta\left(\beta^{2\mu-1
			}K_k\right)(\beta,\phi,\Phi)\right|\,d\beta\,d\Phi\\
			&\lesssim \int_{\|\phi-\Phi\|_\TT\leq |\beta_1-\beta_2|}\|\phi-\Phi\|_{\TT}^{\alpha-1}\,d\Phi\\
			&\qquad\qquad+\int_{\beta_2}^{\beta_1}\int_{\|\phi-\Phi\|_\TT> |\beta_1-\beta_2|}\left(\beta^{2\mu-2}\|\phi-\Phi\|_\TT^{\alpha-1}+
			\beta^{2\mu-1}\|\phi-\Phi\|_\TT^{\alpha-2}\right)\,d\Phi\,d\beta\\
			&\lesssim|\beta_1-\beta_2|^\alpha+\int_{\beta_2}^{\beta_1}\left(1+|\beta_1-\beta_2|^{\alpha-1}\right)\,d\beta\lesssim|\beta_1-\beta_2|^\alpha
			\lesssim\left|\frac{\beta_1-\beta_2}{\beta_1+\beta_2}\right|^\alpha.
		\end{align*}
		Hence, $I_k{\lesssim\left|\frac{\beta_1-\beta_2}{\beta_1+\beta_2}\right|^\alpha}\lesssim\left|\frac{\beta_1-\beta_2}{\beta_1+\beta_2}\right|^\alpha(1+k)^{-2\mu-\alpha}$.

		\item Assume that $k_0\in \ZZ\cap[2,+\infty)$, $\beta_1,\beta_2\in(k_0-1,k_0+1)\subset(1,+\infty)$ and $0<\beta_1-\beta_2<1$. For $k=0$, we get by Lemma \ref{lem4} { and \eqref{Ik}} that
		\begin{align*}
			I_0&{\leq} \int_{\beta_2}^{\beta_1}\left|\p_\beta\left(\beta^{2\mu-1
			}Q_0^{(1)}\right)(\beta,\phi)\right|\,d\beta\lesssim\int_{\beta_2}^{\beta_1}\left(\beta^{2\mu-2}\beta^{-2\mu}+\beta^{2\mu-1}\beta^{-2\mu}\right)d\beta\\
			&\lesssim \int_{\beta_2}^{\beta_1}\left(\beta^{-2}+\beta^{-1}\right)\,d\beta\lesssim{\left|\frac{\beta_1-\beta_2}{\beta_2}\right|}
			\lesssim\left|\frac{\beta_1-\beta_2}{\beta_1+\beta_2}\right|^\alpha.
		\end{align*}
		For $|k-k_0|\leq 1$, it follows from $\beta_1,\beta_2\in(k_0-1, k_0+1)$ that $\beta\sim k_0, k\sim k_0, |\beta-k|<2,$ $\frac{k}{10}\leq \frac{k_0+1}{10}<k_0-1<\beta<k_0+1\leq k+2\leq 3k$ for $\beta\in[\beta_2,\beta_1]$, then by Lemma \ref{lem3} (ii) (iii)  { and \eqref{Ik}}, we have
		\begin{align*}
			I_k&{\leq}\int_{\|\phi-\Phi\|_\TT\leq k_0|\beta_1-\beta_2|}\left(\left|\beta_1^{2\mu-1} K_{k}(\beta_1,\phi,\Phi)\right|+
			\left|\beta_2^{2\mu-1} K_{k}(\beta_2,\phi,\Phi)\right|\right)d\Phi\\
			&\qquad+\int_{\|\phi-\Phi\|_\TT> k_0|\beta_1-\beta_2|}\int_{\beta_2}^{\beta_1}\left|\p_\beta\left(\beta^{2\mu-1
			}K_{k}\right)(\beta,\phi,\Phi)\right|\,d\beta\,d\Phi\\
			&\lesssim \int_{\|\phi-\Phi\|_\TT\leq k_0|\beta_1-\beta_2|}k_0^{2\mu-1}\|\phi-\Phi\|_\TT^{\alpha-1}k_0^{1-2\alpha-2\mu}\,d\Phi+
			\int_{\beta_2}^{\beta_1}\int_{\|\phi-\Phi\|_\TT> k_0|\beta_1-\beta_2|}\\
			&\qquad\Big(k_0^{2\mu-2}\|\phi-\Phi\|_\TT^{\alpha-1}k_0^{1-2\alpha-2\mu}+k_0^{2\mu-1}\|\phi-\Phi\|_\TT^{\alpha-2}k_0^{2-2\alpha-2\mu}\Big)\,d\Phi\,d\beta\\
			&\lesssim k_0^{-2\alpha}\left(k_0|\beta_1-\beta_2|\right)^\alpha
			+\int_{\beta_2}^{\beta_1}\Big(k_0^{-1-2\alpha}+k_0^{1-2\alpha}\left(k_0|\beta_1-\beta_2|\right)^{\alpha-1}\Big)d\beta\\
			&\lesssim k_0^{-\alpha}|\beta_1-\beta_2|^\alpha\lesssim\left|\frac{\beta_1-\beta_2}{\beta_1+\beta_2}\right|^\alpha.
		\end{align*}
		For $k\geq 1$ and $|k-k_0|\geq 2$, we have $\beta\sim k_0$, $|\beta\pm k|\sim |k_0\pm k|$, $|\beta-k|>1$ for $\beta\in[\beta_2,\beta_1]$ and we consider
		two cases $|k-k_0|\leq |\beta_1-\beta_2|k_0/2$ and $|k-k_0|{\geq}|\beta_1-\beta_2|k_0/2$ respectively. If $2\leq |k-k_0|\leq|\beta_1-\beta_2|k_0/2\leq k_0/2$, then by Lemma \ref{lem1} (ii){ and \eqref{Ik}}, we have
		\begin{align*}
			I_k&{\leq}\left|\beta_1^{2\mu-1} Q_{k}^{(1)}(\beta_1,\phi)\right|+
			\left|\beta_2^{2\mu-1} Q_{k}^{(1)}(\beta_2,\phi)\right|\\
			&\lesssim k_0^{2\mu-1}|k_0-k|^{\alpha-1}(k_0+k)^{1-\alpha-2\mu}k^{-\alpha}
			\lesssim k_0^{-2\alpha}|k-k_0|^{\alpha-1}.
		\end{align*}
		Now we claim that
		\begin{equation}\label{p_beta_phi_K_k_estimate}
			\left|\p_\beta Q_k^{(1)}(\beta, \phi)\right|\lesssim
			\left|\frac{\beta-k}{\beta+k}\right|^{\alpha-2}(\beta+k)^{-2\mu}k^{-\alpha},\quad \forall\ |\beta-k|> 1,\ \beta>0,\ k\in\ZZ_+.
		\end{equation}
		This can be proved by using the following facts.
		\begin{itemize}
	\item If $0<\beta\leq k/5$, then $ k\sim \beta+k\sim k-\beta$, and by Lemma \ref{lem2} (i), 
	$$\left|\p_\beta Q_k^{(1)}(\beta,\phi)\right|\lesssim k^{-2\mu-\alpha}\lesssim (\beta+k)^{-2\mu}k^{-\alpha}\lesssim \left|\frac{\beta-k}{\beta+k}\right|^{\alpha-2}(\beta+k)^{-2\mu}k^{-\alpha}.$$
	\item If $\beta\geq2k$, then $ \beta\sim \beta+k\sim \beta-k$, and by Lemma \ref{lem2} (ii),
		\[\left|\p_\beta Q_k^{(1)}(\beta,\phi)\right|\lesssim \beta^{-2\mu}k^{-\alpha}\lesssim  (\beta+k)^{-2\mu}k^{-\alpha}\lesssim \left|\frac{\beta-k}{\beta+k}\right|^{\alpha-2}(\beta+k)^{-2\mu}k^{-\alpha}.\]
	\item If  $k/10<\beta<3k$, $|\beta-k|> 1$, then $k\sim \beta+k$, and  by Lemma \ref{lem3} (i) (iii),
		\[\left|\p_\beta Q_k^{(1)}(\beta,\phi)\right|\lesssim\int_\TT\left|\p_\beta K_k(\beta,\phi,\Phi)\right|\,d\Phi\lesssim \left|\frac{\beta-k}{k}\right|^{\alpha-2}k^{-2\mu-\alpha}\lesssim\left|\frac{\beta-k}{\beta+k}\right|^{\alpha-2}(\beta+k)^{-2\mu}k^{-\alpha}.\]
\end{itemize} 
If $|k-k_0|{\geq}\max\{2,|\beta_1-\beta_2|k_0/2\}$, then {by \eqref{Ik} and \eqref{p_beta_phi_K_k_estimate}, we have}
		\begin{align*}
			I_k&{\leq} \int_{\beta_2}^{\beta_1}\left|\p_\beta\left(\beta^{2\mu-1
			}Q_{k}^{(1)}\right)(\beta,\phi)\right|\,d\beta\\
			&\lesssim \int_{\beta_2}^{\beta_1}\Big(k_0^{2\mu-2}|k_0-k|^{\alpha-1}(k_0+k)^{1-\alpha-2\mu}k^{-\alpha}
			+k_0^{2\mu-1}|k_0-k|^{\alpha-2}(k_0+k)^{2-\alpha-2\mu}k^{-\alpha}\Big)d\beta\\
			&\lesssim|\beta_1-\beta_2|k_0^{2\mu-1}|k_0-k|^{\alpha-2}(k_0+k)^{2-\alpha-2\mu}k^{-\alpha}.
		\end{align*}
		Here we have used $|\beta\pm k|\sim |k_0\pm k|$ for $\beta\in[\beta_2,\beta_1]$ and
		\begin{align*}
			\frac{k_0^{2\mu-2}|k_0-k|^{\alpha-1}(k_0+k)^{1-\alpha-2\mu}k^{-\alpha}}{
				k_0^{2\mu-1}|k_0-k|^{\alpha-2}(k_0+k)^{2-\alpha-2\mu}k^{-\alpha}}=\frac{|k_0-k|}{k_0(k_0+k)}\leq1.
		\end{align*}
	\end{enumerate}
	
	This completes the proof.
\end{proof}

	\section*{Acknowledgments}
We would like to thank the referees for the invaluable comments and suggestions, which have helped us improve the paper significantly.	
	D. Wei is partially supported by the National Key R\&D Program of China under the
grant 2021YFA1001500.	Z. Zhang is partially supported by  NSF of China  under Grant 12288101.


\begin{thebibliography}{99}
		
		\bibitem{ABC} D. Albritton, E. Bru\'e and M. Colombo, Non-uniqueness of Leray solutions of the forced Navier-Stokes equations,
		\textit{Ann. of Math. (2)},  \textbf{196} (2022),  415–455.
		
		\bibitem{Bar} C. Bardos, E. S. Titi and E. Wiedemann, The vanishing viscosity as a selection principle for the Euler equations: the case of 3D shear flow, \textit{C. R. Math. Acad. Sci. Paris},  \textbf{350} (2012),  757–760.
		
		\bibitem{BM} J. Bedrossian and N. Masmoudi,  Inviscid damping and the asymptotic stability of planar shear flows in the 2D Euler equations,\textit{Publ. Math. Inst. Hautes \'Etudes Sci.}, \textbf{122} (2015), 195–300.
		
		\bibitem{Bou} F. Bouchet and H. Morita, Large time behavior and asymptotic stability of the 2D Euler and linearized Euler equations,  \textit{Phys. D}, \textbf{239} (2010), 948–966.
		
		\bibitem{BL} J. Bourgain and  D. Li,  Strong ill-posedness of the incompressible Euler equation in borderline Sobolev spaces,
		\textit{Invent. Math.},  \textbf{201} (2015),  97–157.
		
		\bibitem{Bre} A. Bressan and R. Murray, On self-similar solutions to the incompressible Euler equations, \textit{J. Differential Equations}, \textbf{269} (2020), 5142-5203.
		
		\bibitem{Chang} K. C. Chang, \textit{Methods in Nonlinear Analysis}, Springer Berlin, Heidelberg, 2005.
		
		\bibitem{Che} J. Y. Chemin,  {\it Perfect incompressible fluids}. Translated from the 1995 French original by Isabelle Gallagher and Dragos Iftimie. Oxford Lecture Series in Mathematics and its Applications, {\bf 14}. The Clarendon Press, Oxford University Press, New York, 1998.
		
		\bibitem{CKO1} T. Ci\'eslak, P. Kokocki and W. S. Oza\'nski, Well-posedness of logarithmic spiral vortex sheets,  arXiv:2110.07543, 2021.
		
		\bibitem{CKO2} T. Ci\'eslak, P. Kokocki and W. S. Oza\'nski, Existence of nonsymmetric logarithmic spiral vortex sheet solutions to the 2D Euler equations, arXiv:2207.06056, 2022.
		
		\bibitem{De}  J. M. Delort,  Existence de nappes de tourbillon en dimension deux, \textit{J. Amer. Math. Soc.}, \textbf{4} (1991),  553–586.
		
		\bibitem{Dong} H. Dong and S. Kim, Partial Schauder estimates for second-order elliptic and parabolic equations: a revisit, \textit{Int. Math. Res. Not.}, (IMRN 2019),  2085-2136.
		
		\bibitem{Elgindi} T. M. Elgindi and I.-J. Jeong, Symmetries and critical phenomena in fluids, \textit{Comm. Pure Appl. Math.}, \textbf{73} (2020), 257-316.
		
		\bibitem{EJ} T. M. Elgindi and I.-J. Jeong, On singular vortex patches, I: Well-posedness issue, \textit{Mem. Amer. Math. Soc.},  \textbf{283} (2023), no. 1400, v+89 pp.
		
		
		\bibitem{ELL2013} V. Elling, Algebraic spiral solutions of 2D incompressible Euler, \textit{J. Differential Equations}, \textbf{255} (2013), 3749-3787.
		
		\bibitem{ELL2016} V. Elling, Self-similar 2D Euler solutions with mixed-sign vorticity, \textit{Commun. Math. Phys.}, \textbf{348} (2016), 27-68.
		
		\bibitem{ELL2016-2} V. Elling, Algebraic spiral solutions of the 2D incompressible Euler equations, \textit{Bull. Braz. Math. Soc. (N.S.)},
		\textbf{47} (2016), 323–334.
		
		
		\bibitem{Fedoryuk} M. V. Fedoryuk, \textit{Asymptotic Analysis: Linear Ordinary Differential Equations}, Springer Berlin, Heidelberg, 1993.
		
		\bibitem{GG} C. Garc\'ia and J. G\'omez-Serrano,  Self-similar spirals for the generalized surface quasi-geostrophic equations, to appear in \textit{J. Eur. Math. Soc.}; arXiv: 2207.12363v1, 2022.
		
		
		\bibitem{IJ} A. Ionescu and H. Jia, Axi-symmetrization near point vortex solutions for the 2D Euler equation,  \textit{Comm. Pure Appl. Math.},
		\textbf{75} (2022), 818–891. 	
		
		\bibitem{JS} H. Jia and V. Sverak, Are the incompressible 3D Navier-Stokes equations locally ill-posed in the natural energy space?,
		\textit{J. Funct. Anal.},  \textbf{268} (2015), 3734–3766.	
		
		\bibitem{JLW} Y. Jin, D. Li and X. Wang, Regularity and analyticity of solutions in a direction for elliptic equations, \textit{Pacific J. Math.}, \textbf{276} (2015),  419-436.
		
		\bibitem{Katznelson} Y. Katznelson, \textit{An Introduction to Harmonic Analysis, 3rd ed.}, Cambridge University Press, Cambridge, 2004.
		
		\bibitem{KS} A. Kiselev and V. Sverak, Small scale creation for solutions of the incompressible two-dimensional Euler equation,
		\textit{Ann. of Math. (2)},  \textbf{180} (2014), 1205–1220.	
		
		\bibitem{MB}   A. J. Majda and A. L.  Bertozzi, {\it Vorticity and incompressible flow}. Cambridge Texts in Applied Mathematics, {\bf 27}. Cambridge University Press, Cambridge, 2002.
		
		\bibitem{MP} C. Marchioro and M. Pulvirenti,  \textit{Mathematical theory of incompressible non-viscous fluids},  Applied Mathematical Sciences, 96. Springer-Verlag, New York, 1994.	
		
		
		\bibitem{NIST} F. Olver, D. Lozier, R. Boisvert and C. Clark, \textit{The NIST Handbook of Mathematical Functions}, Cambridge University Press, New York, NY, 2010.
		
		\bibitem{Pull1} D. I. Pullin, The large-scale structure of unsteady self-similar rolled-up vortex sheets, \textit{J. Fluid Mech.}, \textbf{88} (1978), 401–430.
		
				
	\bibitem{Pull2}	D. I. Pullin,  On similarity flows containing two-branched vortex sheets, Mathematical aspects of vortex dynamics (Leesburg, VA, 1988), 97–106, SIAM, Philadelphia, PA, 1989.		
		
		\bibitem{Sve} V. Sverak,  \textit{Lecture notes}, http://www-users.math.umn.edu/~sverak/course-notes2011.pdf.	
		
		\bibitem{van} M. vanDyke, \textit{An Album of Fluid Motion}, The Parabolic Press, Stanford, 1982.
		
		
		\bibitem{Vis} M. Vishik, Instability and non-uniqueness in the Cauchy problem for the Euler equations of an ideal incompressible fluid. Part I,
		arXiv: 1805.09426, 2018.
		
		\bibitem{WZZ} D. Wei, Z. Zhang and W. Zhao, Linear inviscid damping and vorticity depletion for shear flows,
		\textit{Ann. PDE},  \textbf{5} (2019),  Paper no. 3, 101 pp.
		
		\bibitem{Yu} V. Yudovich, Non-stationary flow of an ideal incompressible liquid, \textit{Comput. Math. Math. Phys.},  \textbf{3} (1963), 1407–1457.
		
	\end{thebibliography}
\end{document}